\documentclass[11pt]{article}
\usepackage{arxiv}
\usepackage{tabu}       

\usepackage{nicefrac}       
\usepackage{microtype}      
\usepackage{lipsum}
\usepackage[utf8]{inputenc} 
\usepackage[T1]{fontenc}    

\usepackage{graphicx}  
\usepackage{amsmath}  
\usepackage{hyperref}
\usepackage{multirow}
\usepackage{url}
\usepackage[all]{xy}
\usepackage{microtype}
\usepackage{graphicx}
\usepackage{subcaption}
\usepackage{ upgreek }
\usepackage{booktabs} 
\usepackage{ dsfont }
\usepackage{amsfonts}       
\usepackage{ amssymb }
\usepackage{amsthm}
\usepackage{algorithm}
\usepackage{algorithmic}
\usepackage{enumerate}
\usepackage{forloop}
\usepackage{comment}
\usepackage[numbers]{natbib}
\usepackage{float}
\usepackage{tabularx}
\usepackage[nottoc]{tocbibind}
\usepackage{ latexsym }
\usepackage{pst-node}
\usepackage{amssymb,amsthm}
\usepackage{tikz-cd}
\usepackage{mathabx}
\usepackage{esvect}
\usepackage{tikz}
\usepackage{paralist}

\usepackage{nomencl}

\newcommand{\overbar}{\overline}

\newcommand{\e}{\varepsilon}
\newcommand{\R}{\mathbb{R}}
\newcommand{\N}{\mathbb{N}}

\newcommand{\Filt}{\mathrm{Filt}}
\newcommand{\vertex}{v}

\newcommand\paramspace{\manifoldm}
\newcommand\param{\theta}
\newcommand\parametrization{F}
\newcommand\simplex{\sigma}
\newcommand\manifoldm{\mathcal{M}}
\newcommand\manifoldn{\mathcal{N}}
\newcommand\manifoldx{\mathcal{X}}
\newcommand\elementm{x}

\newcommand\obarcode{\tilde{D}}
\newcommand\GenPos{\tilde{\mathcal{P}}}
\newcommand\GenPosA{\tilde{\mathcal{A}}}
\newcommand\Sing{\mathrm{Sing}}
\newcommand\Crit{\mathrm{Crit}}
\newcommand\Dgm{\mathrm{Dgm}}
\newcommand\Perm{\mathrm{Perm}}
\newcommand\db{d_\infty}
\DeclareMathOperator*{\argmax}{\mathrm{argmax}}

\makeatletter
\newcommand{\rmathbb}[1]{{\mathpalette\dude@rmathbb{#1}}}
\newcommand{\dude@rmathbb}[2]{%
  \scalebox{-1}[1]{$\m@th#1\mathbb{#2}$}%
  }
\makeatother

\theoremstyle{plain}
\newtheorem{theorem}{Theorem}[section]
\newtheorem{lemma}[theorem]{Lemma}
\newtheorem{proposition}[theorem]{Proposition}
\newtheorem{corollary}[theorem]{Corollary}

\newtheorem*{theorem*}{Theorem}
\newtheorem*{proposition*}{Proposition}

\theoremstyle{definition}
\newtheorem{remark}[theorem]{Remark}

\newtheorem{definition}[theorem]{Definition}
\newtheorem{example}[theorem]{Example}

\begin{document}

	\title{A Framework for Differential Calculus on Persistence Barcodes}
	\author{
  Jacob Leygonie\\
  Mathematical Institute\\
  University of Oxford\\
  Oxford OX2 6GG, UK \\
  \texttt{jacob.leygonie@maths.ox.ac.uk} \\
   \And
 Steve Oudot \\
 INRIA Saclay – Île-de-France\\
  1, rue Honor\'e d'Estienne d'Orves\\
  91120 Palaiseau, France\\
  \texttt{steve.oudot@inria.fr} \\
  \AND
  Ulrike Tillmann \\
   Mathematical Institute \\
University of Oxford\\
  Oxford OX2 6GG, UK \\
   \texttt{tillmann@maths.ox.ac.uk} \\
}

	\maketitle

\begin{abstract}
We define notions of differentiability for
maps from and to the space of persistence barcodes. Inspired by the theory of diffeological spaces, the proposed framework uses lifts to the space of
ordered barcodes, from which derivatives can be computed. The two
derived notions of differentiability (respectively from and to the space
of barcodes) combine together naturally to produce a chain rule that
enables the use of gradient descent for objective functions factoring
through the space of barcodes. We illustrate the versatility of this
framework by showing how it can be used to analyze the smoothness of
various parametrized families of filtrations arising in  topological data analysis.

\textup{2020} \textit{Mathematics Subject Classification}: \textup{55N31, 62R40}
\end{abstract}

\tableofcontents
\newpage

\section{Introduction}
\label{section_introduction}
\subsection{Motivation}

\textit{Barcodes} have been introduced in topological data analysis (TDA) as a means to encode the topological structure of spaces and real-valued functions. 
They have been shown to provide complementary information compared to classical geometric or statistical methods, which explains their interest for applications. However, so far they have been essentially used as an alternative representation of the input, engineered by the user, as opposed to optimized to fit the problem best.

Optimizing barcodes using e.g. gradient descent requires to differentiate objective functions that factor through the space~$Bar$ of barcodes: 
\begin{equation} 
    \label{equation_composition}
    \ \xymatrix{ \manifoldm \ar[rr] && Bar \ar[rr] && \R,} 
    \end{equation}
where $\manifoldm$ is a parameter space equipped with a differential structure, typically a smooth finite-dimensional manifold. 
A compelling example arises in the context of supervised learning, where the barcodes can be used as features for data, generated by using some {\em filter function} $f:K\to\R$ on a fixed graph or simplicial complex~$K$.  Instead of considering~$f$ as a hyperparameter, it can be beneficial to optimize it among a family $\{f_\param{}:K\rightarrow \mathbb{R}\}_{\param{} \in \paramspace{}}$ parametrized by a smooth map which we call the \textit{parametrization}:
\[\parametrization{}:\param{}\in \paramspace{} \longmapsto f_{\param{}}\in \mathbb{R}^K.\]
Post-composing $F$ with the {\em persistent homology} operator~$\Dgm_p$ in homology degree~$p$ 
yields a map $\Dgm_p\circ F: \mathcal{M}\to Bar$. Given a loss function $\mathcal{L}:Bar\rightarrow \R$, the goal is then to minimize the functional

\begin{equation} 
    \label{equation_composition_machine_learning}
    \ \xymatrix{  \manifoldm \ar[rr]^{\Dgm_p\circ F} &&  Bar \ar[rr]^{\mathcal{L}} && \R}
    \end{equation}
using variational approaches, which are standard in large-scale learning applications. In order to do so, we need to put a sensible smooth structure on $Bar$ and to derive an analogue of the chain rule, so that we can compute the differential of $\mathcal{L}\circ \Dgm_p \circ F$ as the composition of the differentials of $\mathcal{L}$ and $\Dgm_p \circ F$. 
The difficulty arises as $Bar$ is not a manifold and so far has not been given a structure in which the above makes sense.

Beyond optimization, we want to be able to address other types of applications where differential calculus is involved.
For this,  a variety of potential scenarios must be considered, e.g. when the filter functions are defined on a  fixed smooth manifold, or when the second arrow in~\eqref{equation_composition} takes its values in~$\R^n$ or more generally in some smooth finite-dimensional manifold. The goal of our study is to provide a unified framework that accounts for all these scenarios.
\subsection{Related work}

Despite the lack of a smooth structure on the space~$Bar$, developing heuristic methods to differentiate the composition in Equation \eqref{equation_composition_machine_learning} has been an active direction of research lately, leading to innovative computational applications. In Table~\ref{tab:my-table}, we specify, for each of these contributions, the choice of parametrization $\parametrization{}$ and of loss function  $\mathcal{L}$, the optimization problem under consideration, and the sufficient conditions worked out to guarantee the differentiability of the composition in~\eqref{equation_composition_machine_learning}.

In the context of point cloud inference considered by~\cite{gameiro2016continuation}, the positions of points in a fixed Euclidean space form the parameter space~$\manifoldm$, and the resulting Rips filtration (resp. Alpha filtration) of the total complex on the point cloud is the parametrization $\parametrization{}$. The loss function~$\mathcal{L}$ is given by the least-squares approximation of a fixed barcode. By developing a clear functional point of view on the connection between the barcode of the Rips or Alpha filtration and the positions of the points in the cloud, based on lifts to Euclidean space, the authors show that~$\mathcal{L}$ is 
differentiable wherever the pairwise distances between points in the cloud are distinct. 
The approach is further refined by~\cite{chazal2018density}, where it is observed that the parametrization~$\parametrization$ is a subanalytic map, which implies that the barcode-valued map admits subanalytic (hence generically differentiable) lifts. In turn, this fact is leveraged to show that any probability measure with a density w.r.t. the Hausdorff measure on~$\manifoldm$ induces an expected persistence diagram (viewed as a measure in the plane) with a density w.r.t. the Lebesgue measure.

In many applications,~$\parametrization$ parametrizes \textit{lower-star filtrations}, i.e. filter functions induced by their restrictions to the vertices of~$K$ \citep{gabrielsson2018topology,chen2019topological,1905.12200,2019arXiv190510996H,hu2019topology,poulenard2018topological}. In~\cite{poulenard2018topological}, the problem of shape matching is cast into an optimization problem involving the barcodes of the shapes. \cite{chen2019topological} uses the degree-$0$ persistent homology as a regularizer for classifiers. Similarly, \cite{hu2019topology} proposes a persistence based regularization as an additional loss for deep learning models in the context of image segmentation. In~\cite{2019arXiv190510996H}, a dataset of graphs is seen as part of a bigger common simplicial complex, which allows to learn a filter function which is shared across the whole dataset. These contributions require the differentiability of~\eqref{equation_composition_machine_learning}, and they show that it holds whenever the filter function $f_\param$ is injective over the vertex set.

 Functions on a grid are used in~\cite{gabrielsson2018topology} to tackle the problem of surface reconstruction. These functions are sums of gaussians whose means and variances are parameters one wants to optimize according to an objective/loss that depends on the degree-$1$ persistent homology of the functions. \cite{1905.12200} considers optimization problems involving persistence with many useful applications as in generative modelling, classification robustness, and adversarial attacks. Both contributions need to take the derivative of~\eqref{equation_composition_machine_learning}, and to do so, they require the existence of an inverse  map taking interval endpoints in the persistence diagram $\Dgm_p(f_\param)$ to the corresponding vertices of $K$. This is a strictly weaker requirement than the injectivity of $f_\param$, as used in the previous contributions, because an inverse map always exists (provided for instance by the standard reduction algorithm for persistent homology). However, per se, it does not guarantee the differentiability of the composition---see e.g.~\cite{2019arXiv190510996H} for a counter-example.

\begin{table}[]
\centering
\tabulinesep=2pt
\begin{tabu}{|l|l|l|l|l|l}
\hline

\begin{tabular}[c]{@{}l@{}}\textbf{Targeted} \\ \textbf{Application}   \end{tabular}

         & 
         \begin{tabular}[c]{@{}l@{}}\textbf{Loss} \\ \textbf{function $\mathcal{L}$}   \end{tabular}
         
                                                                    & \textbf{Parametrization $\parametrization$}                                                                                 &  

\begin{tabular}[c]{@{}l@{}}\textbf{Conditions for} \\ \textbf{Differentiability}\\ \textbf{of ~\eqref{equation_composition_machine_learning}}    \end{tabular} 
                                                                                                                           \\ \hline

\begin{tabular}[c]{@{}l@{}}Topological regulariser \\ for classification \citep{chen2019topological}\end{tabular}                  &    

\begin{tabular}[c]{@{}l@{}}Penalty on the \\ degree-$0$ \\ persistent \\ homology\end{tabular}    & \multirow{6}{*}[4em]{\begin{tabular}{c} Parametrizations \\ of lower star \\ filtrations\end{tabular}}                   & \multirow{4}{*}[1.5em]{\begin{tabular}[c]{@{}l@{}}Injectivity of the filter \\ function over the vertices\end{tabular}}                                                    \\ \cline{1-2}

\begin{tabular}[c]{@{}l@{}} Topological regulariser \\ for image segmentation \citep{hu2019topology}     \end{tabular}                                                                                    &       \begin{tabular}[c]{@{}l@{}}Penalty on the \\ degrees $0$ and $1$ \\ persistent \\ homology\end{tabular}        &                                                                                                                             &                                                                                                                                                                       \\ \cline{1-2}

\begin{tabular}[c]{@{}l@{}} Graph classification \citep{2019arXiv190510996H}   \end{tabular} 
                                                                                        & \begin{tabular}[c]{@{}l@{}}Any loss of \\ supervised \\ learning
\end{tabular}                 &                                                                                                                             &                                                                                                                                                                       \\ \cline{1-2}

\begin{tabular}[c]{@{}l@{}} Shape matching \citep{poulenard2018topological}  \end{tabular}                                                                                      & \begin{tabular}[c]{@{}l@{}}Bottleneck or \\ Wasserstein \\ distance\end{tabular}                 &                                                                                                                             &                                                                                                                                                                       \\ \cline{1-2} \cline{4-4} 
\begin{tabular}[c]{@{}l@{}}Surface reconstruction \citep{gabrielsson2018topology}\end{tabular}                                      & \begin{tabular}[c]{@{}l@{}}Sum of \\ persistences\end{tabular}                                                                          &                                                                                                                             & \multirow{2}{*}[-1em]{\begin{tabular}[c]{@{}l@{}}Local correspondence between\\  barcodes interval endpoints \\ and vertices in the simplicial \\ complex\end{tabular}} \\ \cline{1-2}

\begin{tabular}[c]{@{}l@{}}Generative modelling,\\ Adversarial attacks, and\\ Regularization \citep{1905.12200}\end{tabular} & \begin{tabular}[c]{@{}l@{}}Weighted sum\\ of persistences\end{tabular}                       &                                                                                                                             &                                                                                                                                                                       \\ \hline
\begin{tabular}[c]{@{}l@{}}Point cloud continuation \citep{gameiro2016continuation}\end{tabular}                                    & \begin{tabular}[c]{@{}l@{}} Least-squares \\ approximation  \\ of a fixed\\ barcode
\end{tabular}
 & \begin{tabular}[c]{@{}l@{}}Point clouds \\ determining\\ a Rips filtration\end{tabular} & \begin{tabular}[c]{@{}l@{}}Distinct pairwise distances\\ between points\end{tabular}                                                                              \\ \hline
\end{tabu}\\[0.5em]
\caption{Current frameworks for differentiating the composition in~\eqref{equation_composition_machine_learning}. The first column lists the targeted applications. The second and third columns show the choices of loss function~$\mathcal{L}$  and parametrization $F$. The differentiability of $\mathcal{L}\circ \Dgm_p\circ \parametrization$ is guaranteed under the conditions listed in the fourth column.}
\label{tab:my-table}
\end{table}

This variety of applications motivates the search for  a unified framework for expressing the differentiability of the arrows in diagrams of the form:
\begin{equation}
\label{eq_loss_persistence}
 \xymatrix{\paramspace{} \ar[rr] && Bar \ar[rr] &&  \R. } 
\end{equation}

Since the first appearance of this paper as a preprint, there have been novel applications of persistence differentiability in optimisation. For instance, the first author has developped a graph classification framework based on the Laplacian operator~\cite{yim2021optimisation}, applying the differentiability of the persistence map (Theorem~\ref{theorem_PH_differentiable_global}) to the case of extended persistence. In addition, new heuristics to smooth and regularise loss functions as in Eq.~\eqref{eq_loss_persistence} improved the optimisation procedure for specific data science problems~\citep{corcoran2020regularization,solomon2020fast}. Another strong guarantee is provided when the loss in~Eq.~\eqref{eq_loss_persistence} is semi-algebraic (and more generally subanalytic or definable in some o-minimal structure), as then the classic stochastic gradient descent (SGD) algorithm converges to critical points~\cite{davis2020stochastic}. The bridge between this result in non-smooth analysis and persistence based optimisation problems is made in~\cite{carriere2020note}, where sufficient conditions for loss functions as in Eq.~\eqref{eq_loss_persistence} to be semi-algebraic are given. The main results of~\cite{carriere2020note} also derive from our general framework, see Remark~\ref{remark_definable_lift_for_definable_parametrization}. 
\subsection{Contributions and outline of the paper}

Ultimately, our framework should make it possible to determine when and how maps between smooth manifolds~$\manifoldm{}$ and~$\manifoldn{}$ that factor through the space of barcodes can be differentiated:
\[ \xymatrix{\manifoldm{} \ar[rr]^-{B} && Bar \ar[rr]^-{V} && \manifoldn{}}. \]
To achieve this goal, in Section~\ref{section_definition_differentiability_barcode_valued_maps} we define differentiability via lifts in full generality, thereby extending the approach initially proposed by \cite{gameiro2016continuation} for the specific case of parametrizations by Rips filtrations. Here we provide some of the details.
As a space of multi-sets (assumed by default to have finitely many off-diagonal points), $Bar$ does not naturally come equipped with a differential structure. However, it is covered by maps of the form:
\begin{center}
\label{lift_bar}
\begin{tikzpicture}
\node[label=above:{$\mathbb{R}^{2m}\times \mathbb{R}^n$}] (a) at (0,0) {};
\node[label=below:{$Bar$}] (b) at (0,-1) {};
\node[label=right:{$Q_{m,n}$}] (c) at (0,-0.5) {};
\draw[->] (a)--(b);
\end{tikzpicture}
\end{center}
%
where $\mathbb{R}^{2m}\times \mathbb{R}^n$ can be thought of as the space of ordered barcodes with fixed number $m$ (resp. $n$) of finite (resp. infinite) intervals, and where $Q_{m,n}$ is the quotient map modulo the order---turning vectors into multisets (Definition~\ref{definition_ordered_barcodes_quotient_map_Q}). Then, the map $B:\manifoldm{}\rightarrow Bar$ is said to be {\em $r$-differentiable} at parameter~$\param \in \manifoldm{}$ if it admits a local $C^r$ lift $\tilde{B}$ into~$\R^{2m}\times\R^n$ for  some $m,n\in\mathbb{N}$:
\begin{equation}
\label{lift_B}
\begin{gathered}
\begin{tikzpicture}
\node[label=above:{$\mathbb{R}^{2m}\times \mathbb{R}^n$}] (a) at (0,0) {};
\node[] (b) at (0,-1) {$Bar$};
\node[label=right:{$Q_{m,n}$}] (c) at (0,-0.5) {};
\draw[->] (a)--(b);
\node[] (d) at (-3,-1) {$\param\in U\subset \manifoldm{}$};
\draw[->] (d)--(b);
\draw[->,dotted] (d)--(a);
\node[] at (-2,-0.3) {$\exists \tilde{B}$};
\node[] at (-1.2,-1.3) {$B$};
\end{tikzpicture}
\end{gathered}
\end{equation}
This means that the map~$\tilde B$ tracks smoothly and consistently the points in the barcodes $B(\param')$, for $\param'$ ranging over some open neighborhood $U$ of $\param$. 
Dually, the map~$V:Bar \to \manifoldn$ is $r$-differentiable at~$D\in Bar$ if for every possible choice of~$m,n$, the composition~$V \circ Q_{m,n}:\mathbb{R}^{2m}\times \mathbb{R}^n \rightarrow Bar$ is~$C^r$ on an open neighborhood of every pre-image~$\obarcode{}$ of~$D$:
\begin{equation}
\label{lift_V}
\begin{gathered}
\begin{tikzpicture}
\node[label=above:{$\mathbb{R}^{2m}\times \mathbb{R}^n$}] (a) at (0,0) {};
\node[] (b) at (0,-1) {$Bar$};
\node[label=right:{$Q_{m,n}$}] (c) at (0,-0.5) {};
\draw[->] (a)--(b);
\node[] (e) at (3,-1) {$\manifoldn{}$};
\draw[->] (b)--(e);
\node[] at (1.5,-1.3) {$V$};
\end{tikzpicture}
\end{gathered}
\end{equation}
The choice of $m,n$ and pre-image $\obarcode{}$ of $D$ should be thought of as the type of perturbation we allow around $D$. Thus, essentially, $V$ is asked to be smooth with respect to any finite perturbation of~$D$. 
In section \ref{subsection_diffeology} we connect these definitions to the theory of diffeological spaces, showing that our two definitions of differentiability for maps $B$ and $V$ are dual to each other and make the barcode space $Bar$  a \textit{diffeological space}.

We then define the differentials of the maps $B$ and $V$, given simply by the differentials of the lift~$\tilde B: \manifoldm\to\R^{2m}\times\R^n$ (for~$B$) and of the composition $V\circ Q_{m,n}$ on the pre-image $\obarcode\in\R^{2m}\times\R^n$ (for~$V$). Although these differentials taken individually are not defined uniquely, their corresponding diagrams~\eqref{lift_B} and~\eqref{lift_V} combine together as follows:
\begin{center}
\label{composition_V_B}
\begin{tikzpicture}
\node[label=above:{$\mathbb{R}^{2m}\times \mathbb{R}^n$}] (a) at (0,0) {};
\node[] (b) at (0,-1) {$Bar$};
\node[label=right:{$Q_{m,n}$}] (c) at (0,-0.5) {};
\draw[->] (a)--(b);
\node[] (e) at (3,-1) {$\manifoldn{}$};
\draw[->] (b)--(e);
\node[] at (-1.2,-1.3) {$B$};
\node[] at (1.5,-1.3) {$V$};
\node[] (d) at (-3,-1) {$\param{}\in U\subset \manifoldm{}$};
\draw[->] (d)--(b);
\draw[->,dotted] (d)--(a);
\node[] at (-2,-0.3) {$\exists \tilde{B}$};
\end{tikzpicture}
\end{center}
implying that the composition $V\circ B = (V\circ Q_{m,n}) \circ \tilde B$ is a $C^r$ map between smooth manifolds, whose derivative is obtained by composing the differentials of  $B$ and $V$, and this regardless of the choice of lift and pre-image. This is our analogue of the chain rule in ordinary differential calculus (Proposition~\ref{proposition_chain_rule_vectorization_barcode}).

In Sections~\ref{section_discrete_smoothness} and~\ref{section_continuous_smoothness_theorem}, we focus on barcode-valued maps $B:\manifoldm{}\rightarrow Bar$ arising from filter functions on fixed smooth manifolds or simplicial complexes. These maps are usually not differentiable everywhere on their domain. However, motivated by the aforementioned applications, we seek conditions under which $B$ is differentiable almost everywhere on~$\manifoldm$. A natural approach for this would be to use Rademacher's theorem \citep[Thm.~3.1.6]{federer2014geometric}, as we know that $B$ is Lispchitz continuous by the Stability Theorem of persistent homology \citep{bl-indaspb-15,chazal2016structure,cohen2007stability}. However, this approach has several important shortcomings:
\begin{itemize}
    \item it depends on a choice of measure on~$\manifoldm$;
    \item it calls for a generalization of Rademacher's theorem to maps taking values in arbitrary metric spaces, and to the best of our knowledge, existing generalizations only provide directional metric differentials (see e.g. \cite{pansu1989metriques});
    \item more fundamentally, it is not constructive and therefore does not provide formulae for the differentials;
    \item finally, in the context of optimization, it is important to guarantee the existence of differentials/gradients in an open neighborhood of the considered parameter~$\param$, and not just in a full-measure subset.
\end{itemize}
We therefore propose to follow a different approach, seeking conditions that ensure the differentiability of~$B$ on a generic (i.e. open and dense) subset of~$\manifoldm$, with explicit differential.

Our first scenario (Section~\ref{section_discrete_smoothness}) considers a parametrization
$\parametrization{}: \paramspace{} \longrightarrow \mathbb{R}^K$ of filter functions on a fixed simplicial complex~$K$.
Given a homology degree $p \leqslant d$, where $d$ is the maximal simplex dimension in~$K$, the barcode-valued map~$B$ decomposes as $B=\Dgm_p\circ \parametrization$, and in Theorem~\ref{theorem_PH_differentiable_global} we show that $B$ is $r$-differentiable on a generic subset of $\paramspace{}$ whenever $\parametrization$ is $C^r$ over $\manifoldm$ or a generic subset thereof. The proof relies on the fact that the pre-order on the simplices of $K$ induced by the values assigned by the filter function $F(\param)$  is generically constant around~$\param$ in $\paramspace$.
%
We then relate the diffential of~$B$ to those of~$F$ in Proposition~\ref{proposition_derivatives_barcode_simplicial_complex}, yielding a closed formula that can be leveraged in practical implementations.
Finally, we study the behavior of~$B$ at singular points by means of a stratification of the parameter space~$\manifoldm{}$, whereby the top-dimensional strata are the locations where $B$ is differentiable, and the lower-dimensional strata characterize the defect of differentiability of~$B$. We show in Theorem~\ref{theorem_PH_partial_derivatives_stratification} that we can define \textit{directional derivatives} along each incident stratum at any given point~$\param\in\paramspace$. We also show that the barcode valued map can be globally lifted and expressed as a permutation map on each stratum (Corollary~\ref{corollary_global_lift_parametrized}). 

In Section~\ref{section_discrete_smoothness_subsection_examples}  we illustrate the impact of our framework on a series of  examples of parametrizations coming from earlier work, including lower-star filtrations, Rips filtrations and some of their generalizations. For each example, we examine the differentiability of the barcode-valued map and, whenever readily computable, we give the expressions of its differential. This allows us to recover the differentiability results from earlier work in a principled way.

Our second scenario (Section~\ref{section_continuous_smoothness_theorem}) considers a parametrization $\parametrization{}: \paramspace{} \longrightarrow C^\infty(\manifoldx,\mathbb{R})$ of smooth filter functions on a fixed smooth compact $d$-dimensional manifold~$\manifoldx$.
In this scenario, given a parameter~$\param{}\in\paramspace$, the barcode-valued map~$B$ computes all the barcodes of $f_\param$ at once, and collates them in a vector of barcodes:
\[ B: \param{}\in \paramspace{} \longmapsto (\Dgm_0(f_{\param{}}),...,\Dgm_d(f_{\param{}}))\in Bar^{d+1}. \]
We show  that $B$ is $\infty$-differentiable at any parameter $\param{}$ such that $f_\param$ is Morse with distinct critical values (Theorem \ref{smoothness_barcode_manifold}). The key insights are: on the one hand, that at any such parameter~$\param$ the implicit function theorem allows us to smoothly track the critical points of $f_{\param{}'}$ as $\param'$ ranges over a small enough open neighborhood around~$\param{}$; on the other hand, that the Stability Theorem provides a consistent correspondence between the critical points of $f_{\param'}$ and the interval endpoints in its barcodes.

In Section~\ref{section_differentiability_vectorization} we look at
examples of classes of maps $V:Bar \rightarrow \manifoldn{}$. We first
consider persistence images \citep{adams2017persistence} and more generally linear representations of barcodes, as an
illustration of our framework on barcode vectorizations. We show that persistence images and linear representations are $\infty$-differentiable under suitable choices of weighting function
(Propositions~\ref{proposition_persistence_image_smooth} and~\ref{proposition_linear_representation_differentiable}). We then consider the case where~$V:Bar \rightarrow \R$ is the bottleneck or Wasserstein distance to a fixed barcode, and show it is semi-algebraic in a suitable sense (Proposition~\ref{prop_distance_diagram_semi-_algebraic}), which is useful in a context of optimisation. We then focus on the bottleneck distance to a fixed barcode~$D_0$, which we
believe can be of interest in the context of inverse problems. We show that this distance is differentiable on a generic subset of~$Bar$ (Propositions~\ref{prop:dist_empty_diag_diff} and~\ref{prop:dist__diag_diff}).

Finally, throughout the paper we sprinkle our exposition with examples of parametrizations and loss functions that illustrate our results and demonstrate their potential for applications.

\subsection*{Acknowledgements}
The authors wish to thank Vidit Nanda and Oliver Vipond for the frequent conversations that influenced this project. The authors are also indebted to the anonymous reviewers for their valuable insights in the final revisions of the manuscript. JL wishes to thank Heather Harrington for general guidance, Yixuan Wang for sharing knowledge on Morse theory and differential geometry, and finally Ambrose Yim and Naya Yerolemou for their feedback. This research has been supported in part by ESPRC grant EP/R018472/1.

\section{Preliminary notions}
\label{section_general_persistence}

Throughout the paper, vector spaces and homology groups are taken over a fixed field $\mathds{k}$, omitted in our notations whenever clear from the context.
As much as possible, we keep separate terminologies for different notions of differentiability, for instance: maps from or to the space of barcodes are called {\em $r$-differentiable} when maps between manifolds are simply called $C^r$. The only exception to this rule is the term {\em smooth} for maps, which has a versatile meaning that should nonetheless always be clear from the context.

\subsection{Persistence modules and persistent homology}
\label{sec:pers_modules}

\begin{definition}
\label{definition_persistence_module_real_line}
A \textit{persistence module} $\mathbb{V}$ is a functor from the poset $(\mathbb{R},\leqslant)$ to the category $\mathbf{Vect}_\mathds{k}$ of vector spaces over~$\mathds{k}$. 
\end{definition}
In other words, a persistence module is a collection  $\mathbb{V}=\{V_t, v_{s,t}:V_s \rightarrow V_t\}_{(s,t)\in \mathbb{R}^2, s\leqslant t }$ of  vector spaces $V_t$ and linear maps $v_{s,t}$, such that $v_{t,t}=\mathrm{id}_{V_t}$ for all $t\in\R$ and $v_{s,t}\circ v_{r,s}= v_{r,t}$ for all  $r \leqslant s \leqslant t \in \R$. We say that $\mathbb{V}$ is {\em pointwise finite-dimensional} (or {\em pfd} for short) if every $V_t$ is finite-dimensional. Unless otherwise stated, persistence modules in the following  will be pfd.
\begin{definition}
\label{definition_morphism_persistence_module}
A \textit{morphism} $\eta: \mathbb{V} \rightarrow \mathds{W}$ between two persistence modules is a natural transformation between functors.
\end{definition} 
In other words, writing $\mathbb{V}=\{V_t, v_{s,t}\}_{s\leqslant t}$ and $\mathds{W}=\{W_t, w_{s,t}\}_{s\leqslant t}$, a morphism $\eta: \mathbb{V} \rightarrow \mathds{W}$ is a collection of linear maps $\{\eta_t: V_t\rightarrow W_t\}_{t\in \mathbb{R}}$ such that the following diagram commutes for all~$s\leqslant t$:
\[
\xymatrix{V_s \ar^-{v_{s,t}}[r] \ar_-{\eta_s}[d] & V_t \ar^-{\eta_t}[d] \\
  W_s \ar^-{w_{s,t}}[r] & W_t}
\]
We say that $\eta$ is an \textit{isomorphism} of persistence modules if all the $\eta_t$ are isomorphisms of vector spaces.
We denote by $\mathbf{Pers}$ the category of persistence modules. $\mathbf{Pers}$ is an abelian category, so it admits kernels, cokernels, images and direct sums, which are defined pointwise. By Crawley-Bovey's Theorem \citep{crawley2015decomposition}, we know that persistence modules essentially uniquely decompose as direct sums of elementary modules called {\em interval modules}. The interval module~$\mathbb{I}_J$ associated to an interval~$J$ of~$\R$ is defined as the module with copies of the field~$\mathds{k}$ over~$J$ and zero spaces elsewhere, the copies of~$\mathds{k}$ being connected by identity maps.
\begin{theorem}
\label{theorem_crawley_bovey}
For any persistence module $\mathbb{V}$, there is a unique multi-set $\mathcal{J}$ of intervals of~$\R$ such that 
\begin{equation}
    \label{equation_decomposition_into_interval_modules}
    \mathbb{V}\cong \oplus_{J \in \mathcal{J}} \mathbb{I}_J,
\end{equation}
%
\end{theorem}
Persistence modules of particular interest are the ones induced by the sub-level sets of real-valued functions.
\begin{definition}
\normalfont
\label{definition_persistence_homology_module}
Let $f:\manifoldx\rightarrow \mathbb{R}$ be a real-valued function on a topological space. Write $\manifoldx^t:=f^{-1}((-\infty,t])$ for the closed sublevel set of $f$ at level~$t\in \mathbb{R}$. Given  $p\in \mathbb{N}$, the \textit{sublevel set persistent homology} of $f$ in degree $p$ is the (non-necessarily pfd) persistence module $\mathbf{H}_p(f)$ defined by:
  \begin{itemize}
    \item the vector spaces $\{H_p(\manifoldx^t)\}_{t\in \mathbb{R}}$, where $H_p$ is the singular homology functor in degree $p$ with coefficients in $\mathds{k}$;
    \item the linear maps $\{v_{s,t}:H_p(\manifoldx^s)\rightarrow H_p(\manifoldx^t)\}_{s\leqslant t}$ induced by inclusions $\manifoldx^s \hookrightarrow \manifoldx^t$.
\end{itemize}
\end{definition}
In the following we restrict our focus to {\em finite-type} persistence modules induced by {\em tame} functions, defined as follows:
\begin{definition}
\label{definition_finite_type}
A persistence module $\mathbb{V}$ is of \textit{finite type} if it admits a decomposition into finitely many interval modules.
\end{definition}
\begin{definition}
\normalfont
A function $f:\manifoldx\rightarrow \mathbb{R}$ is \textit{tame} if its persistent homology modules in any degree are of finite type.
\label{definition_tame_function}
\end{definition}
In particular, {\em filter functions} on a finite simplicial complex (see below) and {\em Morse functions} on a smooth manifold (see Section~\ref{sec:Morse}) are tame.
%
\begin{definition}
\label{definition_filter_function}
\normalfont
Let $K$ be a finite simplicial complex. A  \textit{filter function} $f:K \rightarrow \mathbb{R}$ is a function that is monotonous with respect to inclusions of faces in~$K$, i.e.
$f(\simplex{})\leqslant f(\simplex{}')$ for all $\simplex \subseteq \simplex'\in K$. This implies in particular that every sublevel set  $K^t:=\{\simplex{}\in K \mid f(\simplex{})\leqslant t\}$ is a sub-complex of~$K$. 
\end{definition}

\subsection{Persistence barcodes / diagrams}

Given a decomposition of a finite-type persistence module~$\mathbb{V}$ as in~\eqref{equation_decomposition_into_interval_modules}, the (finite) multi-set $\mathcal{J}$ is called the {\em barcode} of~$\mathbb{V}$. An alternate representation is as a (finite) multiset~$B$ of points in the plane, where each interval $J\in \mathcal{J}$ is mapped to the point $(\inf J, \sup J)$. To this multiset of points we add $\Delta^\infty$, that is the multiset containing countably many copies of the diagonal~$\Delta:=\{(b,b) \mid b \in \mathbb{R}\}$, to obtain the so-called {\em persistence diagram} of~$\mathbb{V}$. When $\mathbb{V}$ is the sublevel set persistent homology of a tame function~$f$ in degree~$p$, we denote by~$\Dgm_p(f)$ its persistence diagram. Persistence diagrams can also be defined independently of persistence modules as follows:
\begin{definition}
\label{barcode_definition}
A {\em persistence diagram} is the union $B\cup \Delta^\infty$ of a finite multiset~$B$ of elements in $\mathbb{R}\times \bar{\mathbb{R}}$, where $\bar\R := \mathbb{R}\cup \{+\infty\}$, with countably many copies of the diagonal~$\Delta$. The set of persistence diagrams is denoted by $Bar$.
\end{definition}
From now on we also use the terminology {\em barcodes} for persistence diagrams. Following this terminology, we also call {\em intervals} the points in a persistence diagram. Points lying on the diagonal~$\Delta$ are qualified as {\em diagonal},  the others are qualified as {\em off-diagonal}.

\begin{remark}
  In the above definitions we follow the literature on extended persistence, in which persistence diagrams can have points everywhere in the extended plane $\R\times\bar \R$. This is because our framework extends naturally to that setting.
  Note also that, in the literature,  the  diagonal is sometimes not included  in the diagrams. Here we are including it with infinite multiplicity. 
  This is in the spirit of taking  the quotient category of observable persistence modules, as defined by \cite{chazal2014observable}.
\end{remark}
\begin{definition}
\label{definition_matching}
\normalfont
Given two barcodes $D,D'\in Bar$, viewed as multisets, a \textit{matching} is a bijection $\gamma:D \rightarrow D'$. The \textit{cost} of $\gamma$ is the quantity 
\[c(\gamma):=\sup_{x \in D} \|x-\gamma(x)\|_\infty \in \bar\R. \]
\end{definition}
We denote by $\Gamma(D,D')$ the set of all matchings between $D$ and $D'$. 
\begin{definition}
\normalfont
The \textit{bottleneck distance} between two barcodes $D,D'\in Bar$ is  
\[\db(D,D'):= \inf_{\gamma \in \Gamma(D,D')} c(\gamma) \]
\label{definition_bottleneck}
\end{definition}
Given~$q\in \R^{*}_+$, a slight modification of the matching cost yields the~$q$-th {\em Wasserstein distance} on barcodes as introduced in~\cite{cohen2010lipschitz}:
\begin{equation}
\label{eq_wasserstein}
d_q(D,D'):= \inf_{\gamma \in \Gamma(D,D')} (\sum_{x\in D}\|x-\gamma(x)\|_{\infty}^q)^{\frac{1}{q}}
\end{equation}
Since we include all points in the diagonal with infinite multiplicity in our definition of barcodes,~$\db$ is a true metric\footnote{In fact an extended metric as it can take infinite values.} and not just a pseudo-metric. Indeed,  for any $D,D'\in Bar$, we have $\db(D,D')=0\Rightarrow D=D'$. We call \textit{bottleneck topology} the topology induced by $\db$, which by the previous observation makes~$Bar$ a Hausdorff space. 

A key fact is the Lipschitz continuity of the barcode function, known as the {\em Stability Theorem} \citep{bl-indaspb-15,chazal2016structure,cohen2007stability}:
\begin{theorem}
\label{stability theorem}
Let $f,g: \manifoldx\rightarrow \mathbb{R}$ be two real-valued functions with well-defined barcodes. Then,
\[ \db(\Dgm_p(f),\Dgm_p(g))\leqslant \|f-g\|_{\infty}. \]
\end{theorem}
Note that the assumptions in the theorem are quite general and hold in our cases of interest: tame functions on a compact manifold, and filter functions on a simplicial complex.

\subsection{Morse functions}
\label{sec:Morse}

Morse functions are a special type of tame functions, for which there is a bijective correspondence between critical points in the domain and interval endpoints in the barcode. This correspondence, detailed in Proposition~\ref{endpoints_in_barcode_are_critical_values}, will be instrumental in the analysis of Section~\ref{section_continuous_smoothness_theorem}. For a proper introduction to Morse theory, we refer the reader to \cite{milnor2016morse}. %

\begin{definition}\normalfont
\label{definition_morse_function}
Given a smooth $d$-dimensional manifold~$\manifoldx$, a smooth function~$f:\manifoldx\rightarrow \mathbb{R}$ is called \textit{Morse} if its Hessian at critical points (i.e. points where the
gradient of~$f$ vanishes) is nondegenerate. 
\end{definition}
Note that we do not assume a priori
that the values of~$f$ at critical points (called {\em critical
  values}) are all distinct. For such a value~$a$, we call {\em
  multiplicity} of~$a$ the number of critical points in the level-set
$f^{-1}(a)$. We also introduce the notation~$\Crit(f)$ to refer to the
set of critical points, which is discrete in~$\manifoldx$. In particular, if~$\manifoldx$ is compact, which will be the case in this paper,~$\Crit(f)$ is finite. The number of negative eigenvalues of~$f$ at a critical point~$x$ is called the {\em index} of~$x$.


%
\begin{proposition}
\label{endpoints_in_barcode_are_critical_values}
Assume~$\manifoldx$ is compact and all the critical values of~$f$ have multiplicity~$1$. Denote by $E(f)$ the multiset of finite endpoints of off-diagonal intervals (including the left endpoints of infinite intervals) of $\Dgm_0(f)\sqcup...\sqcup \Dgm_d(f)$. Then, $f$ induces a bijection $\Crit(f) \to  E(f)$. 
\end{proposition}
This result is folklore, 
and we give a proof only for completeness.
%
\begin{proof}
\normalfont
Let $a\leqslant b$ be real numbers. Write $\manifoldx^a$ for the sublevel set $f^{-1}((-\infty,a])$. 
If $[a,b]$ contains a unique critical value~$c$ of~$f$, then~$\manifoldx^b$ has the homotopy type of $\manifoldx^a$ glued together with a cell $e_p$ of dimension $p$, where $p$ is the index of the unique critical point~$x$ associated to~$c$ \citep{milnor2016morse}. 
%
Therefore, $H_*(\manifoldx^b,\manifoldx^a)$ is trivial except for~$*=p$ where it is spanned by the homology class of $e_p$. This does not depend on the choice of~$a,b$ surrounding~$c$ and sufficiently close to it. Then, using the long exact sequence in homology, we deduce that either there is
 one birth in degree~$p$ at value~$c$ in the persistent homology module, or there is one death  in degree $p-1$. Hence,~$c$ is either a left endpoint of an interval of~$\Dgm_p(f)$, or a right endpoint of an interval of~$\Dgm_{p-1}(f)$. In either case, we can define the map $x\mapsto f(x)$ for any~$x\in \Crit(f)$, and we have just shown that its codomain is indeed $E(f)$. The map is injective because the critical values of~$f$ have multiplicity~$1$ by assumption. We now show it is onto. Let $a\in \R$ be a non-critical value of~$f$. For any (small enough)~$\e,\eta>0$, the interval~$[a-\eta, a+\e]$ contains no critical value of~$f$, therefore $\manifoldx^{a+\e}$ deform retracts onto $\manifoldx^{a-\eta}$, thus implying that the inclusions $ H_p(\manifoldx^{a-\eta})\rightarrow H_p(\manifoldx^{a+\e})$ are identity maps for any homology degree~$p$. By the decomposition Theorem~\ref{theorem_crawley_bovey}, this implies that $a$ cannot be an endpoint of an interval summand, i.e.~$a\notin E(f)$.
\end{proof} 
The assumption that each critical value of~$f$ has multiplicity 1 is superfluous in Proposition \ref{endpoints_in_barcode_are_critical_values}, if we allow the correspondence map to match trivial intervals. Let~$[a,b]$ be an interval containing a unique critical value~$c$. One can still use Morse theory and glue as many critical cells~$e_p$ to~$\manifoldx^a$ as there are critical points in $f^{-1}(c)$ in order to obtain a CW structure on~$\manifoldx^b$ from the one of~$\manifoldx^a$. Considering the different critical cells, we know exactly the ranks of the morphisms $H_p(\manifoldx^{a})\rightarrow H_p(\manifoldx^{b})$ induced by inclusions in each homology degree $p$. 

\subsection{Diffeology theory}
\label{section_diffeology_theory_preliminary_notions}

\textit{Diffeology} theory provides a principled approach to equip a set with a smooth structure. We use some concepts of the theory in Section~\ref{subsection_diffeology}, where we equip the set $Bar$ of barcodes with a diffeology and identify the resulting smooth maps. We refer the reader to \cite{iglesias2013diffeology} for a detailed introduction to the material presented below. In the following, we call {\em domain} any open set in any arbitrary Euclidean space. 
\begin{definition}
\normalfont
\label{definition_diffeology}
Given a non-empty set $S$, a \textit{diffeology} is a collection $\mathcal{D}$ of pairs $(U,P)$, called \textit{plots}, where $U$ is a domain and $P:U\rightarrow S$ is a map from $U$ to $S$, satisfying the following axioms:
\begin{itemize}
\item[(\textbf{Covering})] For any element $s\in S$ and any integer $n\in \mathbb{N}$, the constant map $x\in \mathbb{R}^n \mapsto s\in S$ is a plot.
\item[(\textbf{Locality})] If for a pair $(U,P)$ we have that, for any $x\in U$ there exists an open neighborhood $U'\subseteq U$ of $x$ such that the restriction $(U',P_{|U'})$ is a plot, then $(U,P)$ itself is a plot.
\item[(\textbf{Smoothness compatibility})] For any plot $(U,P)$ and any smooth map~$F:W\rightarrow U$ where~$W$ is a domain, the composition $(W,P\circ F)$ is a plot.
\end{itemize} 
\end{definition}
If a set $S$ comes equipped with a diffeology $\mathcal{D}$, then it is called a \textit{diffeological space}. We think of a diffeological space $S$ as a space where we impose which functions, the plots, from a manifold to $S$, are smooth. Notice that any set can be made a diffeological space by taking all possible maps as plots. This is the \textit{coarsest} diffeology on $S$, where $\mathcal{D}$ is said to be \textit{finer} than the diffeology $\mathcal{D}'$ if $\mathcal{D}\subset \mathcal{D}'$, and \textit{coarser} if the converse inclusion holds.\footnote{This terminology is the opposite to the one used when comparing topologies.} The prototypical diffeological space is the Euclidean space $\mathbb{R}^n$ with the usual smooth maps from domains to~$\R^n$ as plots. 
\begin{definition}
\label{definition_diffeology_morph}
\normalfont
A \textit{morphism} $f:S\rightarrow S'$, or \textit{smooth map}, between two diffeological spaces $S$ and $S'$, is a map such that for each plot $P$ of $S$, $f\circ P$ is a plot of $S'$. $f$ is called a \textit{diffeomorphism} if it is a bijection and $f^{-1}:S'\rightarrow S$ is smooth. A map $f:A\rightarrow S'$, where $A\subseteq S$, is \textit{locally smooth} if for any plot $P$ of $S$, $f\circ P_{|P^{-1}(A)}$ is a plot of $S'$. $f$ is a \textit{local diffeomorphism} if it is a bijection onto its image and if $f^{-1}$ is locally smooth as a map $S'\supseteq f(A)\to S$. 
\end{definition}
Obviously, identities are smooth, and smooth maps compose together into smooth maps, therefore we can consider the category \textbf{Diffeo} of diffeological spaces. Finite dimensional smooth manifolds with or without boundaries and corners, Fr\'echet manifolds and Frölicher spaces, viewed as diffeological spaces with their usual smooth maps, form strict subcategories of $\textbf{Diffeo}$. In fact, finite dimensional smooth manifolds can be defined in the context of diffeology as follows:
\begin{definition}
\label{definition_diffeological_manifold}
\normalfont
A diffeological space $\manifoldm{}$ is a \textit{$n$-dimensional diffeological manifold} if it is locally diffeomorphic to $\mathbb{R}^n$ at every point in $\manifoldm{}$. 
\end{definition}
\begin{theorem}[{\cite[\textsection~4.3]{iglesias2013diffeology}}]\label{thm:smooth_diffeo_manifold}
    Every  $n$-dimensional smooth manifold~$\manifoldm$ is an $n$-dimensional diffeological manifold once equipped with the diffeology given by the smooth maps $U\to\manifoldm$ from arbitrary domains~$U$. Conversely, every  $n$-dimensional diffeological manifold is an $n$-dimensional smooth manifold.
\end{theorem}
One appealing feature of {\bf Diffeo}, compared to the category of smooth manifolds for instance, is that it is closed under usual set operations---here we only consider coproducts and quotients:
\begin{definition}
\label{definition_union_diffeo}
For an arbitrary family of diffeological spaces $\{(S_j,D_j)\}_{j\in \mathcal{J}}$, the {\em sum diffeology} on $\bigsqcup_{j\in \mathcal{J}} S_j$ is the finest diffeology making the injections $S_i\rightarrow \bigsqcup_{j\in \mathcal{J}} S_j$ smooth. 
\end{definition}
\begin{definition}
\label{definition_quotient_diffeo}
For a diffeological space $(S,\mathcal{D})$ and an equivalence relation $\sim$ on $S$, the~{\em quotient diffeology} on $S/\!\!\sim$ is the finest diffeology making the quotient map $S\rightarrow S/\!\!\sim$~smooth.
\end{definition}

\subsection{Stratified manifolds}
\label{section_stratification_subsection_stratification_definition}

Stratified manifolds play a role in Section~\ref{section_stratification} of this paper.
%
%
For background material on the subject, see e.g. \cite{mather2012notes}. 
%
\begin{definition}
\label{Whitney_stratified_manifold_definition}
Let $\paramspace{}$ be a smooth $d$-dimensional manifold. A \textit{Whitney stratification} $\mathcal{S}_\paramspace{}$ of $\paramspace{}$ is a collection of connected smooth submanifolds (not necessarily closed) of $\manifoldm$, called \textit{strata}, satisfying the following axioms:
\begin{itemize}
  \item[(Partition)] The strata partition $\paramspace{}$.
  \item[(Locally finite)] Each point of $\manifoldm$ has an open neighborhood meeting with finitely many strata.
  \item[(Frontier)] For each stratum $\paramspace{}' \in \mathcal{S}_\paramspace{}$, the set $\overbar{\paramspace{}'}\setminus\paramspace{}'$ is a union of strata, where $\overbar{\paramspace{}'}$ is the closure of $\paramspace{}'$ in $\paramspace$. 
  \item[(Condition b)] Consider a pair of strata $(\paramspace{}',\paramspace{}'')$ and an element $\param{}\in \paramspace{}'$. If there are sequences of points $(\param{}'_{k})_{k \in \mathds{N}}$ and $(\param{}''_{k})_{k \in \mathds{N}}$ lying in $\paramspace{}'$ and $\paramspace{}''$ respectively, both converging to $\param{}$, such that the line $(\param{}'_{k},\param{}''_{k})$ (defined in some local coordinate system around $\param$) converges to some line $l$ and $T_{\param{}''_{k}}\paramspace{}''$ converges to some flat, then this flat contains~$l$.
\end{itemize}
\end{definition}

Stratified maps are those that behave nicely with respect to stratifications. Here we only use a subset of the axioms they satisfy, hence we talk about {\em weakly stratified maps}.
\begin{definition}
\label{definition_stratified_map}

Let $\manifoldm, \manifoldn$ be stratified manifolds. A map $f:\manifoldm \rightarrow \manifoldn$ is \textit{weakly stratified} if the pre-images $f^{-1}(\manifoldn')$, for any stratum $\manifoldn'\in \mathcal{S}_{\manifoldn}$, is a union of strata in $\mathcal{S}_{\manifoldm}$.
\end{definition}

\section{Differentiability for maps from or to the space of barcodes}
\label{section_definition_differentiability_barcode_valued_maps}

In Section~\ref{subsection_differentiability_barcode_valued_maps} we provide a general framework for studying the differentiability of maps from a smooth manifold to $Bar$. Then in Section~\ref{subsection_differentiability_vectorization} we provide the analogue for maps with $Bar$ as domain and a smooth manifold as co-domain. Both frameworks are in some sense dual to each other, and inspired by the theory of diffeological spaces---we develop this connection in Section~\ref{subsection_diffeology}. We then derive a chain rule in Section~\ref{subsection_chain_rule}: if a map between manifolds factors through $Bar$, then it is smooth whenever both terms in the factorization are smooth according to our definitions, and in this case its differential can be computed explicitly.  

\subsection{Differentiability of barcode valued maps}
\label{subsection_differentiability_barcode_valued_maps}

Throughout this section, $\manifoldm{}$ denotes a smooth finite-dimensional manifold without boundary, which may or may not be compact. Our approach to characterizing the smoothness of a barcode valued map is to factor it through the bundle of {\em ordered barcodes}:
\begin{definition}
\normalfont
\label{definition_ordered_barcodes_quotient_map_Q}
For each choice of non-negative integers $m,n$, the space of \textit{ordered barcodes} with $m$ finite bars and $n$ infinite ones is $\mathbb{R}^{2m}\times \mathbb{R}^n$, equipped with the Euclidean norm and the resulting smooth structure. The corresponding \textit{quotient map} $Q_{m,n}:\mathbb{R}^{2m}\times \mathbb{R}^n \rightarrow Bar$ quotients the space by the action\footnote{$\mathfrak{S}_m$ acts on $\mathbb{R}^{2m}$ by permutation of pairs of adjacent coordinates while $\mathfrak{S}_n$ acts on $\mathbb{R}^{n}$ by permutation of coordinates.} of the  product of symmetric groups~$\mathfrak{S}_m\times\mathfrak{S}_n$, that is: for any ordered barcode $\obarcode{}=(b_1,d_1,...,b_{m},d_{m},v_{1},...,v_{n}) \in \mathbb{R}^{2m} \times \mathbb{R}^n$,
\[ Q_{m,n}(\obarcode{}):=\{(b_i,d_i)\}_{i=1}^m \cup \{(v_{j},+\infty)\}_{j=1}^n \cup \Delta^\infty. \]
\end{definition}
One can think of an ordered barcode $\obarcode{}\in \mathbb{R}^{2m}\times \mathbb{R}^{n}$ as a vector describing a persistence diagram with at most $m$ bounded off-diagonal points and exactly $n$ unbounded points. The former have their coordinates encoded in the adjacent pairs of the $2m$ first components in $\obarcode{}$, while the latter have the abscissa of their left endpoint encoded in the last $n$ components of $\obarcode{}$. The quotient
 map~$Q_{m,n}$ forgets about the ordering of the bars in the barcodes. So far $Q_{m,n}$ is merely a map between sets, and it is natural to ask whether it is regular in some reasonable sense: 
\begin{proposition}
 \label{proposition_continuity_quotient_map}
For any $m,n\in \mathds{N}^2$, $Q_{m,n}$ is $1$-Lipschitz when $Bar$ is equipped with the bottleneck topology. 
 \end{proposition}
\begin{proof}
For any two elements $\obarcode{}_1, \obarcode{}_2 \in \mathbb{R}^{2m}\times \mathbb{R}^n$, there is an obvious matching $\gamma$ on their images $ Q_{m,n}(\obarcode{}_1),Q_{m,n}(\obarcode{}_2)$ given by matching the components of the vectors $\obarcode{}_1$ and $\obarcode{}_2$ entry-wise. The cost of this matching is then bounded above by the supremum norm of $\obarcode{}_1-\obarcode{}_2$, by the definition of the matching cost $c(\gamma)$. In turn, the supremum norm is bounded above by the $\ell^2$ norm.
\end{proof}

We then say that a barcode valued map is smooth if it admits a smooth lift into the space of ordered barcodes for some choice of $m,n$:
\begin{definition}
\normalfont
\label{differentiability_definition_barcode}
Let $B: \manifoldm{} \rightarrow Bar$ be a barcode valued map. Let $x \in \manifoldm{}$ and $r\in \mathds{N}\cup\{+\infty\}$.
We say that \textit{$B$ is $r$-differentiable at $x$} if there exists an open neighborhood $U$ of $x$, integers $m,n\in \mathbb{N}$ and a map $\tilde{B}:U\rightarrow \mathbb{R}^{2m}\times \mathbb{R}^n$ of class $C^r$ such that $B=Q_{m,n}\circ \tilde{B}$ on $U$. For an integer $d\in \mathds{N}$, a function $\mathcal{B}: \manifoldm{} \rightarrow Bar^{d+1}$ is \textit{$r$-differentiable at $\elementm{}\in \manifoldm{}$} if each of its $d+1$ components is. We call $\tilde{B}$ a \textit{local lift} of $B$.
\end{definition}
\begin{remark}[Locally finite number of off-diagonal points]
\label{remark_locally_finite_off_points}
\normalfont
If a function~$B$ as above is $r$-differentiable at~$\elementm{} \in \manifoldm{}$, then locally for any~$\elementm{}'$ around~$\elementm{}$ we can upper-bound the number of off-diagonal points arising in~$B(\elementm{}')$ by~$m+n$. Notice that off-diagonal points can possibly appear in $B(\elementm{}')$ and become part of the diagonal $\Delta$ in $B(\elementm{})$, which is to say that Defnition \ref{differentiability_definition_barcode} does not restrict the function $B$ to locally consist in a fixed number of off-diagonal points. Informally, in analogy with the fact that a barcode has finitely many off-diagonal points, our definition of smoothness allows finitely many appearances or disappearances of off-diagonal points in the neighborhood of a barcode.
\end{remark}
\begin{remark}[$0$-differentiability is stronger than bottleneck continuity]
\normalfont
\label{remark_relationship_continuity_B}
If $B:\manifoldm{}\rightarrow Bar$ is $0$-differentiable, then $B$ is continuous when $Bar$ is given the bottleneck topology. This comes from the Lipschitz continuity of $Q_{m,n}$ (Proposition \ref{proposition_continuity_quotient_map}) and the fact that continuity is stable under composition. The converse is false, because, on the one hand, if $B$ is $0$-differentiable then locally the number of off-diagonal points in the image of $B$ is uniformly bounded (see the previous remark), while on the other hand, the number of off-diagonal points appearing in barcodes in any given open bottleneck ball is arbitrarily large.
\end{remark}

\begin{definition}
\label{definition_derivatives_barcode_valued_maps}
\normalfont
Let $B:\manifoldm{}\rightarrow Bar$ be $1$-differentiable at some $x$, and $\tilde{B}:U \rightarrow \mathbb{R}^{2m}\times \mathbb{R}^n$ be a $C^1$ lift of $B$ defined on an open neighborhood $U$ of $x$. The \textit{differential} (or \textit{derivative}) $d_{x,\tilde{B}} B$ of $B$ at $x$ with respect to $\tilde{B}$ is defined to be the differential of $\tilde B$ at~$x$: $$ T_x \manifoldm{} \xrightarrow[d_x \tilde{B}]{} \mathbb{R}^{2m}\times \mathbb{R}^n.$$
\end{definition}
Post-composing with the quotient map, we can see $Q_{m,n}\circ d_{x,\tilde{B}} B: T_x \manifoldm{}\rightarrow Bar$ as a multi-set of co-vectors, one above each off-diagonal point of $B(x)$ (plus some distinguished diagonal points), describing linear changes in the coordinates of the points of $B(x)$ under infinitesimal perturbations of $x$. In this respect, the spaces of ordered barcodes $\mathbb{R}^{2m+n}$ play the role of tangent spaces over $Bar$. For practical computations, it can be convenient to work with an alternate yet equivalent notion of differentiability, based on point trackings:
\begin{definition}
\normalfont
\label{definition_coordinate_system}
Let $B: \manifoldm{} \rightarrow Bar$ be a barcode valued map. Let $x \in \manifoldm{}$ and $r\in \mathds{N}\cup\{+\infty\}$. A $C^r$ \textit{local coordinate system} for $B$ at $x$ is a collection of maps $\{b_i,d_i:U \rightarrow \mathbb{R}\}_{i\in I}$ and $\{v_j: U \rightarrow \mathbb{R}\}_{j \in J}$ for finite sets $I,J$ defined on an open neighborhood $U$ of $x$, such that:
\begin{itemize}
    \item[(\textbf{Smooth})] The maps $b_i,d_i,v_j$ are of class $C^r$;
    \item[(\textbf{Tracking})] For any $x' \in U$ we have the multi-set equality \[B(\elementm{}')=\{(b_i(\elementm{}'),d_i(\elementm{}'))\}_{i\in I} \cup \{(v_j(\elementm{}'),+\infty)\}_{j\in J} \cup \Delta^\infty .\]
\end{itemize}
\end{definition}
Thus, in a local coordinate system, we have maps $b_i,d_i$ (resp. $v_j$) that track the endpoints of bounded (resp. unbounded) intervals in the image barcode through~$B$. We will often abbreviate the data of a local coordinate system of $B$ at $\elementm{}$ by $\mathcal{T}=(U,\{b_i,d_i\}_{i\in I}, \{v_j\}_{j \in J})$.

Our two notions of differentiability are indeed equivalent:
\begin{proposition}
\label{proposition_relating_differentiability_barcode_ordered_barcode}
Let $B: \manifoldm{} \rightarrow Bar$ be a barcode valued map and $\elementm{} \in \manifoldm{}$. Then $B$ is $r$-differentiable at $\elementm{}$ if and only if it admits a $C^r$ local coordinate system at~$\elementm{}$. Specifically, post-composing a $C^r$ local lift $\tilde{B}: U \to \R^{2m}\times\R^n$ around~$x$ with the quotient map~$Q_{m,n}$ yields a $C^r$ local coordinate system, and conversely, fixing an order on the functions of a $C^r$ local coordinate system yields a $C^r$ local lift.
\end{proposition}
\begin{proof}
  $(\Rightarrow)$  Let $\tilde{B}:U\rightarrow \mathbb{R}^{2m}\times \mathbb{R}^n$ be a $C^r$ local lift of $B$ at $x$. Extract the components of $(b_1(\elementm{}'),d_1(\elementm{}'),...,b_m(\elementm{}'),d_m(\elementm{}'),v_1(\elementm{}'),...,v_n(\elementm{}')):=\tilde{B}(\elementm{}')$ to get a local coordinate system, which is $C^r$ over $U$ as $\tilde{B}$ is. 
$(\Leftarrow)$ Let $\mathcal{T}=(U,\{b_i,d_i\}_{i\in I}, \{v_j\}_{j \in J})$  be a $C^r$ local coordinate system for $B$ at $x$. Set $m=|I|$ and $n=|J|$, and fix two arbitrary bijections $s:\{1,...,m\}\to I$ and $t:\{1,...,n\}\to J$. Then the map $\tilde{B}:U\rightarrow \mathbb{R}^{2m}\times \mathbb{R}^n$ defined as:
\[\tilde{B}(\elementm{}'):= [b_{s(1)}(\elementm{}'),d_{s(1)}(\elementm{}'),...,b_{s(m)}(\elementm{}'),d_{s(m)}(\elementm{}'),v_{t(1)}(\elementm{}'),...,v_{t(n)}(\elementm{}')]\] is a lift of~$B$. As a map valued in a Euclidean space,~$\tilde{B}$ is~$C^r$ because all its coordinate functions are. 
\end{proof}

\begin{remark}[Non uniqueness of differentials]
\normalfont
\label{remark_non_uniqueness_derivatives}
It is important to keep in mind that the differential of~$B$ at~$\elementm{}$ is not uniquely defined, as it depends on the
choice of local lift. Indeed, for two distinct lifts $\tilde B, \tilde B'$ of $B$ at $\elementm{}$, we usually get
distinct differentials $d B_{\elementm{},\tilde B}$, $d
B_{\elementm{},\tilde B'}$. For instance, if $\tilde B'$ is obtained from~$\tilde B$ by appending an extra pair of coordinates of the form~$(f, f)$, where $f$ is a smooth real function, then  $d B_{\elementm{},\tilde B'}$ takes its values in a different codomain than that of $d B_{\elementm{},\tilde B}$.
Note that this will not be an issue in the rest of the paper, as any choice of differential will yield a valid chain rule (Section \ref{subsection_chain_rule}).
\end{remark}

\subsection{Differentiability of maps defined on barcodes}
\label{subsection_differentiability_vectorization}

Let $\manifoldn{}$ be a smooth finite-dimensional manifold without boundary. Our notion of differentiability for maps $V: Bar \rightarrow \manifoldn{}$ is in some sense dual to the one for maps $B:\manifoldm{}\rightarrow Bar$, as will be justified formally in the next section.
\begin{definition}
\label{definition_smooth_vectorization}
\normalfont
Let $V: Bar \rightarrow \manifoldn{}$ be a map on barcodes. Let $D\in Bar$ and $r\in \mathds{N}\cup \{+\infty\}$. $V$ is said to be \textit{$r$-differentiable} at $D$, if for all integers $m,n$ and all vectors $\obarcode{}\in \mathbb{R}^{2m}\times \mathbb{R}^n$ such that $Q_{m,n}(\obarcode{})=D$, the map $V\circ Q_{m,n}: \mathbb{R}^{2m}\times \mathbb{R}^n\rightarrow \manifoldn$ is $C^r$ on an open neighborhood of $\obarcode{}$.
\end{definition}
Notice that for each choice of $m,n$ we have a unique map $V\circ Q_{m,n}$, and we must check its differentiability at all the (possibly many) distinct pre-images $\obarcode{}$ of $D$ and for all $m, n$. One can think of a choice of $m,n$ and pre-image $\obarcode{}$ of $D$ as a choice of tangent space of~$Bar$ at~$D$.

\begin{example}[Total persistence function]
\label{example_total_persistence}
Let $V:Bar\rightarrow \mathbb{R}$ be defined as the sum, over bounded intervals $(b,d)$ in a barcode $D$, of the length $(d-b)$. Given $D\in Bar$ and an ordered barcode $\obarcode\in \mathbb{R}^{2m+n}$ such that $Q_{m,n}(\obarcode{})=D$, the map $V\circ Q_{m,n}$ is a linear form and in particular is of class $C^\infty$ at $\obarcode$. Explicitly, we have
$$V\circ Q_{m,n}: (b_1,d_1,...,b_m,d_m,v_1,...,v_n)\in \mathbb{R}^{2m+n} \mapsto \sum_{i=1}^m d_i-b_i \in \mathbb{R}$$
Therefore, $V$ is $\infty$-differentiable everywhere on $Bar$.
\end{example}
The relationship between $0$-differentiability and the bottleneck continuity for maps $V$ is the opposite to the one that holds for maps $B$ (recall Remark~\ref{remark_relationship_continuity_B}):
\begin{remark}[Bottleneck continuity is stronger than $0$-differentiability]
\normalfont
If $V:Bar\rightarrow \manifoldn{}$ is continuous when $Bar$ is equipped with the bottleneck topology, then $V$ is $0$-differentiable. This is because the quotient map $Q_{m,n}$ is continuous (Proposition \ref{proposition_continuity_quotient_map}) and the composition of continuous maps is continuous. The converse is false, as seen for instance when taking $V$ to be the total persistence function: although $0$-differentiable (because $\infty$-differentiable) on $Bar$, $V$ is not continuous in the bottleneck topology as it is unbounded in any open bottleneck ball. 
\end{remark}
\begin{definition}
\label{definition_derivative_vectorization}
\normalfont
Let $V: Bar \rightarrow \manifoldn{}$ be $1$-differentiable at $D\in Bar$, and $\obarcode{}\in \mathbb{R}^{2m+n}$ be a pre-image of $D$ via $Q_{m,n}$. The \textit{differential} (or \textit{derivative}) of $V$ at $D$ with respect to $\obarcode{}$ is the map 
$$d_{D,\obarcode{}} V: \mathbb{R}^{2m+n} \xrightarrow[d_{\obarcode{}} (V\circ Q_{m,n})]{} T_{V(D)} \manifoldn{}.$$
\end{definition}

\subsection{Chain rule}
\label{subsection_chain_rule}

We now combine the previous definitions to produce a chain rule.
\begin{proposition}
\label{proposition_chain_rule_vectorization_barcode}
Let $B:\manifoldm{} \rightarrow Bar$ be $r$-differentiable at $\elementm{} \in \manifoldm{}$, and  $V:Bar \rightarrow \manifoldn{}$ be $r$-differentiable at $B(\elementm{})$. Then: 
\begin{itemize}
    \item [(i)] $V\circ B:\manifoldm{} \rightarrow \manifoldn{}$ is $C^r$ at $\elementm{}$ as a map between smooth manifolds;
    \item [(ii)] If $r\geq 1$, then for any local $C^1$ lift $\tilde{B}:U \rightarrow \mathbb{R}^{2m+n}$ of $B$ around~$\elementm$ we have: 
      \[ d_{\elementm{}} (V\circ B)= d_{B(x),\tilde{B}(\elementm{})}V \circ d_{\elementm{},\tilde{B}}B. \]
\end{itemize}
\end{proposition}
The meaning of this formula is that, even though the differentials of~$B$ and of~$V$ may
depend on the choice of lift~$\tilde
B:\manifoldm\to\mathbb{R}^{2m+n}$, their composition
does not, and in fact it matches with the usual differential of~$V\circ B$ as a map between smooth manifolds.
\begin{proof}
Since $B$ is $r$-differentiable at $\elementm$, there exists an open neighborhood~$U$ of~$\elementm$ and a local $C^r$ lift $\tilde{B}: U \rightarrow \mathbb{R}^{2m}\times \mathbb{R}^n$ for some integers $m,n$, such that $B|_U = Q_{m,n}\circ \tilde B$. Meanwhile, since $V$ is $r$-differentiable at $B(\elementm)$, the map $V\circ Q_{m,n}: \mathbb{R}^{2m}\times \mathbb{R}^n\rightarrow \manifoldn{}$ is $C^r$ at $\tilde{B}(\elementm{})$. This implies that the composition $V\circ B|_U=(V\circ Q_{m,n}) \circ \tilde{B}$ is $C^r$ at~$\elementm{}$, and therefore that $V\circ B$ itself is $C^r$ at~$\elementm$ since $U$ is open. This proves~(i). The formula of~(ii) follows then from applying the usual chain rule to $(V\circ Q_{m,n})$ and $\tilde{B}$, which are $C^1$ maps between smooth manifolds without boundary.
\end{proof}
\begin{example}
In~\cite{2019arXiv190510996H}, given a $C^\infty$ neural network architecture $F_0:\R^{N}\rightarrow \R^{K_0}$ valued in the set of functions over the vertices of a fixed graph $K$, the optimization pipeline requires taking the gradient of the following loss function:
\[\mathcal{L}:\param{} \in \R^N \longmapsto \sum_{(b,d)\in \Dgm_p(F_0(\param))\setminus \Delta \ \mathrm{ bounded}} s(b,d) \in \R,\]
where $s:\R^2\rightarrow \R$ is a fixed smooth map, and $\Dgm_p(F_0(\param))$ is the degree-$p$ persistence diagram associated to the lower star filtration induced by $F_0(\param)$ on $K$ (see Section \ref{section_discrete_smoothness_subsection_examples_ex1} dedicated to the full analysis of lower star filtrations). We may see $\mathcal{L}$ as the composition:
\[ \mathcal{L}: \xymatrix{ \param \in \R^N \ar^-{B}[r] &  \Dgm_p(F_0(\param))\in Bar \ar^-{V}[r] & V(\Dgm_p(F_0(\param)))\in \R}, \]
where $V: D\in Bar \mapsto \sum_{(b,d)\in D\setminus \Delta \ \mathrm{ bounded}} s(b,d) \in \R$. On the one hand, $B$ is $\infty$-differentiable at every $\param$ where $\parametrization_0(\param)$ is injective, as will be detailed in Section \ref{section_discrete_smoothness_subsection_examples_ex1}. On the other hand, $V$ is $\infty$-differentiable everywhere on $Bar$, a fact obtained exactly as in the case of the total persistence function of Example \ref{example_total_persistence}. By the chain rule (Proposition \ref{proposition_chain_rule_vectorization_barcode}), we deduce that the loss $\mathcal{L}$ is smooth at every $\param$ where $\parametrization_0(\param)$ is injective. Thus we recover the differentiability result of \cite{2019arXiv190510996H}. In fact, the upcoming Theorem~\ref{theorem_PH_differentiable_global} ensures that $B$ is $\infty$-differentiable over an open dense subset of $\R^N$, and therefore so is $\mathcal{L}$ by the chain rule. 
\end{example}

\subsection{Higher-order derivatives}

The notions of derivatives introduced in Definitions~\ref{definition_derivatives_barcode_valued_maps} and~\ref{definition_derivative_vectorization} extend naturally to higher orders. For simplicity, we place ourselves in the Euclidean setting,  letting $\manifoldm=\R^N$ and $\manifoldn=\R^{N'}$ for some $N, N'\in \mathbb{N}$. 
\begin{definition}
\label{definition_higher_order_derivatives}
Let $B:\R^N\rightarrow Bar$ be $r$-differentiable at some $x$, and $\tilde{B}:U \rightarrow \mathbb{R}^{2m}\times \mathbb{R}^n$ be a $C^r$ lift of $B$ defined on an open neighborhood $U$ of $x$. The \textit{$r$-th differential} (or \textit{derivative}) of $B$ at $x$ with respect to $\tilde{B}$ is defined to be the $r$-th Fr\'echet differential of $\tilde{B}$ at $x$:
$$d^r_x \tilde{B}: (\mathbb{R}^{N})^r \xrightarrow[]{} \mathbb{R}^{2m}\times \mathbb{R}^n. $$
\end{definition}
Dually:
\begin{definition}
\label{definition_derivative_vectorization_higher}
\normalfont
Let $V: Bar \rightarrow \R^{N'}$ be $r$-differentiable at $D\in Bar$, and $\obarcode{}\in \mathbb{R}^{2m+n}$ be a pre-image of $D$ via $Q_{m,n}$. The \textit{$r$-th differential} (or \textit{derivative}) of $V$ at $D$ with respect to $\obarcode{}$ is the $r$-th Fr\'echet differential of $V\circ Q_{m,n}$ at $\obarcode$: 
$$d^r_{\obarcode{}} (V\circ Q_{m,n}):(\mathbb{R}^{2m+n})^r \xrightarrow[]{} \R^{N'}.$$
\end{definition}
Note that, given maps $B:\R^N\to Bar$ and $V: Bar\to \R^{N'}$ that are $r$-differentiable at $\elementm$ and $B(\elementm)$ respectively, the chain rule of Section~\ref{subsection_chain_rule} adapts readily to higher-order derivatives of $B\circ V$ at $x$. 

Meanwhile, we get a natural {\em Taylor expansion} of $B$ at $x$ with respect to $\tilde{B}$: 
\[T^r_{x,\tilde{B}} B: 
  h\in \mathbb{R}^{N} \longmapsto \tilde{B}(x)+ d_{x}\tilde{B}(h)+ \cdots+ \frac{1}{r!}d^r_{x}\tilde{B}(h,\cdots,h) \in \R^{2m}\times \R^n. \]

\begin{proposition}
\label{proposition_Taylor}
Let $B:\R^N\rightarrow Bar$ be $r$-differentiable at some $x$, and $\tilde{B}:U \rightarrow \mathbb{R}^{2m}\times \mathbb{R}^n$ be a $C^r$ lift of $B$ defined on an open neighborhood $U$ of $x$. 
%
%
Then, \[\db(B(x+h),(Q_{m,n}\circ T^r_{x,\tilde{B}} B) (h))=o(\|h\|^r).\]
\end{proposition}
\begin{proof}
This follows from applying the standard Taylor-Young theorem to $\tilde{B}$, then post-composing by $Q_{m,n}$---which is $1$-Lipschitz by Proposition~\ref{proposition_continuity_quotient_map}.
\end{proof}

To our knowledge, there is in general no equivalent of this result for the map $V$, due to the lack of a Lipschitz-continuous section of $Q_{m,n}$.

\subsection{The space of barcodes as a diffeological space}
\label{subsection_diffeology}

In this subsection, we detail how $Bar$, when viewed as the quotient of a disjoint union of Euclidean spaces, is canonically made into a diffeological space, as defined in Section~\ref{section_diffeology_theory_preliminary_notions}. We then show that the resulting notions of diffeological smooth maps from and to~$Bar$ coincide with the definitions~\ref{differentiability_definition_barcode} and~\ref{definition_smooth_vectorization} of differentiability we chose for maps from and to $Bar$ in the previous sections, thus making these two definitions dual to each other.

As a set, $Bar$ is isomorphic to $\left(\bigsqcup_{m,n\in \mathbb{N}}\R^{2m+n}\right)/\!\!\sim$, where $\sim$ is the transitive closure of the following relations for $m,n$ ranging over~$\mathbb{N}$:
\begin{itemize}
    \item For any permutations $\pi,\tau$ of $\{1,...,m\}$ and $\{1,...,n\}$ respectively,  \[[(b_i,d_i)_{i=1}^m, (v_j)_{j=1}^n]\sim [(b_{\pi(i)},d_{\pi(i)})_{i=1}^m, (v_{\tau(j)})_{j=1}^n]\text{,}\]
    which indicates that persistence diagrams are multisets (i.e. intervals are not ordered);
    \item Any element $[(b_i,d_i)_{i=1}^m, (v_j)_{j=1}^n]\in \R^{2m+n}$ such that one of the first $m$ adjacent pairs $(b_i, d_i)$ satisfies  $b_i=d_i$ is equivalent to the element of $\R^{2(m-1)+n}$ obtained by removing $(b_i,d_i)$. These identifications correspond to quotienting multisets by the diagonal~$\Delta$.
\end{itemize}

Since the Euclidean spaces~$\R^{2m+n}$ are equipped with their Euclidean diffeologies, we obtain a canonical diffeology~$\mathcal{D}(Bar)$ over~$Bar$ from Definitions~\ref{definition_union_diffeo} and~\ref{definition_quotient_diffeo}. The plots of~$\mathcal{D}(Bar)$ can be concretely characterized as follows:
\begin{proposition}
\label{proposition_plots_infinity_differentiable}
Let $U \subseteq \mathbb{R}^d$ be open and $B:U \rightarrow Bar $. Then~$B$ is a plot in~$\mathcal{D}(Bar)$ if and only if, for every $x\in U$, there exists an open neighborhood $V\subseteq U$ of $x$ and a $C^\infty$ lift $\tilde{B}:V\rightarrow \mathbb{R}^{2m+n}$ such that $B_{|V}= Q_{m,n} \circ \tilde{B}$. 
\end{proposition}
In other words, a plot in~$\mathcal{D}(Bar)$ is an $\infty$-differentiable map from a domain $U$ to $Bar$.
\begin{proof}
Note that the characterization of the quotient diffeology, as given in Definition~\ref{definition_quotient_diffeo}, is in fact the  characterization of the so-called {\em push-forward diffeology} induced by the quotient map---see \cite[\textsection~1.43]{iglesias2013diffeology}. According to that characterization, $B:U \rightarrow Bar$ is a plot if and only if, for every element $z\in U$, there exists an open neighborhood $W\subseteq U$ of $z$ such that the restriction $B_{|W}$ admits a lift\footnote{Strictly speaking, according to \cite[\textsection~1.43]{iglesias2013diffeology} there is also the alternative that the restriction $B_{|W}$ be constant, but in this case it also admits a lift to $\bigsqcup_{m,n\in \mathbb{N}}\R^{2m+n}$. Indeed, calling~$D$ the unique barcode in the image of~$B_{|W}$, we can choose one pre-image~$\tilde D$ of~$D$  in one of the spaces of ordered barcodes~$\R^{2m+n}$, then take $\tilde B$ to be the constant map $W\to\{\tilde D\}$.} $\tilde{B}:W \rightarrow \bigsqcup_{m,n\in \mathbb{N}}\R^{2m+n}$, i.e. a plot~$\tilde{B}$ of $\bigsqcup_{m,n\in \mathbb{N}}\R^{2m+n}$ that matches with $B_{|W}$ once post-composed with the quotient map modulo $\sim$. In turn, by the characterization of the sum diffeology in \cite[\textsection~1.39]{iglesias2013diffeology}, $\tilde{B}$ is a plot of $\bigsqcup_{m,n\in \mathbb{N}}\R^{2m+n}$ if and only if, for any $x\in W$, there is an open neighborhood $V\subseteq W$ of $x$ and a pair of indices~$(m,n)$  such that the restriction $\tilde{B}_{|V}$
maps into $\R^{2m+n}$ and is in fact a plot of 
$\R^{2m+n}$.
Equivalently, we have $B_{|V} = Q_{m,n} \circ \tilde{B}_{|V}$, where $\tilde{B}_{|V}$ is of class $C^\infty$ (since the spaces of ordered barcodes are equipped with their canonical Euclidean diffeologies).
\end{proof}

\begin{corollary}
  The smooth maps in \textbf{Diffeo} from a smooth manifold  $\manifoldm{}$ without boundary (equipped with the diffeology from Theorem~\ref{thm:smooth_diffeo_manifold}) to the diffeological space $Bar$ are exactly the $\infty$-differentiable maps from $\manifoldm{}$ to $Bar$.
\end{corollary}
\begin{proof}
Let $B:\manifoldm{}\rightarrow Bar$ be a smooth map in \textbf{Diffeo}. For any plot $\phi:U\rightarrow \manifoldm{}$, the composition $B\circ\phi$ is a plot in $\mathcal{D}(Bar)$, therefore it locally rewrites as $Q_{m,n}\circ \tilde{B}$ for some $C^\infty$ lift $\tilde{B}$, by Proposition \ref{proposition_plots_infinity_differentiable}. Choosing $\phi$ to be a local coordinate chart, we then locally have $B=Q_{m,n}\circ \tilde{B}\circ \phi^{-1}$, which means that $B$ is $\infty$-differentiable. Conversely, if $B$ is $\infty$-differentiable, it locally rewrites as $B=Q_{m,n}\circ \tilde{B}$, hence for any plot $\phi:U\rightarrow \manifoldm{}$ the composition $B \circ \phi$ locally rewrites as $Q_{m,n}\circ \tilde{B}\circ \phi$ and therefore is a plot in $\mathcal{D}(Bar)$ by Proposition \ref{proposition_plots_infinity_differentiable}. 
\end{proof}
Dually:
\begin{corollary}
The smooth maps in \textbf{Diffeo} from the diffeological space $Bar$ to a smooth manifold $\manifoldn{}$ without boundary (equipped with the diffeology from Theorem~\ref{thm:smooth_diffeo_manifold})  are exactly the $\infty$-differentiable maps from $Bar$ to $\manifoldn{}$.
\end{corollary}
\begin{proof}
Let $V:Bar \rightarrow \manifoldn{}$ be a smooth map in \textbf{Diffeo}. By Proposition \ref{proposition_plots_infinity_differentiable}, any $\infty$-differentiable map $B:U\rightarrow Bar$ defined on a domain $U$ is a plot, therefore the composition $V\circ B:U\to\manifoldn$ is a plot hence $C^\infty$. In
particular, the map $Q_{m,n} = Q_{m,n}\circ \text{Id}_{\R^{2m+n}} :
\R^{2m+n} \to Bar$ is $\infty$-differentiable, therefore $V\circ Q_{m,n}$ is $C^\infty$. This shows that $V$ is $\infty$-differentiable. Conversely, if $V$ is $\infty$-differentiable, the maps $V\circ Q_{m,n}:\mathbb{R}^{2m}\times \mathbb{R}^n\rightarrow \manifoldn$, for varying integers $m,n$, are $C^\infty$. By Proposition \ref{proposition_plots_infinity_differentiable}, if $B:U\rightarrow Bar$ is a plot, then it locally rewrites as $Q_{m,n}\circ \tilde{B}$ for some $C^\infty$ lift $\tilde{B}$, therefore $V\circ B$ is locally of the form $(V\circ Q_{m,n}) \circ \tilde{B}$, which is of class $C^\infty$ as a map between manifolds by the chain rule. Thus, $V\circ B$ is a plot, and therefore $V$ is smooth in \textbf{Diffeo}. 
\end{proof}

Conceptually, we have made $Bar$ into a diffeological space by viewing it as the quotient of the direct limit of the spaces of ordered barcode. Then, $\infty$-differentiable maps are simply morphisms in \textbf{Diffeo} from or to smooth manifolds, rather than maps satisfying the a priori unrelated definitions \ref{differentiability_definition_barcode} and \ref{definition_smooth_vectorization}. More generally, by seeing $Bar$ as one object in $\textbf{Diffeo}$ where morphisms can come in or out, we have notions of smooth maps from or to $Bar$ with respect to any other diffeological space. For instance, a map $f:Bar\rightarrow Bar$ is smooth if and only if all the maps $f\circ Q_{m,n}$, for varying integers $m,n$, are $\infty$-differentiable (the proof is left as an exercise to the reader). Note however that diffeology does not characterize the $r$-differentiable maps for finite~$r$ nor the maps that are differentiable only locally, two concepts that are prominent in our analysis.

\section{The case of barcode valued maps derived from real functions on a simplicial complex}
\label{section_discrete_smoothness}

In this section we consider barcode valued maps $B_p: \manifoldm{} \to Bar$ that factor through the space~$\mathbb{R}^K$ of real functions on a fixed finite abstract simplicial complex~$K$:
\[ B_p:\ \xymatrix{\manifoldm{} \ar^-{F}[r] & \mathbb{R}^K \ar^-{\Dgm_p}[r] & Bar} \]
In other words, we consider barcodes derived from real functions
on~$K$. Note that $\Dgm_p$, the barcode map in degree~$p$, is only
defined on the subspace of filter functions, i.e. functions $K\to
\mathbb{R}$ that are monotonous with respect to inclusions of faces
in~$K$. This subspace is a convex polytope bounded by the hyperplanes
of equations $f(\simplex) = f(\simplex')$ for
$\simplex\subsetneq\simplex'\in K$. From now on, we consistently assume that
$F$ takes its values in this polytope.

\begin{example}[Height filters]
\label{example_parametrization}
\normalfont
Given an embedded simplicial comple~$K\subseteq \mathbb{R}^d$, let $\manifoldm{}=\mathbb{S}^{d-1}$ and $F:\param \mapsto (\simplex\in K \mapsto \max_{x\in\simplex}\, \langle \param,x \rangle)$. The filter functions considered here are the height functions on~$K$, parametrized on the unit sphere~$\mathbb{S}^{d-1}$ by the map~$F$.
\end{example}

By analogy with the previous example, we generally call~$F$ the {\em
  parametrization} associated to~$B$, although it may not always be a
topological embedding of~$\manifoldm{}$ into~$\mathbb{R}^{K}$ (it
may not even be injective). We also call~$\manifoldm$ the {\em parameter space}, and use the generic notation~$\param{}$ to refer to an element in~$\paramspace$.

As we shall see in Section~\ref{section_discrete_smoothness_subsection_smoothness_theorem}, a local coordinate system for the map $B_p$ at $\param{}\in\manifoldm{}$ can be derived when  the order of the values of the filter function $F(\param{})$ remains constant locally around $\param{}$. For this purpose we introduce the following equivalence relation on filter functions $K\to\mathbb{R}$:

\begin{definition}
\normalfont
\label{definition_ordering_equivalent_and_integer_filter}
Given a filter function $f:K\rightarrow \mathbb{R}$, the increasing order of its values induce a pre-order on the simplices of $K$. Two filter functions $f,g$ are said to be \textit{ordering equivalent}, written $f\sim g$, if they induce the same pre-order on $K$. This relation is an equivalence relation on filter functions, and we denote by $[f]$ the equivalence class of $f$. The (finite) set of equivalence classes is denoted by~$\Omega(\mathbb{R}^K)$.
\end{definition}
In order to compare barcodes across an entire equivalence class of functions, we introduce  {\em barcode templates} as follows:
\begin{definition}
\label{definition_pairing_simplices}
Given a filter function $f\in \mathbb{R}^K$ and a homology degree $0\leqslant p \leqslant d$, a \textit{barcode template} $(P_p, U_p)$ is composed of a multiset $P_p$ of pairs of simplices in $K$, together with a multiset $U_p$ of simplices in $K$, such that:
\begin{equation}
\label{equation_pairing_definition}
    \Dgm_p(f)=\big\{(f(\simplex),f(\simplex{}'))\big \}_{(\simplex,\simplex{}')\in P_p}\cup \big\{(f(\simplex),+\infty)\big \}_{\simplex{}\in U_p}\cup \Delta^\infty
\end{equation}
Note that we do not require a priori that $\dim \simplex =p$ and $\dim\simplex' = p+1$.
\end{definition}
\begin{proposition}
\label{proposition_existence_pairing}
For any filter function $f\in \mathbb{R}^K$ and homology degree $0\leqslant p \leqslant d$, there exists a barcode template $(P_p,U_p)$ of $f$.
\end{proposition}
\begin{proof}
  Consider the interval decomposition $\mathbf{H}_p(f) \cong \oplus_{J \in \mathcal{J}} \mathbb{I}_J$ of the $p$-th persistent homology module of~$f$. Note that every interval endpoint in the decomposition corresponds to the $f$-value of some simplex of $K$ (since the persistent homology module has internal isomorphisms in-between these values). For every bounded interval $J$ with endpoints $b,d\in\R$ choose an element $(\sigma_J, \sigma'_J)$ in $f^{-1}(b)\times f^{-1}(d) \subseteq K\times K$, then form the multiset $P_p := \{(\sigma_J, \sigma'_J) \mid J\in\mathcal{J}\ \mathrm{bounded}\}$. Meanwhile, for every unbounded interval $J$ with finite endpoint $v\in\R$ choose an element $\sigma_J$ in $f^{-1}(v)$, then form the multiset $U_p := \{\sigma_J \mid J\in\mathcal{J}\ \mathrm{unbounded}\}$.
\end{proof}
Barcode templates get their name from the fact that they are an invariant of the ordering equivalence relation~$\sim$:
\begin{proposition}
\label{proposition_pairing_same_equivalence_class}
If $f, f'$ are ordering equivalent filter functions, then any barcode template
of $f$ is also a barcode template of~$f'$ and vice-versa.
\end{proposition}
The proof, detailed hereafter, relies on the following elementary lemma.
\begin{lemma}
\label{lemma_reparametrization_module}
Let $\mathbb{V}$ be a persistence module, and $h:\mathbb{R}\rightarrow \mathbb{R}$ be a continuous increasing function. Denote by $\mathbb{V}_h$ the shift of $\mathbb{V}$ by $h$, i.e for any $s\leqslant t$, $\mathbb{V}_{h,t}:=\mathbb{V}_{h(t)}$ and $v_{s,t}^{\mathbb{V}_h}:=v_{h(s),h(t)}^{\mathbb{V}}$. If $\mathbb{V}$ decomposes as $\mathbb{V}\cong \oplus_{J\in \mathcal{J}} \mathbb{I}_J$, then $\mathbb{V}_{h}\cong \oplus_{J\in \mathcal{J}} \mathbb{I}_{h
^{-1}(J)}$.
\end{lemma}
\begin{proof}
The operation that takes a persistence module to its shift by $h$ is an endofunctor of $\mathbf{Pers}$ which commutes with direct sums. In particular it preserves isomorphisms.
\end{proof}
\begin{proof}[Proof of Proposition \ref{proposition_pairing_same_equivalence_class}]
  Let $f,f'$ be two ordering equivalent filter functions. Since $f\sim f'$, we have $f(\simplex{})=f(\simplex{}')\Rightarrow f'(\simplex{})=f'(\simplex{}')$ for any pair of simplices $\simplex{},\simplex{}'\in K$. Therefore the map $h:f(\simplex{})\in f(K)\mapsto f'(\simplex{})\in f'(K)$ is well-defined. Furthermore, $h$ is an increasing function and we extend it monotonously and continuously over all $\mathbb{R}$. Then, by the reparametrization Lemma~\ref{lemma_reparametrization_module}, any barcode template of $f$ is also a barcode template of $f'$. 
\end{proof}

\subsection{Generic smoothness of the barcode valued map}
\label{section_discrete_smoothness_subsection_smoothness_theorem}

We now state our first significant results (one local and the other
global) about the differentiability of the map~$B_p$ in the context of
this section.  Equipping~$\mathbb{R}^{K}$ with the usual Euclidean
norm, we assume that the parametrization~$F$ is of class~$C^r$ as a
map $\manifoldm{}\to\mathbb{R}^{K}$.  Under this hypothesis, we show
that
$B_p$ is $r$-differentiable in the sense of Definition
\ref{differentiability_definition_barcode} on a generic (open and dense) subset of~$\manifoldm$. The intuition behind these
results is that, whenever the filter functions $F({\param{}'})$ are all
ordering equivalent in a neighborhood of $\param{}$, we can pick a
barcode template that is consistent across all filter functions~$F({\param{}'})$ in this neighborhood (by Propositions \ref{proposition_existence_pairing}
and \ref{proposition_pairing_same_equivalence_class}) and the Equation
\eqref{equation_pairing_definition} then behaves like a local
coordinate system for~$B$ at~$\param{}$.

Here is our local result:
\begin{theorem}[Local discrete smoothness]
\label{theorem_PH_differentiable_local}
Let $\param\in\paramspace$. Suppose the parametrization $F:\manifoldm{} \rightarrow \mathbb{R}^K$ is of class~$C^r$ ($r\geq 0$) on some open neighborhood~$U$ of~$\param$, and that $F(\param') \sim F(\param)$ for all $\param'\in U$. Then, $B_p$ is $r$-differentiable at~$\param$.
\end{theorem}

\begin{proof}
Note that, as an open set, $U$ is an open submanifold of~$\paramspace$
of same dimension. By Proposition \ref{proposition_existence_pairing}, we can pick a barcode template $(P_p,U_p)$ for~$F(\param)$. By Proposition \ref{proposition_pairing_same_equivalence_class}, this barcode template is consistent for all $F(\param{}')$ where $\param'\in U$. Therefore, we can locally write:
$$\forall \param{}'\in U, \, B_p(\param{}')=\big\{(F(\param{}')(\simplex{}),F(\param{}')(\simplex{}')) \big\}_{(\simplex{},\simplex{}')\in P_p} \cup \big\{(F(\param{}')(\simplex{}),+\infty) \big\}_{\simplex{}\in U_p}\cup \Delta^\infty$$
which is a local coordinate system for $B_p$ at $\param$. This local coordinate system is $C^r$ because
$\parametrization{}$ itself is $C^r$ over~$U$. As a result, $B_p$
 is $r$-differentiable at $\param$, by
Proposition~\ref{proposition_relating_differentiability_barcode_ordered_barcode}.
\end{proof}

%
\begin{corollary}
\label{corollary_distinct_values_implies_smooth}
Let $\param\in\paramspace$. Suppose that the parametrization~$\parametrization{}$ is of class~$C^r$ ($r\geqslant 0$) on some open neighborhood of~$\param$, and that the filter function~$F(\param)$ is injective. Then, $B_p$ is $r$-differentiable at $\param{}$.
\end{corollary}
\begin{proof}
For such a $\param{}$, all the quantities $F(\param{})(\simplex{})-F(\param{})(\simplex{}')$ for $\simplex{}\neq \simplex{}'\in K$ are either strictly positive or strictly negative. Therefore, by continuity they keep their sign in an open neighborhood of $\param{}$, over which all filter functions are thus ordering equivalent. The result follows then from Theorem~\ref{theorem_PH_differentiable_local}.
\end{proof}

Here is our global result:
\begin{theorem}[Global discrete smoothness]
\label{theorem_PH_differentiable_global}
Suppose the parametrization $F:\manifoldm{} \rightarrow \mathbb{R}^K$ is continuous over~$\paramspace$ and of class~$C^r$ ($r\geq 0$) on some open subset~$U$ of~$\paramspace$. Then, $B_p$ is $r$-differentiable on  the set $U\cap \tilde\paramspace$, where
\begin{equation}
    \label{equation_locally_constant}
    \tilde\manifoldm{}:= \big\{\param\in \manifoldm{} \mid  \exists\text{ open neighborhood $U_\param$ of $\param$ s.t. }F(\param') \sim F(\param{}) \text{ for all } \param'\in U_\param\},
\end{equation}
%
which is generic (i.e. open and dense) in~$\paramspace$. In particular, if $F$ is $C^r$ on some generic subset of~$\paramspace$ in the first place, then so is~$B_p$ (on some possibly smaller generic subset).
\end{theorem}
\begin{proof}
  Observe that $\tilde\paramspace$ is open in~$\paramspace$. As a consequence, for every $\param\in U\cap \tilde \paramspace$ there is some open neighborhood on which $F$ is $C^r$ and all the filter functions~$F(\param')$ are ordering equivalent, which by Theorem~\ref{theorem_PH_differentiable_local} implies that~$B_p$ is $r$-differentiable at~$\param$. Thus, all that remains to be shown is that~$\tilde\paramspace$ is dense
in~$\paramspace$, which is the subject of Lemma~\ref{lemma_tildeM_dense} below.
\end{proof}

\begin{lemma}
\label{lemma_tildeM_dense}
If a parametrization~$\parametrization:\manifoldm{} \rightarrow \mathbb{R}^K$ is continuous, then the set~$\tilde\manifoldm{}$ (as defined in Eq.~\eqref{equation_locally_constant}) is dense in~$\manifoldm$.
\end{lemma}
\begin{proof}
Let $h:\paramspace{}\rightarrow
\mathbb{R}$ be a continuous function. Consider the boundary of the
zero-level set~$h^{-1}(0)$:
  \[\partial h^{-1}(0) = \overbar{h^{-1}(0)}\ \setminus\ (h^{-1}(0))^{\mathrm{o}}.\]
Since $h$ is continuous, $h^{-1}(0)$ is closed in~$\paramspace$, therefore $\partial h^{-1}(0)$ is closed with empty interior, i.e. its complement $(\partial h^{-1}(0))^c$ in $\paramspace{}$ is open and dense.

Consider now the case of function $h_{\simplex{},\simplex{}'}: \param{}\in \paramspace{} \mapsto F(\param{})(\simplex{})-F(\param{})(\simplex{}')\in\mathbb{R}$ for some fixed simplices $\simplex{}\neq \simplex{}'$ of $K$. The map $h_{\simplex{},\simplex{}'}$ is continuous by continuity of the parametrization $\parametrization{}$, therefore the previous paragraph implies that $(\partial h_{\simplex{},\simplex{}'}^{-1}(0))^c$ is generic in $\paramspace{}$. Hence, the finite intersection 
\[\hat{\paramspace}:= \bigcap_{\simplex \neq \simplex{}'\in K} (\partial h_{\simplex{},\simplex{}'}^{-1}(0))^c\]
is also generic in~$\paramspace{}$. We now show that~$\hat{\paramspace}$ is a subspace of~$\tilde{\paramspace}$. 

Let $\param{}\in \hat{\paramspace}$ and $\simplex{}\neq \simplex{}'\in K$. If
$h_{\simplex, \simplex'}(\param)>0$, then by continuity we have
$h_{\simplex, \simplex'}>0$ over some open neighborhood $V_{\simplex{},\simplex{}'}$ of $\param{}$. Similarly if
$h_{\simplex, \simplex'}(\param)<0$. And if $h_{\simplex{},\simplex{}'}(\param)=0$, then, since $\param{}\in \hat{\paramspace{}}$, $\param{}$ lies in the interior of the level set $h_{\simplex{},\simplex{}'}^{-1}(0)$, and therefore there is also an open neighborhood $V_{\simplex{},\simplex{}'}$ of $\param{}$ over which $h_{\simplex{},\simplex{}'}=0$. Let $V$ be the finite intersection $\bigcap_{\simplex \neq \simplex{}'\in K} V_{\simplex{},\simplex{}'}$, which is open and non-empty in $\paramspace{}$. For every $\simplex{}\neq \simplex{}'\in K$, the sign $F(\param{}')(\simplex{})-F(\param{}')(\simplex{}')$
is constant over all $\param{}'\in V$, where by sign we really distinguish between three possibilities: negative, positive, null. Therefore, the pre-order on the simplices of $K$ induced by $F(\param{}')$ is constant over the $\param{}'\in V$. In other words, all the $F(\param{}')$ are ordering equivalent. Therefore, $\param\in \tilde{\paramspace}$. Since this is true for any $\param\in \hat{\paramspace}$, we conclude that $\hat{\paramspace}\subseteq \tilde{\paramspace}$, and so the latter is also dense in~$\paramspace$.
\end{proof}

\begin{example}[Height functions again]
\label{ex:height_func_diff1}
  Let us reconsider the scenario of Example~\ref{example_parametrization}. The parametrization $\parametrization$ of height filters is~$C^0$ on the entire sphere~$\mathbb{S}^{d-1}$. Moreover,~$\parametrization$ is smooth at every direction~$\param\in \mathbb{S}^{d-1}$ that is not orthogonal to some difference~$\vertex-\vertex'$ of vertices $\vertex\neq  \vertex'\in K_0$ in $\mathbb{R}^d$. The set~$U$ of such directions is generic in~$\mathbb{S}^{d-1}$, therefore~$B_p$ is $\infty$-differentiable over the generic subset~$U\cap \tilde{\mathbb{S}}^{d-1}$ by Theorem~\ref{theorem_PH_differentiable_global}, with~$\tilde{\mathbb{S}}^{d-1}$ defined as in Eq.~\eqref{equation_locally_constant}. In fact, we have $U\cap \tilde{\mathbb{S}}^{d-1} = U$ in this case. Indeed, for any direction $\param\in U$, the values of the height function $h_\param$  at the vertices of $K$ are pairwise distinct, and by continuity this remains true in a neighborhood of $\param$. The pre-order on the simplices of $K$ induced by the height function is then constant over this neighborhood. 
\end{example}

In Theorems~\ref{theorem_PH_differentiable_local}
and~\ref{theorem_PH_differentiable_global}, one cannot avoid the
condition that filter functions are locally ordering equivalent. Indeed, in the next examples, we highlight that there is generally no hope for the barcode valued map~$B_p$ to be differentiable everywhere, even if the
parametrization~$\parametrization{}$ is. This is because, essentially,
the time of appearance of a simplex is a maximum of smooth functions,
which can be non-smooth at a point where two functions achieve the
maximum. The condition that the induced pre-order is locally constant
around~$\param{}$ is only a sufficient condition though, because a
maximum of two smooth functions can still be smooth at a point where
the maximum is attained by the two functions. We provide a second
example to illustrate this fact.

\begin{example}[Singular parameter]
\label{non_differentiable_example}
\normalfont
Let us consider the following geometric simplicial complex~$K$ on the real line:\\

\begin{tikzpicture}
[scale=1, vertices/.style={draw, fill=black, circle, inner sep=0.0001pt}]
\usetikzlibrary{arrows.meta}
\node[vertices,label=below:{$a$ at 0},ultra thick] at (-2,0){};
\node[vertices,label=below:{$b$ at 1}, ultra thick] at (2,0){};
\draw (-2,0)--(2,0);
\end{tikzpicture}

That is, $K$ has vertices $K_0=\{a,b\}$ with respective coordinates $\{0,1\}$, and edges $K_1=\{ ab\}$. Consider the parametrization that filters the complex according to the squared euclidean distance to a point, i.e $\parametrization{}:\param\in \mathbb{R} \mapsto( \simplex\in K \mapsto \max_{x\in\simplex} (x-\param)^2)$. The map~$B_0$ is then essentially a real function that tracks the squared euclidean distance of the vertex closest to $\param$, specifically:
\[ B_0(\param)= \{(\min(\param^2,(1-\param)^2),+\infty)\} \cup \Delta^\infty. \]
Hence, $B_0$ is not differentiable at $\param=\frac{1}{2}$ since $\frac{1}{2}$ is a singular point of the map $\param \mapsto \min(\param^2,(1-\param)^2)$. Meanwhile, for $ \param < \frac{1}{2} $, we have $F(\param)(a)< F(\param)(b)$, whereas whenever $\param> \frac{1}{2}$, we have $F(\param)(a)> F(\param)(b)$. In particular, the pre-order induced by the filter functions $F(\param)$ is not constant around $\param=\frac{1}{2}$, and so $\frac{1}{2}\notin\tilde{\mathbb{R}}$.  
\end{example}

\begin{example}[Only sufficient condition]
\label{diffentiable_example_barcode_function_differentiable}
\normalfont
We remove the edge $ab$ from the geometric complex $K$ in the previous example, and we see the points $a$ and $b$ as lying on the $x$-axis of $\mathbb{R}^2$. Consider the parametrization of height filters $\parametrization{}:\param \in \mathbb{S}^1 \mapsto (\simplex\in K\mapsto \max_{x\in\simplex} \langle \param,x\rangle)$.  The map $B_p$ is then trivial for each degree~$p$ except $0$, where it writes as follows:
\[ B_0(\param)=\{(\langle \param,a\rangle,+\infty), (\langle \param,b\rangle,+\infty)\}\cup \Delta = \{(0,+\infty), (\langle \param,(1,0)\rangle ,+\infty)\}\cup \Delta^\infty. \]
We see that we have a valid local coordinate system given by the two smooth maps $\param \mapsto 0$ and $\param \mapsto \langle \param,(0,1)\rangle $, so the map~$B_0$ is $\infty$-differentiable everywhere on $\mathbb{S}^1$ by Proposition \ref{proposition_relating_differentiability_barcode_ordered_barcode}. Meanwhile, we have $F(\param)(a)<F(\param)(b)$ whenever $ \langle \param,(1,0)\rangle > 0 $, and $F(\param)(a)>F(\param)(b)$ whenever $ \langle \param,(1,0)\rangle < 0$, therefore the pre-order induced by the filter functions $F(\param)$ is not constant around $\param=(0,1)$ and $v=(0,-1)$, hence $(0,1),(0,-1)\notin\tilde{\mathbb{R}}$. 
\end{example}

\subsection{Differential of the barcode valued map}

Given a continuous parametrization $F:\manifoldm{}\rightarrow \mathbb{R}^K$ of class $C^1$ on some open set $U\subseteq \manifoldm{}$, Theorem \ref{theorem_PH_differentiable_global} guarantees that a barcode template, through Equation \eqref{equation_pairing_definition}, provides a $C^1$ local coordinate system for~$B_p$ around each point~$\param\in U\cap \tilde{\manifoldm}$. In turn, by Proposition~\ref{proposition_relating_differentiability_barcode_ordered_barcode}, any arbitrary ordering on the functions of this local coordinate system induces a $C^1$ local lift of~$B_p$. Hence we have the following formula for the corresponding differential:
\begin{proposition}
\label{proposition_derivatives_barcode_simplicial_complex}
Given~$\param\in U\cap \tilde{\manifoldm}$ and a barcode template $(P_p,U_p)$ of $F(\param{})$, for any choice of ordering $(\simplex_1, \simplex'_1), \cdots, (\simplex_m, \simplex'_m), \tau_1, \cdots, \tau_n$ of $(P_p,U_p)$, the map
\[ \tilde B_p :\param' \mapsto \left[ (F(\param')(\simplex_i), F(\param')(\simplex'_i))_{i=1}^m, (F(\param')(\simplex_j))_{j=1}^n \right] \]
is a local $C^1$ lift of~$B_p$ around~$\param$, and the corresponding differential for~$B_p$ at $\param$ is:
\[ d_{\param, \tilde B_p} B_p(.)=\left[ (d_\param F(\cdot)(\simplex_i), d_\param F(\cdot)(\simplex'_i))_{i=1}^m, (d_\param F(\cdot)(\tau_j))_{j=1}^n \right]. \]
\end{proposition}

\begin{remark}[Algorithm for computing derivatives]
\label{remark_algorithmic computation}
Suppose we are given a parametrization $F$ whose differential we can compute. Let $\param\in\paramspace$. If 
the barcode of $F(\param)$ is given to us, then the proof of Proposition~\ref{proposition_existence_pairing} provides an algorithm to build a barcode template~$(P_p, U_p)$ for $\parametrization(\param)$. If the barcode of $\parametrization(\param)$ is not given in the first place, then the matrix reduction  algorithm for computing persistence \citep{elz-tps-02,zomorodian2005computing} outputs both the barcode and a barcode template. In both scenarios, 
%
 Proposition~\ref{proposition_derivatives_barcode_simplicial_complex} gives a formula to compute a differential of $B_p$ at~$\param$ from the barcode template~$(P_p, U_p)$.
 The optimization pipelines mentioned in the introduction \citep{gabrielsson2018topology,chen2019topological,gameiro2016continuation,2019arXiv190510996H,poulenard2018topological} apply this strategy to compute  differentials.
\end{remark}

\subsection{Directional differentiability of the barcode valued map along strata}
\label{section_stratification}

In this section we
define directional derivatives for the barcode valued map~$B_p: \manifoldm\to Bar$ at points where it may not be differentiable in the sense of Definition~\ref{differentiability_definition_barcode}. For this we stratify the parameter space~$\paramspace$ in such a way that $B_p$ is differentiable on the top-dimensional strata, then we define its derivatives on lower-dimensional strata via directional lifts.  Intuitively, the strata in~$\paramspace$ are prescribed by the ordering equivalence classes in~$\R^K$, as we know from Theorem~\ref{theorem_PH_differentiable_local} that the pre-order on simplices plays a key role in the differentiability of~$B_p$.  

Formally, consider the stratification of~$\R^K$ formed by the collection $\Omega(\mathbb{R}^K)$ of ordering equivalence classes. This is a Whitney stratification, obtained by cutting $\mathbb{R}^K$ with the hyperplanes $\{f(\simplex)=f(\simplex')\}$ for varying simplices $\simplex\neq \simplex'\in K$. We look for stratifications of~$\paramspace$ that make the parametrization~$\parametrization$ weakly stratified (in the sense of Definition~\ref{definition_stratified_map}) and smooth on each stratum. 
Here are 
typical scenarios where such stratifications exist:
\begin{proposition}\label{semi-algebraic_stratif}
Let $\parametrization: \paramspace\to\R^K$ be a continuous parametrization. Suppose that, either
\begin{itemize}
    \item[(i)] $\paramspace$ is a semi-algebraic set in $\R^N$ and $\parametrization$ is a semi-algrebraic map, or 
    \item[(ii)] $\paramspace$ is a compact subanalytic set in a real analytic manifold and $\parametrization$ is a subanalytic map.
\end{itemize}
Then, 
there is a Whitney stratification of~$\paramspace$, made of semi-algebraic (resp. subanalytic) strata, such that $\parametrization$ is weakly stratified with $C^\infty$ restrictions to each stratum. 
\end{proposition}
\begin{proof}
  This is Section~I.1.7 of \cite{goresky1988stratified}, after observing that the stratification~$\Omega(\mathbb{R}^K)$ is made of semi-algebraic strata.
\end{proof}

\begin{example}
\label{stratification_sphere}
We consider the parametrization~$\parametrization$ of height filters on the sphere $\mathbb{S}^{d-1}$ from Example \ref{ex:height_func_diff1}. By Proposition~\ref{semi-algebraic_stratif}, there is a stratification of~$\mathbb{S}^{d-1}$ that makes~$\parametrization$ weakly stratified and $C^\infty$ on each stratum. To be more specific, such a stratification is obtained by taking the pre-images\footnote{This is called the {\em pull-back stratification}. In fact, for any  smooth map~$\parametrization:\paramspace\to\R^K$ that is transverse with respect to~$\Omega(\R^K)$ and to any stratification of~$\paramspace$ (e.g. the trivial one), the pull-back of~$\Omega(\R^K)$ via~$\parametrization$  makes the latter weakly stratified and $C^\infty$ on each stratum---see e.g. \cite[\textsection I.1.3]{goresky1988stratified}.} of the strata of~$\Omega(\R^K)$ via $\parametrization$. Figure~\ref{fig7} illustrates the result in the case $d=3$, where the obtained stratification of~$\mathbb{S}^2$ is made of an arrangement of great circles, each circle being the pre-image of a set $\{\parametrization(\param)(\vertex)=\parametrization(\param)(\vertex')\}$ for vertices $\vertex\neq \vertex'$.  
 %
\begin{figure} 
  \begin{subfigure}[b]{0.5\linewidth}
    \centering
    \includegraphics[width=0.75\linewidth]{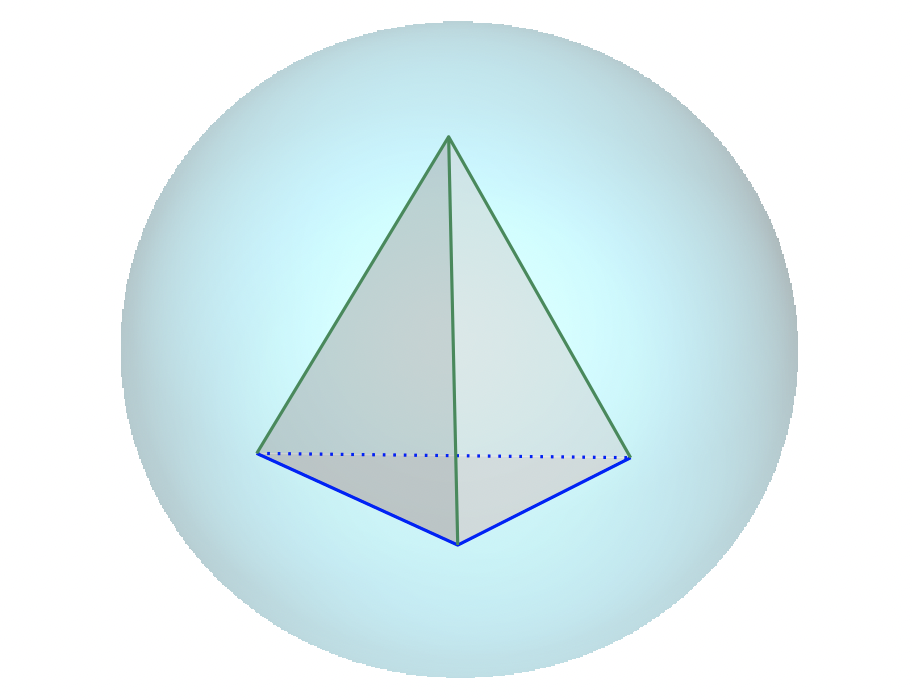} 
    \label{fig7:c} 
  \end{subfigure}
  \begin{subfigure}[b]{0.5\linewidth}
    \centering
    \includegraphics[width=0.75\linewidth]{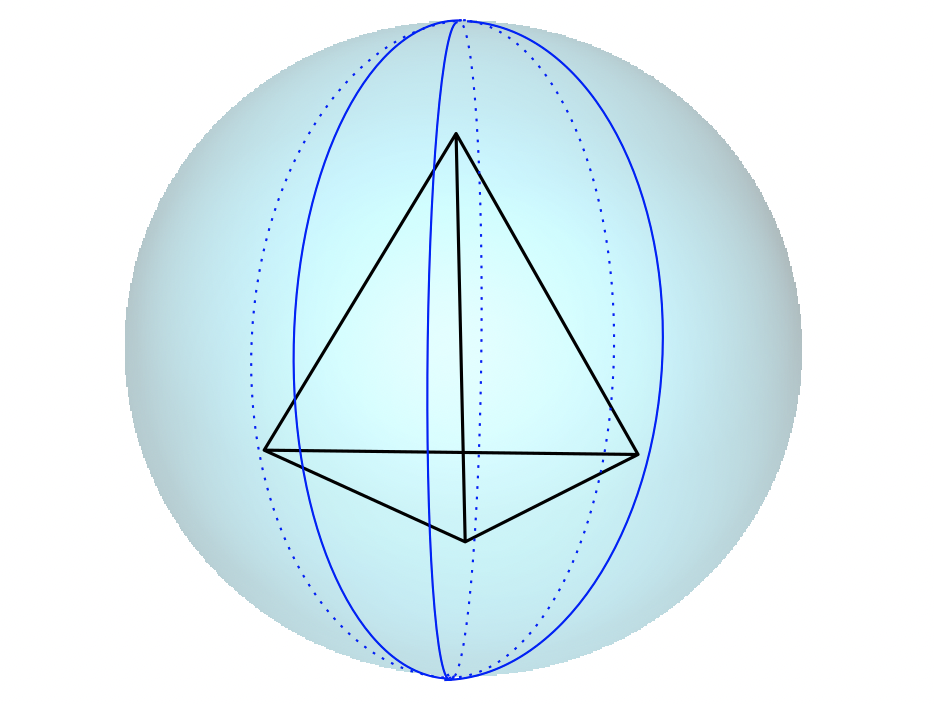} 
    \label{fig7:d} 
  \end{subfigure} 
   \begin{subfigure}[b]{0.5\linewidth}
    \centering
    \includegraphics[width=0.75\linewidth]{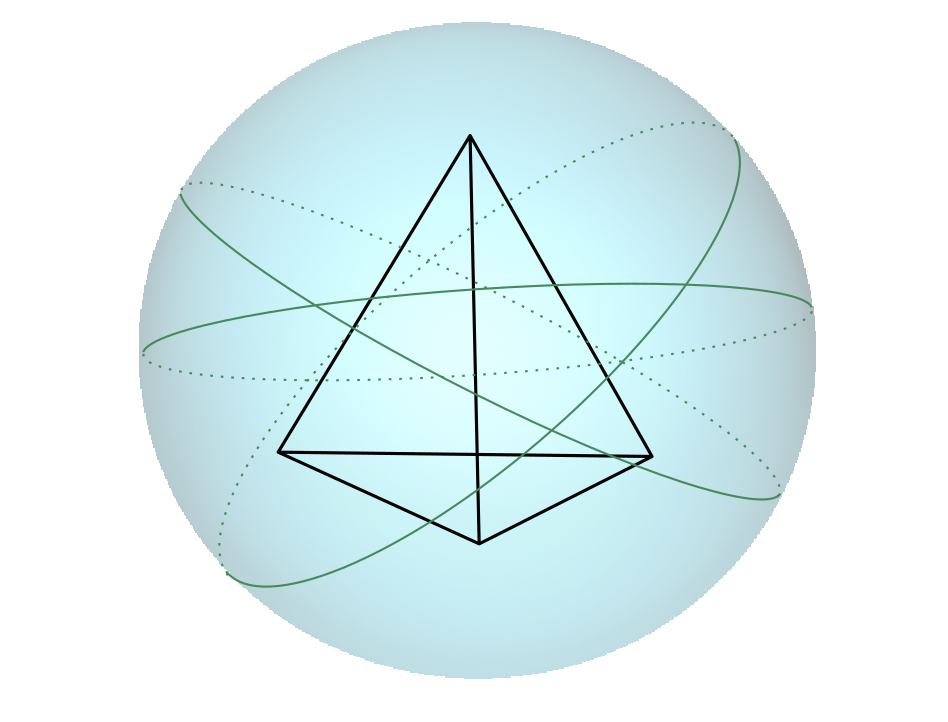} 
    \label{fig7:e} 
  \end{subfigure}
   \begin{subfigure}[b]{0.5\linewidth}
    \centering
    \includegraphics[width=0.75\linewidth]{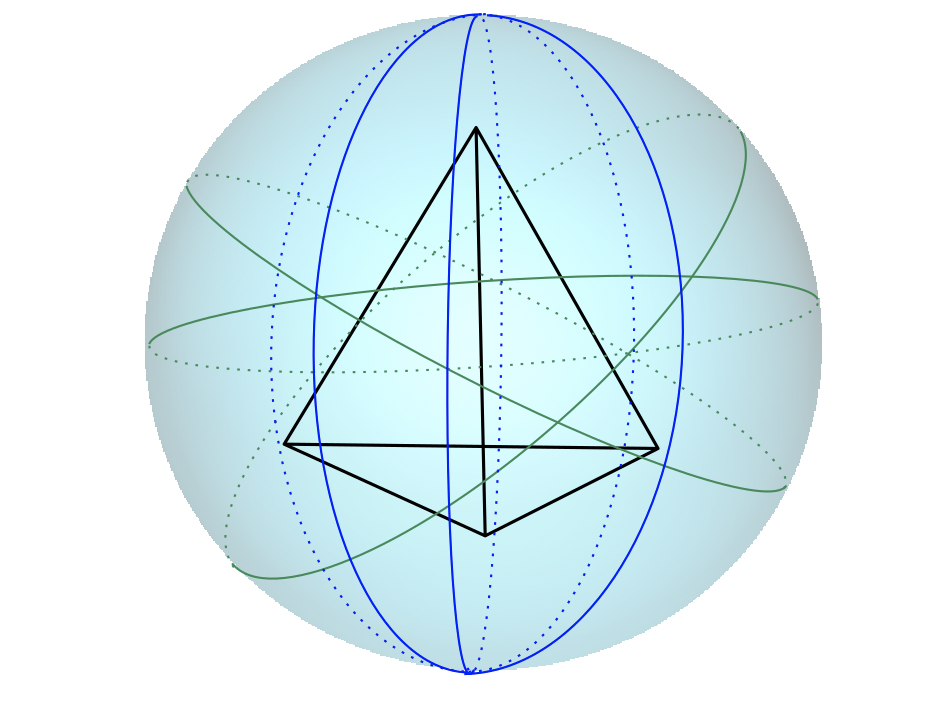} 
    \label{fig7:f} 
  \end{subfigure}
  \caption{The regular tetrahedron~$K$ embedded in~$\R^3$ (top left) induces the stratification of $\mathbb{S}^2$ for the parametrization of height filters (bottom right). The intersections of great circles are $0$-dimensional strata, the parts of the great circles that do not intersect with each other are $1$-dimensional strata, the rest of $\mathbb{S}^2$ forms the 2-dimensional strata. Each great circle corresponds to unit vectors that are orthogonal to a given edge of the simplicial complex. The edges joining the vertices in the base of the tetrahedron produce the blue circles (top right), while the edges joining the apex of the tetrahedron to the base face produce the green circles (bottom left). }
  \label{fig7} 
\end{figure}
\end{example}

Once a stratification~$\mathcal{S}_\paramspace$ of~$\paramspace$ is given, we can introduce a notion of derivative for~$B_p$ at $\param\in\manifoldm$ in the direction of an {\em incident} stratum~$\paramspace'$, i.e. a stratum whose closure in~$\paramspace$ contains~$\param$.
%
\begin{definition}
\normalfont
\label{definition_partial_derivative}
Let~$B:\paramspace \rightarrow Bar$ be a map defined on a stratified space~$(\paramspace{},\mathcal{S}_\paramspace)$. Let $\param{} \in \paramspace{}$, and let $\paramspace'\in\mathcal{S}_\paramspace$ be a stratum incident to~$\param{}$. The map $B$ is $r$-\textit{differentiable} at~$\param{}$ \textit{along} $\paramspace'$ if there is an open neighborhood $U$ of $\param{}$ in~$\paramspace$ and a $C^r$ map $\tilde{B} : U \rightarrow \mathbb{R}^{2m}\times \mathbb{R}^n$ for some integers $m,n$ such that $B=Q_{m,n}\circ \tilde{B}$ on~$U\cap \paramspace'$. The differential~$d_\param \tilde{B}$ is called a \textit{directional derivative} of~$B$ at~$\param$ along~$\manifoldm'$.
\end{definition}

This definition agrees with the notions of $r$-differentiability and derivatives introduced in Section~\ref{section_definition_differentiability_barcode_valued_maps} when $\paramspace'$ contains an open neighborhood around~$\param$, i.e. for $\param$ located in a top-dimensional stratum~$\paramspace'$. When $\param$ is located in some lower-dimensional stratum, it admits finitely many incident strata~$\paramspace'$ (possibly not top-dimensional), each one of which yields a specific directional derivative at~$\param$. The definition of each derivative involves a local $C^r$ lift~$\tilde B$ of~$B$ near~$\param$ in~$\paramspace'$. This lift is required to extend smoothly over an open neighborhood~$U$ in~$\paramspace$, to ensure that $\tilde B$ and its derivatives have well-defined limits  at~$\param$.
\begin{theorem}[Discrete smoothness along strata] \label{theorem_PH_partial_derivatives_stratification}
Let $r\in \mathbb{N}$ and $\parametrization{}:\paramspace{} \rightarrow \mathbb{R}^K$. Suppose $\mathcal{S}_\manifoldm{}$ is a Whitney stratification of~$\paramspace$ such that:
\begin{itemize}
    \item [\bf{(i)}] $\parametrization$ is a weakly stratified map with respect to $\mathcal{S}_\manifoldm$ and $\Omega(\R^K)$, and
    \item[\bf{(ii)}] the restriction of $\parametrization$ to each stratum of $\mathcal{S}_\manifoldm$ is $C^r$, and
    \item [\bf{(iii)}] for every $\param\in \manifoldm$ and every incident stratum $\manifoldm'\in \mathcal{S}_\manifoldm$, there is an open neighborhood $U$ of $\param$ in~$\manifoldm$ such that  $\parametrization_{|\manifoldm'\cap U}$ extends to a $C^r$ map $U \to \R^K$.
\end{itemize}
Then, at every $\param{} \in \paramspace{}$, the barcode valued map $B_p: \manifoldm{}  \rightarrow  Bar$ is $r$-differentiable along each stratum incident to~$\param$. In particular, $B_p$ is $r$-differentiable in the sense of Definition~\ref{differentiability_definition_barcode} inside each top-dimensional stratum.
\end{theorem}
\begin{proof}
Let $\param\in\paramspace$ and $\paramspace'$ a stratum incident to~$\param{}$. By~{\bf (i)}, combined with Propositions~\ref{proposition_existence_pairing} and~\ref{proposition_pairing_same_equivalence_class}, there exists a barcode template $(P_p,U_p)$ that is consistent across all $\parametrization(\param')$ for $\param'\in \paramspace'$. Therefore, for all $\param'\in \manifoldm'$:
\begin{equation}
    \label{equation_barcode_function_each_stratum_pairing}
     B_p(\param{}')=\big\{(F(\param{}')(\simplex{}),\parametrization{}(\param{}')(\simplex{}')) \big\}_{(\simplex{},\simplex{}')\in P_p} \cup \big\{(F(\param{}')(\simplex{}),+\infty) \big\}_{\simplex{}\in U_p}\cup \Delta^\infty,
\end{equation}
which by~{\bf(ii)} provides a $C^r$ local coordinate system for ${B_p}_{|\manifoldm'}$. Then by Proposition \ref{proposition_relating_differentiability_barcode_ordered_barcode}, there is a $C^r$ lift of ${B_p}_{|\manifoldm'}$, whose coordinate functions are of the form $\param'\mapsto \parametrization(\param')(\simplex)$. Using~{\bf(iii)}, we extend each coordinate function of this lift (hence the lift itself) to an open neighborhood $U$ of $\param$ in $\manifoldm$.
\end{proof}
Combining Proposition~\ref{semi-algebraic_stratif} with Theorem~\ref{theorem_PH_partial_derivatives_stratification} yields the following:
%
\begin{corollary} \label{cor:sem-alg_directional}
Under the hypotheses of Proposition~\ref{semi-algebraic_stratif},
%
there is a Whitney stratification of~$\paramspace$, made of semi-algebraic (resp. subanalytic) strata, such that~$B_p$ is $\infty$-differentiable on the top-dimensional strata (whose union is generic in~$\paramspace$). If furthermore~$F$ is globally~$C^r$, then $B_p$ is everywhere $r$-differentiable  along incident strata.
\end{corollary}


\begin{example}
\label{example_partial_derivatives_distance_function}
\normalfont
Consider again the setup of Example \ref{non_differentiable_example}.
We stratify $\mathbb{R}$ by the point $\{\frac{1}{2}\}$ and the half-lines $(-\infty;\frac{1}{2})$ and $(\frac{1}{2}; +\infty)$. The parametrization $\parametrization{}$ is $C^\infty$ and sends strata into strata, therefore by Theorem~\ref{theorem_PH_partial_derivatives_stratification} the barcode valued map $B_0$ admits directional derivatives everywhere on $\R$. More precisely, recall that we have a lift $\tilde{B_0}:\param \mapsto \min(\param^2,(1-\param)^2)$, which is smooth in the top-dimensional strata, while at $\param=\frac{1}{2}$ it admits directional derivatives along the two half-lines, whose values are $1$ and $-1$ respectively and thus do not agree.
\end{example}
\begin{example}
\label{example_partial_derivatives_height}
Consider again the stratification $\mathcal{S}_{\mathbb{S}^{d-1}}$ by the great circles of the parameter space $\mathbb{S}^{d-1}$ associated to the parametrization of height filters (Example~\ref{stratification_sphere}). By Corollary~\ref{cor:sem-alg_directional}, we know that there exists a refinement~$\mathcal{S}'_{\mathbb{S}^{d-1}}$ of~$\mathcal{S}_{\mathbb{S}^{d-1}}$ such that   $B_p$ admits directional derivatives along incident strata of~$\mathcal{S}'_{\mathbb{S}^{d-1}}$ at every point $\param\in\mathbb{S}^{d-1}$. In fact, we can even take $\mathcal{S}'_{\mathbb{S}^{d-1}}$ to be $\mathcal{S}_{\mathbb{S}^{d-1}}$ itself. Indeed, all the directions in a given stratum $\manifoldm'\in \mathcal{S}_{\mathbb{S}^{d-1}}$ induce the same pre-order on the simplices of $K$, therefore
\begin{itemize}
    \item the restriction $\parametrization_{|\manifoldm'}$ is valued in a stratum of $\Omega(\R^K)$, and
    \item for every simplex $\sigma\in K$, there is a vertex $\bar \vertex(\simplex)$ such that $\parametrization_{|\manifoldm'}(.)(\simplex)=\langle . , \bar \vertex(\simplex{})\rangle$.
\end{itemize}
Consequently, the assumptions of Theorem~\ref{theorem_PH_partial_derivatives_stratification} hold, and the barcode valued map~$B_p$ admits directional derivatives along incident strata of~$\mathcal{S}_{\mathbb{S}^{d-1}}$ at every point~$\param\in\mathbb{S}^{d-1}$. 
\end{example}

\subsection{The barcode valued map as a permutation map}

In this section, we work out a global lift of the barcode valued map, which restricts nicely to each stratum of a stratification of~$\paramspace$. To do so, we first focus on the map~$\Dgm$ which, given a filter function~$f\in \R^K$ on a fixed simplicial complex~$K$ of dimension~$d$, returns the vector of all its barcodes $(\Dgm_p(f))_{p=0}^d$. We observe that~$\Dgm$ admits a global Euclidean lift, and furthermore, that this lift is essentially a permutation map on each stratum of $\Omega(\R^K)$. Throughout, we fix an ordering of the simplices of $K$, so that the canonical basis of~$\R^K$ turns into a basis of $\R^{\#K}$, and we let  $\phi: \R^K \to \R^{\#K}$ be the corresponding isomorphism.

\begin{proposition}
\label{proposition_global_lift_permutation}
There exist integers $m_p,n_p$ for $0\leqslant p\leqslant d$ such that $\sum_{p=0}^d (2m_p+n_p) =\#K$, and a map $\Perm:\R^K \rightarrow \prod_{p=0}^d \R^{2m_p}\times \R^{n_p}\cong \R^{\#K}$ whose restriction $\Perm_{|S}$ to each ordering equivalence class $S\in \Omega(\R^K)$ is a permutation matrix, and such that the following diagram commutes:\footnote{Strictly speaking, like $\Dgm$, this diagram applies only to the set of filter functions in $\R^K$.} 
\begin{equation}
\label{lift_global}
\begin{gathered}
\begin{tikzpicture}

\node[] (a) at (-2,0) {$\R^{\#K}$};

\node[] (b) at (2.8,0) {$\prod_{p=0}^d \mathbb{R}^{2m_p}\times \mathbb{R}^{n_p}$};

\node[] (c) at (-2,-2) {$\R^K$};

\node[] (d) at (2,-2) {$Bar^{d+1}$};

\draw[->] (a)--(b);
\draw[->] (c)--(d);
\draw[->] (c)--(a);
\draw[->] (2.0,-0.5)--(d);

\node[label=right:{$Q:= \prod_{p=0}^d Q_{m_p,n_p}$}] (q) at (2,-1) {};
\node[label=above:{$\Dgm$}] (Dgm) at (-0.2,-2.1) {};
\node[label=left:{$\phi$}] (phi) at (-2,-1) {};
\node[label=above:{$\Perm$}] (P) at (-0.2,0) {};

\end{tikzpicture}
\end{gathered}
\end{equation}
\end{proposition}
For simplicity, from now on we identify $f\in \R^K$ with its image in $\R^{\#K}$ without explicitly mentioning the map $\phi$.
\begin{proof}
Given a filter function $f\in \R^K$, we define a total barcode template $(P,U)$ for $f$ to be the data of $d+1$ barcode templates $(P_p,U_p)$ for $f$ in each homology degree, such that each simplex of $K$ appears exactly once, in a unique $P_p$ or $U_p$. We further require that the pairs $(\simplex,\simplex')$ appearing in $P_p$ consist of a $p$-dimensional simplex $\simplex$ and a ($p+1$)-dimensional simplex $\simplex'$, while the unpaired simplices appearing in $U_p$ must be $p$-dimensional. A simplex $\simplex$ is then labelled {\em positive} if it appears as the first component of a pair in some $P_p$ or $U_p$, and {\em negative} otherwise. 

Note that total barcode templates always exist, by an argument similar to (yet somewhat more involved than) the one used in the proof of Proposition~\ref{proposition_existence_pairing}. Alternatively, note that applying the matrix reduction algorithm for computing persistence \citep{elz-tps-02,zomorodian2005computing} to the sublevel-sets filtration of $f$ produces a total barcode template. By Proposition~\ref{proposition_pairing_same_equivalence_class}, total barcode templates are invariant under ordering equivalences. We therefore fix a unique total barcode template $(P(S),U(S))$ per ordering equivalence class $S\in \Omega(\R^K)$ (there are only finitely many such classes), and we denote by $m_p(S):=\#P_p(S)$, $n_p(S):=\#U_p(S)$ their sizes in each homology degree~$p$. 

Since the barcode templates $(P(S),U(S))$ are total, we have $\sum_{p=0}^d (2m_p(S)+n_p(S))=\#K$. Besides, since the number of infinite intervals in the barcode of a filter function is given by the Betti numbers of the simplicial complex $K$, an easy induction on the homology degree shows that the number of positive (resp. negative) simplices in each homology degree is independent of the choice of filter function and of total barcode template. Therefore, the integers $m_p(S),n_p(S)$ do not depend on the stratum $S$. 

For each stratum $S\in \Omega(\R^K)$ and homology degree $p$, we pick arbitrary orderings  $(\simplex_{k,S},\simplex'_{k,S})_{k=0}^{m_p}$  of $P_p(S)$ and $(\tau_{k,S})_{k=0}^{n_p}$ of $U_p(S)$. Any filter function $f\in S$ admits $(P(S),U(S))$ as total barcode template, therefore we get that $\Dgm_p(f)=Q_{m_p,n_p}((f(\simplex_{k,S}),f(\simplex'_{k,S}))_{k=0}^{m_p}, (f(\tau_{k,S}))_{k=0}^{n_p})$ in every homology degree $p$. We simply set $\Perm(f):=[(f(\simplex_{k,S}),f(\simplex'_{k,S}))_{k=0}^{m_p}, (f(\tau_{k,S}))_{k=0}^{n_p}]_{p=0}^d\in \prod_{p=0}^d \R^{2m_p}\times \R^{n_p}$, which ensures the commutativity of~\eqref{lift_global}. Since each simplex of~$K$ appears exactly once in $(P(S), U(S))$, the vector $\Perm(f)$ is a re-ordering of the coordinates of~$f$ (i.e. of its values on the simplices) and therefore $\Perm_{|S}$ is a permutation matrix.
\end{proof}

We now turn to the parametrized barcode valued map
\[B: \param \in \manifoldm \xrightarrow[\parametrization{}]{} \parametrization{}(\param)\in \R^K \xrightarrow[\Dgm]{} [\Dgm_p(F(\param))]_{p=0}^d\in Bar^{d+1}\]
determined by a parametrization $\parametrization{}:\paramspace \rightarrow \R^K$ of filter functions. We show that if~$\manifoldm$ admits a Whitney stratification~$\mathcal{S}_\manifoldm$ satisfying the assumptions of Theorem~\ref{theorem_PH_partial_derivatives_stratification}, then~$B$ admits a global lift~$\tilde{B}$ that acts as a permutation of $\parametrization{}$-values on each stratum.

\begin{corollary}
\label{corollary_global_lift_parametrized}
Using the same notations as in Proposition~\ref{proposition_global_lift_permutation}, the map
\[\tilde{B}:\param \in \paramspace \longmapsto \Perm(\parametrization(\param))\in \prod_{p=0}^d \R^{2m_p}\times \R^{n_p}\]
is a global lift of $B$, i.e $Q\circ \tilde{B}=B$ everywhere on $\manifoldm$. If moreover~$\manifoldm$ admits a Whitney stratification~$\mathcal{S}_\manifoldm$ satisfying the assumptions of Theorem~\ref{theorem_PH_partial_derivatives_stratification}, then $\tilde{B}=\Perm_{\manifoldm'}\circ F$ for some permutation matrix $\Perm_{\manifoldm'}$ over each stratum $\manifoldm'\in \mathcal{S}_{\manifoldm}$. Consequently, $B$ is $r$-differentiable along incident strata everywhere on $\manifoldm$, with directional derivatives given by the ones of $\tilde{B}$.
\end{corollary}
The last part of the statement expresses the fact that directional derivatives of~$B$ are simply given by permuting the directional derivatives of the coordinate functions of~$\parametrization$. 
\begin{proof}
The first part of the statement is a straight consequence of Proposition~\ref{proposition_global_lift_permutation}. Let~$\mathcal{S}_\manifoldm$ be a stratification satisfying the assumptions of Theorem~\ref{theorem_PH_partial_derivatives_stratification}. As $\parametrization$ is weakly stratified with respect to $\mathcal{S}_\manifoldm$ and $\Omega(\R^K)$, it sends strata into strata and therefore by Proposition~\ref{proposition_global_lift_permutation} we have $\tilde{B}=\Perm_{\manifoldm'}\circ \parametrization$ for some permutation matrix $\Perm_{\manifoldm'}$ over each stratum $\manifoldm'\in \mathcal{S}_{\manifoldm}$. Then, since~$\parametrization$ admits local smooth extensions over each stratum~$\manifoldm'$ of~$\mathcal{S}_\manifoldm$, so do its coordinate functions and in turn so does~$\tilde{B}=\Perm_{\manifoldm'}\circ \parametrization$. These local extensions of~$\tilde{B}$ yield directional derivatives for~$B$ along incident strata.
\end{proof}

\begin{remark}
\label{remark_definable_lift_for_definable_parametrization}
Recall that the map~$\Perm$ is a linear map when restricted to the strata of~$\Omega(\R^K)$, which are simply polyhedra in~$\R^K$. Therefore, if~$\manifoldm$ is a semi-algebraic set (resp. subanalytic set or definable set in an o-minimal structure) and~$\parametrization$ is a semi-algebraic (resp. subanalytic or definable) map, then the global lift~$\tilde{B}=\Perm\circ \parametrization$ of Corollary~\ref{corollary_global_lift_parametrized} is itself a semi-algebraic (resp. subanalytic or definable) map. Thus, we recover Proposition~3.2 and Corollary~3.3 of~\cite{carriere2020note}. Meanwhile, the differentiability of~$\tilde{B}$ on top-dimensional strata (as per Corollary~\ref{cor:sem-alg_directional}) recovers their Proposition~3.4. 
\end{remark}

We conclude this section with a side result whose proof (deferred to the appendix~\ref{appendix_proof_coercivity}) relies on Proposition~\ref{proposition_global_lift_permutation}.  This result states that $\Dgm$ is locally an isometry on top-dimensional strata of $\Omega(\R^K)$. It involves the distance~$d_0(f)$  of any filter function~$f\in\R^K$ to the union of strata of $\Omega(\R^K)$ of codimension at least~$1$:
%
\[
d_0(f) = \frac{1}{2}\min_{\simplex\neq \simplex'} |f(\simplex)-f(\simplex')|.
\]
\begin{proposition}
\label{proposition_local_coercivity}
Let $f,g\in \mathbb{R}^{K}$ be two filter functions that are located in the closure of a common top-dimensional stratum $S\in \Omega(\R^K)$. Then:
\begin{equation}
    \label{equation_coercivity}
   \max_{0\leqslant p \leqslant d }\db(\Dgm_p(f),\Dgm_p(g))\geqslant \min(\|f-g\|_\infty, \max(d_0(f),d_0(g))).
\end{equation}
In particular, for any filter function~$f\in\R^K$ located in a top-dimensional stratum, the map~$\Dgm$ is a local isometry in a closed ball of radius $d_0(f)$ around~$f$, specifically:
\begin{align}
 \label{equation_first_isometry}
    \forall g \in \R^K,\ \|f-g\|_\infty \leq d_0(f) \Longrightarrow  \\\max_{0\leqslant p \leqslant d }\db(\Dgm_p(f),\Dgm_p(g)) &= \|f-g\|_\infty \nonumber
\end{align}
\begin{align}
\label{equation_second_isometry}
    \forall g,h \in \R^K,\ \max (\|f-g\|_\infty, \|f-h\|_\infty) \leq \frac{d_0(f)}{3} \Longrightarrow \\ \max_{0\leqslant p \leqslant d }\db(\Dgm_p(g),\Dgm_p(h))&=  \|g-h\|_\infty. \nonumber
\end{align}
\end{proposition}


\section{Application to common simplicial filtrations}
\label{section_discrete_smoothness_subsection_examples}

In this section we leverage Theorems~\ref{theorem_PH_differentiable_local} and~\ref{theorem_PH_differentiable_global} in the case of a few important classes of parametrizations of filter functions on a simplicial complex~$K$ of dimension~$d$. In each case, we derive a characterization of the parameter values where~$B_p$ is differentiable, and whenever possible we provide an explicit differential of~$B_p$ using Proposition~\ref{proposition_derivatives_barcode_simplicial_complex}. In the following we fix a homology degree~$0\leqslant p \leqslant d$.

\subsection{Lower star filtrations}
\label{section_discrete_smoothness_subsection_examples_ex1}

Parametrizations of lower star filtrations are involved in most practical scenarios \citep{ gabrielsson2018topology, chen2019topological,1905.12200, 2019arXiv190510996H,poulenard2018topological}, here we provide a common analysis of their differentiability.

\begin{definition}
\normalfont
\label{lower_filt}
Given a function $f:K_0\rightarrow \mathbb{R}$ defined on the vertices of $K$, we extend it to each simplex $\simplex{}$ of $K$ by its highest value on the vertices of $\simplex{}$. The sub-level sets of this function together form the
\textit{lower-star filtration} of $K$ induced by~$f$.
\end{definition}

One interest of lower-star filtrations is that any parametrization $\manifoldm\to\R^{K_0}$ on the vertex set of~$K$ induces a valid parametrization $\manifoldm\to\R^K$ on $K$ itself.
Sufficient conditions for the differentiablity of such parametrizations are easy to work out thanks to the following observation:
\begin{proposition}
\label{proposition_lower_star_nice_filtration}
Let $\parametrization_0:\paramspace{} \rightarrow \mathbb{R}^{K_0}$ be a~$C^r$ parametrization of filter functions on the vertices of~$K$. Then, the induced parametrization $\parametrization:\paramspace{} \rightarrow \mathbb{R}^K$ is $C^r$ at each $\param{} \notin \Sing(\parametrization_0)$, where~$\Sing(\parametrization_0)$ is the boundary of the set:
\[\{\param{} \in \paramspace{}, \ \exists (v,v')\in K_0, \parametrization_0(\param{})(v)=\parametrization_0(\param{})(v')\}. \]
Specifically, for every $\param\notin \Sing(\parametrization_0)$, letting $$\bar v: \simplex\in K \mapsto \argmax_{v \ \text{vertex in} \ \simplex{}}\parametrization_0(\param{})(v) \in K_0$$ by breaking ties wherever necessary, there is an open neighborhood~$U$ of~$\param$ such that $\parametrization(\param')(\simplex) = \parametrization_0(\param')(\bar v(\simplex))$ for every $\param'\in U$ and $\simplex\in K$, from which follows that $\parametrization$ is~$C^r$ at~$\param$.
\end{proposition}
\begin{proof}
The continuity of $\parametrization{}$ comes from the continuity of $\parametrization{}_0$ and of the $\max$ function. If $\param{} \in \paramspace{} \setminus \Sing(\parametrization_0)$, then the pre-order on~$K_0$ induced by $\parametrization_0(.)$ is constant in an open neighborhood~$U$ of~$\param$. 
We want to check that $\parametrization{}$ is $C^r$ at $\param{}$, i.e. that all maps $\param{}' \mapsto \parametrization(\param{}')(\simplex{})$ are $C^r$ at $\param$, for a fixed simplex $\simplex{} \in K$. For $\simplex$ a vertex of~$K$, this is true by assumption because $\parametrization{}(.)(\simplex)=\parametrization{}_0(.)(\simplex)$. For an arbitrary simplex $\simplex{}$, $\parametrization{}(.)(\simplex{})= \max_{v \ \text{vertex in} \ \simplex{}}\parametrization_0(.)(v)$. Since the pre-order induced on $K_0$ by $\parametrization_0$ is constant over $U$, the maximum above is attained at vertex $\bar v(\simplex{})$, and this fact holds for all $\param'$ in~$U$. Thus, $\parametrization(.)(\simplex{})_{|U}=\parametrization_0(.)(\bar v(\simplex{}))_{|U}$, which allows us to conclude.
\end{proof}
\begin{remark}
Recall that~$\Sing(\parametrization_0)$ is by definition the boundary of~$\{\param{} \in \paramspace{}, \ \exists (v,v')\in K_0, \parametrization_0(\param{})(v)=\parametrization_0(\param{})(v')\}$, whose complement may not be generic (in fact it may even be empty, e.g. when $F_0=0$). This shows the interest of working with locally constant pre-orders on vertices, and not just with locally injective parametrizations as in the works of \cite{gabrielsson2018topology,chen2019topological,1905.12200,2019arXiv190510996H,poulenard2018topological}.
\end{remark}
Defining~$\Sing(\parametrization_0)$ and~$\bar v$ as in Proposition~\ref{proposition_lower_star_nice_filtration}, and combining this result with  Proposition~\ref{proposition_derivatives_barcode_simplicial_complex}, we deduce the following result on the differentiability of~$B_p$, which only relies on the differentiability of~$\parametrization_0$: 
\begin{corollary}
\label{cor_lower_star_barcode_differentiability}
For any $C^r$ parametrization $\parametrization_0:\paramspace \rightarrow\mathbb{R}^{K_0}$ on the vertices of $K$, the induced barcode valued map~$B_p: \param{}\in \manifoldm{}\mapsto \Dgm_p(\parametrization(\param{}))\in Bar$ is $r$-differentiable outside $\Sing(\parametrization_0)$. Moreover, at $\param{}\in \paramspace{}\setminus \Sing(\parametrization_0)$, for any barcode template $(P_p,U_p)$ of $\parametrization(\param{})$ and any choice of ordering $(\simplex_1, \simplex'_1), \cdots, (\simplex_m, \simplex'_m)$, $\tau_1, \cdots, \tau_n$ of $(P_p,U_p)$, the map~$\tilde B_p:\paramspace \rightarrow \R^m\times \R^n$ defined by:
%
\[
  \tilde B_p: \param' \longmapsto  \left[ (F_0(\param')(\bar v(\simplex_i)), F_0(\param')(\bar v(\simplex'_i)))_{i=1}^m, (F_0(\param')(\bar v(\simplex'_j)))_{j=1}^n \right]
    \]
%
    is a local $C^r$ lift of~$B_p$ around~$\param$. The corresponding differential for~$B_p$ at $\param$ is:
\[d_{\param, \tilde B_p} B_p(\cdot)=
\left[ (d_\param F_0(\cdot)(\bar v (\simplex_i)), d_\param F_0(\cdot)(\bar v(\simplex'_i)))_{i=1}^m, (d_\param F_0(\cdot)(\bar v(\simplex'_j)))_{j=1}^n\right]. \]
\end{corollary}
\begin{proof}
  For $\param \in \manifoldm\setminus \Sing(\parametrization_0)$, the pre-order on the vertices $K_0$ induced by $\parametrization_0$ is constant in an open neighborhood~$U$ of~$\param$. By Proposition~\ref{proposition_lower_star_nice_filtration}, each $\parametrization(\param')(\simplex)$ rewrites as $\parametrization_0(\param')(\bar v(\simplex))$ for $\param'\in U$, which implies that the pre-order on the simplices of~$K$ induced by $\parametrization$ is also constant over~$U$. The fact that  $B_p$ is $r$-differentiable at~$\param$ follows then from  Theorem~\ref{theorem_PH_differentiable_local}, since~$\parametrization$ itself is $C^r$ on an open neighborhood of~$\param$ (again by Proposition~\ref{proposition_lower_star_nice_filtration}, and by the fact that~$\Sing(\parametrization_0)$ is closed). The rest of the corollary is an immediate consequence of Proposition \ref{proposition_derivatives_barcode_simplicial_complex}.
\end{proof}
\begin{example}
\label{example_derivative_height}
Consider our running example of parametrization of height filtrations $\parametrization_0(\param) = h_\param: v\in K_0 \mapsto \langle v, \param\rangle \in \mathbb{R} $, where $K$ is  a fixed geometric simplicial complex in $\R^d$ and $\param \in \mathbb{S}^{d-1}$. In this case, we know from Example~\ref{ex:height_func_diff1} that $B_p$ is generically $\infty$-differentiable. Corollary~\ref{cor_lower_star_barcode_differentiability} provides  another proof of this fact:
since~$F_0$ is $C^\infty$, $B_p$ is $\infty$-differentiable outside~$\Sing(\parametrization_0)$, which has generic complement in~$\mathbb{S}^{d-1}$. Moreover, the components of the differential of~$B_p$ at~$\param\in \mathbb{S}^{d-1} \setminus \Sing(\parametrization_0)$ are the~$d_\param \parametrization_0(\cdot)(v)$, whose corresponding gradients (in the tangent space~$T_\param \mathbb{S}^{d-1}$ equipped with the Riemannian structure inherited from~$\R^d$) are $v - \langle v, \param\rangle\, \param$. 
\end{example}
%

\subsection{Rips filtrations of point clouds}
\label{section_discrete_smoothness_subsection_examples_ex2}

Given a finite point cloud $P=(p_1, \cdots, p_n) \in\R^{nd}$, the \textit{Rips} filtration of~$P$ is a filtration of the total complex~$K:=2^{\{1,\cdots,n\}}\setminus \{\emptyset\}$ with $n:=\# P$ vertices, where the time of appearance of a simplex~$\simplex \subseteq \{1, \cdots, n\}$ is $\max_{i,j \in \simplex} \|p_i - p_j\|_2$. 
\cite{gameiro2016continuation} optimize the positions of the points of~$P$ in~$\R^d$ so that the barcode of the Rips filtration reaches some target barcode. Here we see~$\R^{nd}$ as our parameter space~$\manifoldm$, and we consider the parametrization
\[ \parametrization{}(P)(\simplex{}):=\max_{i, j \in \simplex{}} \, \|p_i-p_j\|_2. \]
The differentiability result of \cite{gameiro2016continuation} can be expressed as a result on the differentiability of the barcode-valued map~$B_p = \Dgm_p\circ \parametrization$ using our framework. We require that  the points of~$P$ lie in general position as defined hereafter:
\begin{definition}[\cite{gameiro2016continuation}]
\normalfont
\label{general_position}
$P$ is in \textit{general position} if the following two conditions hold:
\begin{itemize}
    \item[(i)] $\forall i\neq j\in \{1,...,n\}$, $p_i \neq p_j$;
    \item[(ii)] $\forall \{i,j\}\neq \{k,l\}$, where $i,j,k,l \in \{1,...,n\}$, $\|p_i - p_j\|_2 \neq \|p_k - p_l\|_2$. 
\end{itemize}
We denote by~$\GenPos\subseteq \R^{nd}$ the subspace of point clouds in general position. 
\end{definition}
\begin{proposition}
\label{proposition_genericity_general_position}
$\GenPos$ is generic in $\R^{nd}$.
\end{proposition}
\begin{proof}
The set of point clouds $P$ such that $p_i\neq p_j$ for all $1\leqslant i\neq j \leqslant n$ is clearly generic in $\R^{nd}$. Moreover, the maps $P=(p_1,...,p_n)\mapsto \|p_i-p_j\|_2^2- \|p_k-p_l\|_2^2$ are smooth everywhere and are submersions on a generic subset of $\R^{nd}$, therefore their $0$-sets have generic complements, whose (finite) intersection is also generic. 
\end{proof}
We next observe that the parametrization $\parametrization{}$ is $C^\infty$ at point clouds~$P$ in general position. 
\begin{proposition}
\label{proposition_Rips_nice_filtration}
The parametrization~$\parametrization:\R^{nd} \rightarrow \mathbb{R}^K$ is $C^\infty$ over~$\GenPos$. 
%
%
Specifically, given $P\in \GenPos$, letting $\{\bar v(\simplex), \bar w(\simplex)\} = \argmax_{i, j \in \simplex{}} \, \|p_i-p_j\|_2$ for every $\simplex\in K$, there is an open neighborhood~$U$ of~$P$ such that $F(P')(\simplex) = \|p'_{\bar v(\simplex)} - p'_{\bar w(\simplex)}\|_2$ for every $P'=(p'_1, \cdots, p'_n)\in U$ and $\simplex\in K$, from which follows that $\parametrization$ is~$C^\infty$ at~$P$.
\end{proposition}
\begin{proof}
The continuity of~$\parametrization$ follows from the continuity of the Euclidean norm and $\max$ function.
  Assuming $P$ is in general position, the distances $\|p_i-p_j\|_2$, for $i\neq j$ ranging in $\{1, \cdots, n\}$, are strictly ordered. By continuity of $\parametrization$, this order remains the same over an open neighborhood $U$ of~$P$ in~$\R^{nd}$. Therefore, every $P'=(p'_1, \cdots, p'_n)\in U$ is also in general position, and $F(P')(\simplex) = \|p'_{\bar v(\simplex)} - p'_{\bar w(\simplex)}\|_2$  for all $\simplex\in K$. Now, the map  $P'\mapsto \|p'_{\bar v(\simplex)} - p'_{\bar w(\simplex)}\|_2$ is $C^\infty$ at~$P$ for each $\simplex$ because $p'_{\bar v(\simplex)} \neq p'_{\bar w(\simplex)}$. This implies that $\parametrization$ is $C^\infty$ at~$P$.  
\end{proof}
Defining $\bar v,\bar w$ as in Proposition~\ref{proposition_Rips_nice_filtration}, and combining this result with  Proposition~\ref{proposition_derivatives_barcode_simplicial_complex}, we deduce the following differential of~$B_p$, which only relies on derivatives of the Euclidean distance between points: 

\begin{corollary}
\label{cor_rips_barcode_differentiability}
The barcode valued map~$B_p: P\in \R^{nd}\mapsto \Dgm_p(\parametrization(P))\in Bar$ is $\infty$-differentiable in $\GenPos$. Moreover, at $P\in \GenPos$, for any barcode template $(P_p,U_p)$ o~$\parametrization(P)$ and any choice of ordering $(\simplex_1, \simplex'_1), \cdots, (\simplex_m, \simplex'_m)$, $\tau_1, \cdots, \tau_n$ of $(P_p,U_p)$, the map~$\tilde B_p$ defined on a point cloud $P'=(p'_1,...,p'_n)$ by:
 
\[
    \tilde B_p(P')= \left[ \left(\|p'_{\bar v (\simplex{}_i)}-p'_{\bar w (\simplex{}_i)}\|_2, \|p'_{\bar v (\simplex{}'_i)}-p'_{\bar w (\simplex{}'_i)}\|_2\right)_{i=1}^m, \left( \|p'_{\bar v (\tau{}_j)}-p'_{\bar w (\tau_j)}\|_2 \right)_{j=1}^n\right]
    \]
    is a local $C^\infty$ lift of~$B_p$ around~$P$. The corresponding differential~$d_{P, \tilde B_p} B_p:\R^{nd} \rightarrow  \R^{2m}\times\R^n$ is defined on a tangent vector~$u \in \R^{nd}$ by:
    \[d_{P, \tilde B_p} B_p(u)=
\left[ \left(\langle \mathbf{P}_{\bar v(\simplex_i), \bar w(\simplex_i)}, u \rangle,  \langle \mathbf{P}_{\bar v(\simplex'_i), \bar w(\simplex'_i)}, u \rangle \right)_{i=1}^m, \left(\langle \mathbf{P}_{\bar v(\tau_j), \bar w(\tau_j)}, u  \rangle\right)_{j=1}^n\right], \]
where $\mathbf{P}_{i, j}$ denotes the vector with $\frac{p_i-p_j}{\|p_i-p_j\|_2}$ as $i$-th component (resp. $\frac{p_j-p_i}{\|p_i-p_j\|_2}$ as $j$-th component) and $0$ as other components.
\end{corollary}
This result implies in particular  that $B_p$ is generically $\infty$-differentiable, since by Proposition~\ref{proposition_genericity_general_position} the set of point clouds in general position is generic in $\R^{nd}$. 
\begin{proof}
By Proposition \ref{proposition_Rips_nice_filtration}, $F$ is $C^\infty$ in $\GenPos$, which is open by Proposition~\ref{proposition_genericity_general_position}. Given~$P$ in general position, the distances $\|p_i-p_j\|_2$, for $i\neq j$ ranging in $\{1, \cdots, n\}$, are strictly ordered, and this order remains the same over an open neighborhood $U$ of $P$ in $\R^{nd}$ by continuity. By Proposition \ref{proposition_Rips_nice_filtration} again, we have $F(P')(\simplex{})=\|p'_{\bar v (\simplex{})}-p'_{\bar w (\simplex{})}\|_2$ for every $P'=(p'_1,...,p'_n)\in U$ and $\simplex\in K$. Therefore, the pre-order induced by $F$ on the simplices of $K$ is constant over $U$. Consequently, $B_p$ is $\infty$-differentiable at $P$ by Theorem \ref{theorem_PH_differentiable_local}. The rest of the statement is an immediate consequence of Proposition~\ref{proposition_derivatives_barcode_simplicial_complex}.
\end{proof}
%
%
We conclude this section by considering a parametrization that constraints the points $p_1,...,p_n$ to evolve along smooth submanifolds $\manifoldm_1,...,\manifoldm_n$ of $\R^d$: 

\begin{proposition}
\label{proposition_rips_differetniable_submanifolds}
Let $\manifoldm_1,...,\manifoldm_n$ be smooth submanifolds of $\R^d$. Denoting by $\iota:\manifoldm_1\times ...\times \manifoldm_n \hookrightarrow \R^{nd}$ the inclusion map, the barcode valued map $B_p=\Dgm_p\circ F \circ \iota$
is generically $\infty$-differentiable.
\end{proposition}
\begin{proof}
Let $\manifoldm:=\manifoldm_1\times ...\times \manifoldm_n$. The parametrization $\parametrization\circ \iota$ is $C^\infty$ at parameters $\param\in\paramspace$ such that, locally, for all $\param'$ in a sufficiently small open neighborhood around $\param$, the following quantities are constant:
\begin{itemize}
     \item[(i)] the indices of the points at distance $0$ of each other in $\iota(\param')$, and
    \item[(ii)] the pre-order on the pairwise distances in $\iota(\param')$.  
    \end{itemize}
Note that, in this case, the point clouds~$\iota(\param')$ are not necessarily in general position, but the way they violate conditions (i) and (ii) of Definition~\ref{general_position} is constant. Let $U'$ (resp. $U$) denote the set of points in $\paramspace$ where (i) (resp. (ii)) is satisfied. From the above, $\parametrization\circ\iota$ is $C^\infty$ over $U\cap U'$. We now show that $U\cap U'$ is generic in $\paramspace$. 

Calling $U_{ijkl}$ the quadric $\{P\in \R^{nd} \mid \|p_i-p_j\|_2=\|p_k-p_l\|_2\}$, and  $U'_{ij}$ the hyperplane $\{P\in \R^{nd} \mid p_i=p_j\}$, for $i,j,k,l$ ranging in $\{1, \cdots, n\}$, we have: 
\begin{align*}
    U &= \bigcap_{\{i,j\}\neq\{k,l\}} \manifoldm \setminus \partial \iota^{-1}(U_{ijkl}) \\ 
     U'&= \bigcap_{i\neq j} \, \manifoldm\setminus \partial \iota^{-1}(U'_{ij}).
\end{align*} 
Indeed, for any $\{i,j\}\neq\{k,l\}$, the order between $\|p_i-p_j\|_2$ and $\|p_k-p_l\|_2$ in $\iota(\theta)$ is strict when $\param$ is in the (open) complement of $\iota^{-1}(U_{ijkl})$, constantly an equality when $\param$ is inside  the (open) interior $\iota^{-1}(U_{ijkl})^\circ$, and not locally constant when $\param$ lies on the boundary $\partial\iota^{-1}(U_{ijkl})$. Hence the formula for~$U$. The formula for~$U'$ follows from the same argument. 

The sets $\partial \iota^{-1}(U_{ijkl})$ and $\partial \iota^{-1}(U'_{ij})$ are boundaries of closed sets, and thus their complements in~$\paramspace$ are generic. As finite intersections of generic sets,~$U$ and~$U'$ are themselves generic. Theorem~\ref{theorem_PH_differentiable_global} allows us to conclude.
%
\end{proof}

\subsection{Rips filtrations of clouds of ellipsoids}
\label{section_discrete_smoothness_subsection_examples_ex3}
As pointed out by \cite{breiding2018learning}, in some cases, growing isotropic balls around the points of~$P=(p_1,...,p_n)\in \mathbb{R}^{nd}$ may result in a loss of geometric information. It is then advised to grow rather ellipsoids with distinct covariance matrices around each point, to account for the local anistropy of the problem. Formally, the \textit{Ellipsoid-Rips} filtration of $P$ with respect to the vector of covariance matrices $A=(A_1,...,A_n)\in (S_{d,+}(\mathbb{R}))^n$ is a filtration of the total complex $K:=2^{\{1,...,n\}}\setminus \{\emptyset\}$ with $n:=\#P$ vertices, in which the time of appearance of a simplex $\simplex{}\subseteq \{1,...,n\}$ is given by:
\[ \max_{i,j\in \sigma} r_{i,j}(A) \quad \mathrm{where} \quad r_{i,j}(A):=\left\|\frac{p_i-p_j}{\frac{1}{2}\left(\sqrt{q_i\left(\frac{p_i-p_j}{\|p_i-p_j\|_2}\right)}+\sqrt{q_j\left(\frac{p_j-p_i}{\|p_i-p_j\|_2}\right)}\right)}\right\|_2, \]
where the $q_i: x\in \R^d \mapsto \langle A_ix,x \rangle $ are the quadrics determined by the positive definite matrices $A_i$.\footnote{The quantity $r_{i,j}(A)$ serves as a proxy for the intersection of the two ellipsoids with covariance matrices $A_i$ and $A_j$ centered at $p_i$ and $p_j$ suggested by \cite{breiding2018learning}, as the problem of computing intersections of quadrics is in general NP-hard.} Here we see the space $(S_{d,+}(\mathbb{R}))^n$ as our parameter space $\manifoldm$, whose smooth structure is inherited from that of $\left(\R^{\frac{d(d+1)}{2}}\right)^n$, and we consider the parametrization:
\[ \parametrization(A)(\simplex{}):=\max_{i,j\in \sigma} r_{i,j}(A). \]
We are then interested in the differentiability of the barcode valued map $B_p=\Dgm_p\circ F$. Inspired by the case of isotropic Rips filtrations, we require that the covariance matrices in~$A$ lie in general position as defined hereafter:
\begin{definition}
\normalfont
\label{general_position_ellipsoid}
The pair $(A,P)$ is in \textit{general position} if the two following conditions hold:
\begin{itemize}
    \item all points in $P$ are distinct, i.e, $p_i\neq p_j$ whenever $1\leqslant i\neq j \leqslant n$~;
    \item all pairwise "ellipsoidal" distances are distinct, i.e, $r_{i,j}(A)\neq r_{k,l}(A)$ whenever $\{i,j\}\neq \{k,l\}\subseteq \{1,...,n\}$.  
\end{itemize}
\end{definition}
\begin{proposition}
\label{genericity_general_position_ellipsoid}
Assume the points of $P$ to be pairwise distinct. Then, the set of vectors of covariance matrices~$A$ such that $(A,P)$ is in general position is generic in~$S_{d,+}(\mathbb{R}^d)^n$.
\end{proposition}
\begin{proof}
First, we claim that the sets $O_{ijkl}:=\{A \in S_{d,+}(\mathbb{R}^d)^n \mid r_{i,j}(A)= r_{k,l}(A)\}$, for $\{i,j\}\neq \{k,l\}$, are level-sets of some smooth real valued functions on $S_{d,+}(\mathbb{R}^d)^n$ whose gradients are nowhere zero. To prove this fact, we introduce the quantities $C:=\frac{\|p_i-p_j\|_2}{\|p_k-p_l\|_2}$ and $(x,y):=(\frac{p_i-p_j}{\|p_i-p_j\|_2},\frac{p_k-p_l}{\|p_k-p_l\|_2})$. Then: 
\begin{align*}
  A=(A_1,...,A_n)\in O_{ijkl} & \Leftrightarrow r_{i,j}(A)= r_{k,l}(A) \\
 & \Leftrightarrow \frac{\sqrt{q_i(x)}+\sqrt{q_j(x)}}{\sqrt{q_k(y)}+\sqrt{q_l(y)}} = C  \\
& \Leftrightarrow \frac{\sqrt{<A_i x,x>}+\sqrt{<A_j x,x>}}{\sqrt{<A_ky,y>}+\sqrt{<A_l y,y>}} = C.
\end{align*}
Note that $x,y$ are non zero because points in~$P$ are distinct. Therefore, the map $f_{ijkl}:=A \in S_{d,+}(\mathbb{R})^d  \mapsto \frac{\sqrt{<A_i x,x>}+\sqrt{<A_j x,x>}}{\sqrt{<A_ky,y>}+\sqrt{<A_l y,y>}} \in \mathbb{R}$ is well-defined and smooth on $S_{d,+}(\mathbb{R})^n$ as the two inner products in the denominator are always strictly positive. We want to compute $\nabla f_{ijkl}=(\nabla_{A_1} f_{ijkl}, ...,  \nabla_{A_n} f_{ijkl})$ where $\nabla_{A_t} f_{ijkl}$ is the gradient of $f_{ijkl}$ with respect to the $t$-th component of $A$. For $t=i$:
\[ \nabla_{A_i} f_{ijkl}=\frac{1}{\sqrt{<A_ky,y>}+\sqrt{<A_l y,y>}} \times \frac{1}{2 \sqrt{<A_i x,x>}} \times \nabla_{A_i} <A_i x, x>. \]
The first  two factors are strictly positive scalars for any $A\in S_{d,+}(\mathbb{R})^d$. The last factor is the gradient of a non-zero linear map, so it is non-zero. As a consequence, the gradient $\nabla_{A} f_{ijkl}$ is nowhere zero, which proves our claim.

Then, by the constant rank theorem, each $O_{ijkl}$ is a smooth sub-manifold of $S_{d,+}(\mathbb{R}^d)^n$ of dimension strictly lower than that of $S_{d,+}(\mathbb{R}^d)^n$. Taking their (finite) union allows us to conclude. 
\end{proof}
From this point, the same chain of arguments as in the isotropic case allows us to show that the parametrization $\parametrization{}$ is $C^\infty$ at vectors of covariance matrices~$A$ in general position, and to express the differential of $B_p$ at~$A$. Assume the points of~$P$ to be pairwise distinct, and denote by~$\GenPosA\subseteq S_{d,+}(\mathbb{R}^d)^n$ the subspace of covariance matrices~$A$ such that $(A,P)$ is in general position.
\begin{proposition}
\label{proposition_ellipsoid_nice_filtration}
The parametrization $F:S_{d,+}(\mathbb{R}^d)^n \rightarrow \R^K$ is $C^\infty$ over~$\GenPosA$. %
%
%
Specifically, given $A\in \GenPosA$, letting $\{\bar v(\simplex), \bar w(\simplex)\}= \argmax_{i,j\in \sigma} r_{i,j}(A)$ for every $\sigma\in K$, there is an open neighborhood~$U$ of~$A$ such that $F(A')(\simplex)=r_{\bar v(\simplex),\bar w(\simplex)}(A')$ for every $A'=(A'_1,...,A'_n)\in U$ and $\simplex\in K$, from which follows that~$F$ is~$C^\infty $ at~$A$.
\end{proposition}
\begin{proof}
  Let $A\in  \GenPosA$. Then, the maps $r_{i,j}$ are $C^\infty$ because the points of~$P$ are pairwise distinct, and furthermore  the quantities~$r_{i,j}(A)$, for~$i\neq j$ ranging in~$\{1,\cdots,n\}$, are strictly ordered. By continuity, this order remains the same over an open neighborhood~$U$ of~$A$ in $S_{d,+}(\mathbb{R}^d)^n$. Therefore, for every~$A'\in U$, for all~$\simplex\in K$, we have $F(A')(\simplex)=r_{\bar v(\simplex),\bar w(\simplex)}(A')$. This implies that~$F$ is~$C^\infty$ at~$A$.
\end{proof}
Defining $\bar v, \bar w$ as in Proposition \ref{proposition_ellipsoid_nice_filtration}, and combining this result with Proposition \ref{proposition_derivatives_barcode_simplicial_complex}, we deduce the following formula for the differential of~$B_p$, which only rely on derivatives of the maps~$r_{i,j}$:
\begin{corollary}
\label{cor_ellipsoid_barcode_differentiability}
The barcode valued map~$B_p: A \in S_{d,+}(\mathbb{R}^d)^n \mapsto \Dgm_p(\parametrization(A))\in Bar$ is $\infty$-differentiable over $\GenPosA$. Moreover, at $A\in \GenPosA$, for any barcode template $(P_p,U_p)$ of $\parametrization(A)$ and any choice of ordering $(\simplex_1, \simplex'_1), \cdots, (\simplex_m, \simplex'_m)$, $\tau_1, \cdots, \tau_n$ of $(P_p,U_p)$, the map~$\tilde B_p$ defined by:

\[
   A'=(A'_1,...,A'_n) \longmapsto  \left[ \left(r_{\bar v (\simplex{}_i), \bar w (\simplex{}_i)}(A'), r_{\bar v (\simplex{}'_i), \bar w (\simplex{}'_i)}(A')\right)_{i=1}^m, \left( r_{\bar v (\tau{}_j), \bar w (\tau{}_j)}(A') \right)_{j=1}^n\right]
    \]
    is a local $C^\infty$ lift of~$B_p$ around~$P$, whose differential provides a closed formula for~$d_{A, \tilde B_p} B_p$.
\end{corollary}
This result implies in particular  that $B_p$ is generically $\infty$-differentiable, since by Proposition~\ref{genericity_general_position_ellipsoid} the set of vectors of covariance matrices in general position is generic in $S_{d,+}(\mathbb{R}^d)^n$ (provided the points of~$P$ are pairwise distinct).
\begin{proof}
By Proposition \ref{proposition_ellipsoid_nice_filtration}, $F$ is $C^\infty$ in~$\GenPosA$, which is open by Proposition~\ref{genericity_general_position_ellipsoid}. Given~$A\in \GenPosA$, the quantities $r_{i,j}(A)$, for~$i\neq j$ ranging in $\{1, \cdots, n\}$, are strictly ordered, and this order remains the same over an open neighborhood~$U$ of~$A$ in $S_{d,+}(\mathbb{R}^d)^n$ by continuity. By Proposition \ref{proposition_ellipsoid_nice_filtration} again, we have $F(A')(\simplex{})=r_{\bar v (\simplex{}), \bar w (\simplex{})}(A')$ for every $A'=(A'_1,...,A'_n)\in U$ and $\simplex\in K$. Therefore, the pre-order induced by $F$ on the simplices of $K$ is constant over $U$. Consequently,~$B_p$ is $\infty$-differentiable at~$A$ by Theorem \ref{theorem_PH_differentiable_local}. The rest of the statement is an immediate consequence of Proposition~\ref{proposition_derivatives_barcode_simplicial_complex}.
\end{proof}

\begin{remark}
Corollaries~\ref{cor_rips_barcode_differentiability} and~\ref{cor_ellipsoid_barcode_differentiability} can be combined together to generically differentiate the barcode valued map~$B_p$ with respect to both the point positions and the covariance matrices. The corresponding parameter space is~$\R^{nd}\times S_{d,+}(\mathbb{R}^d)^n$. 
\end{remark}

\subsection{Arbitrary filtrations of a simplicial complex}
\label{section_discrete_smoothness_subsection_examples_ex4}

In certain scenarios, the optimization takes place in the entire space of filter functions~$\Filt(K)$ on a fixed simplicial complex~$K$. For instance, in the context of topological simplification of a filter function $f_0$, as described by \cite{attali2009persistence,elz-tps-02}, one looks for a filter function $f\in \R^K$ which is $\e$-close to $f_0$ in supremum norm and whose diagram $\Dgm_p(f)$ equals $\Dgm_p(f_0)\setminus \Delta_{\epsilon}$, where $\Delta_{\epsilon}$ is the set of intervals of $\Dgm_p(f_0)$ that are $\e$-close to the diagonal. One way to formalize this question is as a soft-constrained optimization problem, whereby the bottleneck distance to the simplified barcode is to be minimized in tandem with the supremum-norm distance to the original function:
\[
\min_{f\in \Filt(K)} \db(\Dgm_p(f), \Dgm_p(f_0)\setminus \Delta_\epsilon) + \lambda\, \|f-f_0\|_\infty,
\]
for some fixed mixing parameter~$\lambda$. This optimization problem can be tackled using a variational approach, for which it is more convenient to work in the manifold~$\R^K$ containing~$\Filt(K)$. However, in order to avoid leaving $\Filt(K)$, we consider the parametrization of~$\R^K$ given by the indicator function of~$\Filt(K)$:
\[ \begin{array}{rccl}
  \parametrization := \mathds{1}_{\Filt(K)}: & \R^K & \to & \R^K \\[0.5em]
  & f & \mapsto & \left\{\begin{array}{l}
  f\ \mathrm{if}\ f\in \Filt(K)\\[0.5em]
  0\ \mathrm{otherwise,}
  \end{array}\right.
  \end{array} \]
which is smooth generically. The optimisation becomes then: 
\begin{equation}
\label{equation_simplification}
\min_{f\in \R^K} \db(\Dgm_p(F(f)), \Dgm_p(f_0)\setminus \Delta_\epsilon) + \lambda\, \|F(f)-f_0\|_\infty.
\end{equation}
Implementing a variational approach such as gradient descent requires both terms in~\eqref{equation_simplification} to be differentiable. The second term is generically differentiable, as the parametrization~$F$ and the norm~$\|\cdot\|_\infty$ are. The first term is the composition
%
%
\begin{equation}\label{eq:decomp_simplification}
\xymatrix{ f\in\R^K \ar[r] & \Dgm_p(F(f))\in Bar \ar[r] & \db(\Dgm_p(F(f)), \Dgm_p(f_0)\setminus \Delta_\epsilon) \in \bar\R,
} 
\end{equation}
which by the chain rule (Proposition~\ref{proposition_chain_rule_vectorization_barcode}) is differentiable as long as both arrows are. Since $\parametrization{}$ is generically differentiable, so is the first arrow by Theorem~\ref{theorem_PH_differentiable_global}.  The second arrow is the bottleneck distance to a fixed diagram and therefore also generically differentiable, as will be argued in Section~\ref{section_differentiability_vectorization}. There, we also view Eq.~\eqref{equation_simplification} as an instance of semi-algebraic loss function, which can be minimised via Stochastic Gradient Descent (SGD).
%
%
\section{The case of barcode valued maps derived from real functions on a manifold}
\label{section_continuous_smoothness_theorem}

In this section we consider barcode valued maps that factor through the space~$\R^\manifoldx$ of real functions on a fixed smooth compact $d$-manifold $\manifoldx$ without boundary. Since we seek statements about the differentiability of~$B$,  we restrict the focus to maps that factor through $C^\infty(\manifoldx,\R)$ equipped with the standard Whitney $C^\infty$ topology:\footnote{This topology coincides with all usual topologies on $C^\infty(\manifoldx,\R)$ because $\manifoldx$ is compact.}
\[ B:\ \xymatrix{\manifoldm{} \ar^-{F}[r] & C^\infty(\manifoldx,\R) \ar^-{\Dgm}[r] & Bar^{d+1}}. \]
Here, $\Dgm$ is the map that takes a function~$f\in
C^\infty(\manifoldx,\R)$ to the vector of its barcodes
$(\Dgm_p(f))_{p=0}^{d}$. It is well-defined on
$C^\infty(\manifoldx,\R)$, as continuous functions on triangulable
spaces have well-defined persistence
diagrams \citep{chazal2016structure}. However, as in the previous sections, we want to work only with  barcodes that have finitely many off-diagonal points, therefore we further assume
that $F$ takes its values in the subset $\text{Tame}(\manifoldx)$ of
tame $C^\infty$ functions---note that $\text{Tame}(\manifoldx)$ contains the generic subset of Morse functions
on~$\manifoldx$ \citep{milnor2016morse}. Hence the factorization:
\[ B:\ \xymatrix{\manifoldm{} \ar^-{F}[r] & \text{Tame}(\manifoldx) \ar^-{\Dgm}[r] & Bar^{d+1}}. \]
As before, we call~$\parametrization$ the {\em parametrization} associated to~$B$, and $\paramspace$ the {\em parameter space}, whose elements are generally refered to as~$\param$. We also denote $F(\param{})$ by $f_{\param{}}$ to emphasize the fact that $F$ is valued in a function space. The map $\Dgm$ takes~$f_\param$ to the vector of its barcodes $(\Dgm_p(f_\param))_{p=0}^{d}$, so we can take advantage of the bijective correspondence between the critical points of~$f_\param$ (provided~$f_\param$ is Morse) and the interval endpoints in this vector (Proposition~\ref{endpoints_in_barcode_are_critical_values}). 

As in the case of a parametrization valued in the space of filter functions on a simplicial complex, we need~$\parametrization$ to be smooth in some reasonable sense to ensure that the composite~$B$ is $\infty$-differentiable. For this, we define a curve~$c:\R\rightarrow C^\infty(\manifoldx,\R)$ to be {\em differentiable} if the limit~$\lim_{h\rightarrow 0}\frac{c(t+h)-c(t)}{h}$ exists for all~$t\in \R$. The limit can be viewed as a curve, and when iterated limits exist, we say that~$c$ is a {\em smooth} curve. We then say that the parametrization~$\parametrization$ is {\em smooth}\footnote{This notion of smooth map is key in the theory developped in \cite{frolicher1988linear,kriegl1997convenient} for adapting the concepts of differential geometry to a wide category of infinite-dimensional vector spaces, including Banach and Fr\'echet spaces.} if it sends every smooth curve~$\param(t)$ in~$\manifoldm$ to a smooth curve~$\parametrization(\param(t))$ in~$C^\infty(\manifoldx,\R)$. By Corollary~11.9 in~\cite{michor1980manifolds}, if~$\parametrization$ is smooth, then its uncurrified version
\begin{equation}
\label{equation_exp_law_finite_manifold}
\tilde{\parametrization}:(\param,x) \in \manifoldm{}\times \manifoldx\longmapsto \parametrization(\param)(x) \in \mathbb{R}
\end{equation}
is a smooth map in the usual sense, to which we can therefore apply standard results from differential calculus, typically the implicit function theorem. This will be instrumental in the proof of our main result (Theorem~\ref{smoothness_barcode_manifold}).

\subsection{Smoothness of the barcode valued map}

\begin{theorem}[Continuous smoothness]
\label{smoothness_barcode_manifold}
Let $\parametrization{}: \paramspace{} \rightarrow C^\infty(\manifoldx, \R)$ be a parametrization of class $C_c^\infty$ valued in $\text{Tame}(\manifoldx)$. Let $\param{}\in \paramspace{}$ be a parameter such that $f_\param{}$ is Morse with critical values of multiplicity 1. Then, $B$ is $\infty$-differentiable at $\param{}$. 
\end{theorem}
\begin{proof}
Since $f_{\param{}}$ is a Morse function on a compact manifold, $\Crit(f_{\param{}})$ is a finite set whose cardinality we denote by $N_\param{}$. We will proceed by proving the following statements in sequence:
\begin{enumerate}
    \item[(i)] There exist an open neighborhood $U$ of $\param{}$ and smooth maps $\pi_l:U\rightarrow \manifoldx$ for $1\leqslant l \leqslant N_\param{}$ that track the critical points, that is:
      \begin{equation}
    \label{equation_tracking_critical_points}
    \forall \param{}' \in U,
    \Crit(f_{\param{}'})=\{\pi_{l}(\param{}')\}_{1\leqslant
      l \leqslant N_\param{}}
\end{equation}
    \item[(ii)] Shrinking $U$ if necessary, we further have that for any $\param{}'\in U$, $f_{\param{}'}$ is Morse with critical values of multiplicity~1.
    \item[(iii)] Let $\param{}'\in U$ and $(b,d)\in \Dgm_p(\manifoldx,f_{\param{}'})\setminus \Delta$ for some homology degree $p$. Then, either $d=+\infty$, in which case there exists a unique $1\leqslant l \leqslant N_\param{}$ such that $b=f_{\param{}'}(\pi_l(\param{}'))$, or $d<+\infty$, in which case there exist unique $1\leqslant l\neq l' \leqslant N_\param{}$ such that $(b,d)=(f_{\param{}'}(\pi_l(\param{}')),f_{\param{}'}(\pi_{l'}(\param{}')))$.
    \item[(iv)] For all $\param{}_1, \param{}_2 \in U$, $1\leqslant l\neq l'\leqslant N_\param{}$, and $0\leqslant p \leqslant d$, we have: \\ $(f_{\param{}_1}(\pi_l(\param{}_1)),f_{\param{}_1}(\pi_{l'}(\param{}_1)))\in \Dgm_p(f_{\param{}_1})$ (resp. $(f_{\param{}_1}(\pi_l({\param{}}_1)),+\infty)\in \Dgm_p(f_{\param{}_1})$) if and only if $(f_{\param{}_2}(\pi_l(\param{}_2)),f_{\param{}_2}(\pi_{l'}(\param{}_2)))\in \Dgm_p(f_{\param{}_2})$ (resp. $(f_{\param{}_2}(\pi_l(\param{}_2)),+\infty)\in \Dgm_p(f_{\param{}_2})$).
    \item[(v)] There exist smooth local coordinate systems for $B_p$ at $\param{}$ for every $0\leqslant p \leqslant d$. Therefore, by Proposition \ref{proposition_relating_differentiability_barcode_ordered_barcode}, the barcode valued map $B$ is $\infty$-differentiable at $\param{}$.
\end{enumerate}
 The proofs of assertions (i) and (ii) use differential geometry: we show that we can smoothly track the critical points of $f_{\param{}'}$ as $\param{}'$ varies in a neighborhood of $\param{}$. The proof of assertion (iii) simply exploits the fact that the endpoints in the barcodes of a Morse function are its critical values (Propostion \ref{endpoints_in_barcode_are_critical_values}). Assertion (iv) means that the critical points do not exchange their contributions to the persistence diagrams when the parameter is varying. This will be shown using standard tools in persistence theory. Assertion (v) is obtained by re-indexing the set $\{1,...,N_\param{}\}$ such that, through this re-indexation, the maps $\param{}'\mapsto f_{\param{}'}(\pi_l(\param{}'))$ provide local coordinate systems as defined in Definition \ref{definition_derivatives_barcode_valued_maps}. 

\paragraph*{Proof of assertion (i): \, }

The tangent bundle $T\manifoldx=\bigsqcup_{x\in \manifoldx}\{x\}\times T_x\manifoldx$ is a smooth manifold of dimension $2d$.
%
Let $x_1,...,x_{N_\param{}}$ be the critical points of $f_{\param{}}$. Locally, in an open neighborhood $\mathds V$ of these critical points, the tangent bundle is parallelizable, i.e. we have a diffeomorphism $T\mathds V \cong \mathds V \times \mathbb R^ d$ and the projection onto the second component provides a smooth map to $\mathbb R^d$. 
Consider the map:
\[ \partial \parametrization{}: (\param{}',x)\in \paramspace{}\times \mathds V \mapsto \nabla f_{\param{}'}(x) \in T_x\mathbb{V} \cong \mathbb{R}^d, \]
which is smooth due to the smoothness of~$\tilde{\parametrization{}}$, see Eq.~\eqref{equation_exp_law_finite_manifold}. 
Then, at the critical points we have $\partial \parametrization{}(\param{},x_l)=\nabla f_{\param{}}(x_l)=0$. Moreover, because $f_\param{}$ is Morse, $\nabla_x \partial \parametrization{}(\param{},x_l)= \nabla^2f_{\param{}}(x_l)$ is invertible, where $\nabla_x \partial \parametrization{}$ denotes the first derivative of $\partial \parametrization{}$ with respect to its second argument. We can then apply the implicit function theorem to $\partial \parametrization{}$: there exist an open neighborhood $U_l$ of $\param{}$, an open neighborhood $V_l$ of $x_l$ (contained in $\mathds V$) and a smooth diffeomorphism $\pi_l: U_l \rightarrow V_l$ such that 
\begin{equation}
\label{eq_implicit_function_theorem_critical_points}
    \forall (\param{}',x) \in U_l\times V_l, \, \partial \parametrization{}(\param{}',x)=0 \Longleftrightarrow x=\pi_l(\param{}').
\end{equation}
Let $U=\bigcap_{l=1}^{N_\param{}} U_l$.  After shrinking each $V_l$ so that it equals $\pi_l(U)$, we obtain that \eqref{eq_implicit_function_theorem_critical_points}~holds over~$U\times V_l$ for every $1\leqslant l\leqslant N_\param{}$. Now, by definition of $\partial \parametrization{}$ and the $(\Leftarrow)$ of \eqref{eq_implicit_function_theorem_critical_points}, we have
\[ \forall \param{}' \in U, \{\pi_l(\param{}')\}_{1\leqslant l\leqslant N_\param{}} \subseteq \Crit(f_{\param{}'}). \]
We now show the converse inclusion. From the $(\Rightarrow)$ in Equation \eqref{eq_implicit_function_theorem_critical_points}, it is sufficient to prove that no critical points of $f_{\param{}'}$ can be found in the compact set $W:=\manifoldx\setminus (\bigcup_{l=1}^{N_\param{}} V_l)$ when $\param{}'$ ranges over $U$. We equip~$\manifoldx$ with an arbitrary  Riemannian metric $g$, and we consider the smooth map:
\[ \partial G: (\param{}', \elementm{})\in U\times \manifoldx \longmapsto g(\nabla f_{\param{}'}(x), \nabla f_{\param{}'}(x))\in \mathbb{R}, \]
where $\nabla f_{\param{}'}(x)\in T_x \manifoldx$. In particular, $\partial G (\param{}',\elementm{})$ is zero if and only if $\elementm{}$ is a critical point of $f_{\param{}'}$.
As a result, $\partial G$ does not vanish on~$\{\param{}\} \times W$ since $W$  includes no critical point of~$f_{\param'}$. By the compactness of $W$ and the continuity of $\partial G$, there exists an open neighborhood $U'$ of $\param{}$ such that $\partial G_{|U'\times W}$ does not vanish either. Assertion~(i) follows after shrinking $U$ to $U\cap U'$. 

\paragraph*{Proof of assertion (ii): \,}
Let $U$ be as in assertion~(i). Since $f_\param{}$ is Morse, $\nabla_x \partial \parametrization{}(\param{},x_l)= \nabla^2f_{\param{}}(x_l)$ is invertible for each $l\in \{1,...,N_{\param{}}\}$. $\partial \parametrization{}$ is of class $C^1$ as it is of class $C^\infty$, so we get open neighborhoods $U_l'$ of $\param{}$ and $V'_l$ of $x_l$ such that $\nabla_x \partial \parametrization{}$ is invertible over $U'_l \times V'_l$. We shrink $U$ to $U \cap (\bigcap_{l=1}^{N_\param{}} U'_l)$ and each $V_l$ to $V_l \cap V'_l$, so that the critical points of $f_{\param{}'}$ are non-degenerate for $\param{}' \in U$. Shrinking $U$ further if necessary, a similar argument ensures that the critical values of $f_{\param{}'}$ have multiplicity 1 for all $\param{}'\in U$. This concludes the proof of assertion (ii).

\paragraph*{Proof of assertion (iii): \,}
Let $\param{}'\in U$. Let $(b,d)\in \Dgm_p(f_{\param{}'})\setminus \Delta$ for some homology degree $0\leqslant p\leqslant d$. We assume that $d<+\infty$. From assertion (ii), $f_{\param{}'}$ is Morse with critical values of multiplicity $1$. Therefore, by Proposition \ref{endpoints_in_barcode_are_critical_values}, $f_{\param{}'}$ induces a bijection between the multisets $\Crit(f_{\param{}'})$ and $E(f_{\param{}'})$. Meanwhile, assertion~(i) provides the equality $\Crit(f_{\param{}'})= \{\pi_l(\param{}')\}_{1\leqslant l \leqslant N_\param{}}$, so $f_{\param{}'}$ induces a bijection $\{\pi_l(\param{}')\}_{1\leqslant l \leqslant N_\param{}} \rightarrow E(f_{\param{}'})$. By taking pre-images of $b$ and $d$ which are in $E(f_{\param{}'})$, there exist some unique indices $1\leqslant l\neq l' \leqslant N_\param{}$ such that $(b,d)=(f_{\param{}'}(\pi_l(\param{}')),f_{\param{}'}(\pi_{l'}(\param{}')))$. The case $d=+\infty$ is proven the same way.

\paragraph*{Proof of assertion (iv): \,}
The maps $(\param{}_1,\param{}_2) \in U^2 \mapsto |f_{\param{}_1}(\pi_l(\param{}_1))- f_{\param{}_2}(\pi_{l'}(\param{}_2))|\in \mathbb{R}_+$, for varying $1\leqslant l\neq l' \leqslant N_\param{}$, are continuous. They are strictly positive at $(\param{},\param{})$ because $f_\param{}$ has critical values of multiplicity $1$, so $$\inf_{1 \leqslant l\neq l' \leqslant N_\param{} }{|f_{\param{}}(\pi_l(\param{}))-f_{\param{}}(\pi_{l'}(\param{}))|} > 0.$$ By continuity, shrinking~$U$ further if necessary, we have $$\inf_{1 \leqslant l\neq l' \leqslant N_\param{},\, (\param{}_1,\param{}_2) \in U^2 }{|f_{\param{}_1}(\pi_l(\param{}_1))-f_{\param{}_2}(\pi_{l'}(\param{}_2))|}>0.$$ Let $\e$ be a real number such that:
\begin{equation}
    \label{epsilon_continuous_smoothness_inter_critical}
    0<\e < \inf_{1 \leqslant l\neq l' \leqslant N_\param{},\, (\param{}_1,\param{}_2) \in U^2 }{|f_{\param{}_1}(\pi_l(\param{}_1))-f_{\param{}_2}(\pi_{l'}(\param{}_2))|}.
\end{equation}
By continuity of $\tilde{\parametrization{}}$ and compactness of $\manifoldx$, we can shrink $U$ further\footnote{This does not prevent our choice of $\e$ from satisfying~\eqref{epsilon_continuous_smoothness_inter_critical}, because shrinking $U$ increases the right-hand side of this equation.} so that $\|f_{\param{}_1}-f_{\param{}_2}\|_{\infty} \leqslant \frac{\e}{2}$ for any $\param{}_1,\param{}_2 \in U$.
  From the Stability Theorem~\ref{stability theorem} we then have:
  \begin{equation}
     \label{equation_applying_stability_theorem_continuous_smoothness}
     \forall \param{}_1,\param{}_2 \in U, \forall 0\leqslant p \leqslant d, \quad \db(\Dgm_p(f_{\param{}_1}),\Dgm_p(f_{\param{}_2}))\leqslant \frac{\e}{2}.
 \end{equation} 
 Let us fix two parameters $\param{}_1,\param{}_2 \in U$ and a homology degree~$p$. Let $1\leqslant l_1\neq l_1'\leqslant N_\param{}$ be such that $(f_{\param{}_1}(\pi_{l_1}(\param{}_1)),f_\param{}(\pi_{l_1'}(\param{}_1)))\in \Dgm_p(f_{\param{}_1})$. From Equation \eqref{equation_applying_stability_theorem_continuous_smoothness}, there exists a matching $\gamma: \Dgm_p(f_{\param{}_1}) \rightarrow \Dgm_p(f_{\param{}_2})$ with cost $c(\gamma)\leqslant \frac{\e}{2}$. In particular, if we denote $(b,d):=\gamma(f_{\param{}_1}(\pi_{l_1}(\param{}_1)),f_{\param{}_1}(\pi_{l'_1}(\param{}_1)))\in \R^2$, then
\begin{equation}
    \label{equation_matched_points_less_epsilon_distance}
    |f_{\param{}_1}(\pi_{l_1}(\param{}_1))- b|\leqslant \frac{\e}{2}  \quad \text{ and } \quad |f_{\param{}_1}(\pi_{l'_1}(\param{}_1))- d|\leqslant \frac{\e}{2}.
\end{equation}
Of course we cannot have $d=+\infty$. Also, we cannot have $(b,d)\in \Delta$, i.e $b=d$, because then the triangle inequality would imply that $|f_{\param{}_1}({\pi_{l_1}(\param{}_1)})-f_{\param{}_1}({\pi_{l_1'}(\param{}_1)})|\leqslant \frac{\e}{2}+\frac{\e}{2}= \e$, which contradicts \eqref{epsilon_continuous_smoothness_inter_critical}. Thus, $(b,d)$ is a bounded off-diagonal point of $\Dgm_p(f_{\param{}_2})$. By assertion~(iii), there exist indices $1\leqslant l_2 \neq l_2' \leqslant N_\param{}$ such that $b=f_{\param{}_2}(\pi_{l_2}(\param{}_2))$ and $d=f_{\param{}_2}(\pi_{l_2'}(\param{}_2))$. Equations \eqref{equation_matched_points_less_epsilon_distance} and \eqref{epsilon_continuous_smoothness_inter_critical} together force $l_2=l_1$ and $l_2'=l_1'$. Hence,  $(f_{\param{}_2}(\pi_{l_1}(\param{}_2)),f_{\param{}_2}(\pi_{l_1'}(\param{}_2)))=(b,d) \in \Dgm_p(f_{\param{}_2})$, which proves the result. The case of an index $1\leqslant l \leqslant N_\param{}$ such that $(f_{\param{}_1}(\pi_l(\param{}_1)),+\infty)\in \Dgm_p(f_{\param{}_1})$ is treated in the same way.

\paragraph*{Proof of assertion (v): \,}
For any homology degree $0\leqslant p\leqslant d$, by assertion~(iii), each bounded off-diagonal interval $(b,d)$ in $\Dgm_p(f_{\param{}})\setminus \Delta$ can be rewritten as $(f_{\param{}}(\pi_{l_{b,p}}(\param{})),f_{\param{}}(\pi_{{l}_{d,p}}(\param{})))$ for some indices $l_{b,p}\neq l_{d,p}$. Similarly, each interval $(v,+\infty)$ can be rewritten as $(f_{\param{}}(\pi_{l_{v,p}}(\param{})),+\infty)$ for some index $l_{v,p}$. By assertion (iv), for any parameter $\param{}'\in U$, $B_p(\param{}')$ equals
\[\big\{ (f_{\param{}'}(\pi_{l_{b,p}}(\param{}')),f_{\param{}'}(\pi_{{l}_{d,p}}(\param{}')))\big\}_{(b,d)\in \Dgm_p(f_{\param{}})\setminus \Delta} \cup \big\{ (f_{\param{}'}(\pi_{l_{v,p}}(\param{}')),+\infty)\big\}_{(v,+\infty)\in \Dgm_p(f_{\param{}})}  \cup \Delta^\infty. \]
This provides a smooth local coordinate system (see Definition \ref{definition_coordinate_system}) for~$B_p$ at~$\param{}$, therefore $B_p$~is $\infty$-differentiable at~$\param{}$ by Proposition \ref{proposition_relating_differentiability_barcode_ordered_barcode}. Since this is true for every $0\leqslant p \leqslant d$, $B$~itself is $\infty$-differentiable at~$\param$. 
\end{proof}

\begin{remark}[Multiplicity one]
\normalfont
\label{remark_multiplicity_one_necessary}
The upcoming Figure~\ref{fig:two_torus} shows how important the assumption that~$f_\param{}$ has critical values of multiplicity~1 is for the conclusion of Theorem~\ref{smoothness_barcode_manifold} to hold.   
Roughly speaking, the assumption implies that the critical points do not exchange their contributions to the persistence diagrams of~$f_\param$ under perturbations of~$\param$. We proved this fact using the Stability Theorem for persistence diagrams (see the proof of assertion~(iv) above), however it is also a consequence of the so-called {\em structural stability theorem} for dynamical systems \citep{palis2000structural}. This result implies that the gradient vector field induced by a Morse function $f_\theta$ with distinct critical values is \textit{structurally stable}, and as an immediate consequence, that the Morse-Smale complex of $f_\theta$ does not change as we smoothly perturb $f_\param{}$. The Morse-Smale complex allows us to recover the persistence module completely and, in turn, the barcode of $f_\param{}$. 
\end{remark}

\subsection{Discussion: generic differentiability}

Theorem \ref{smoothness_barcode_manifold} guarantees that $B$ is $\infty$-differentiable at parameters $\param{}$ that produce Morse functions with critical values of multiplicity 1. The set of such functions is a generic subspace of~$C^{\infty}(\manifoldx,\mathbb{R})$ \citep{golubitsky2012stable}. 
We can also argue that, under some extra conditions on the parametrization $\parametrization{}$, the set~$D(\paramspace{},\manifoldx)$ of parameters $\param{} \in \paramspace{}$ that produce Morse functions $f_\param$ with critical values of multiplicity~1 is generic in $\paramspace{}$:
\begin{proposition}[\cite{nicolaescu2011invitation}]
\label{sufficiently_large_implies_nice}
If~$\parametrization$ is smooth and {\em generically large}, i.e. for generic $x\in \manifoldx$ the map $\param{} \in \paramspace{} \mapsto df_\param{} (x) \in T_x\manifoldx^*$ is a submersion, then $D(\paramspace{},\manifoldx)$ is generic in $\paramspace{}$.
\end{proposition}
There are important examples where this result applies, such as for instance:
\begin{example}[\cite{nicolaescu2011invitation}]
\normalfont
\label{canonical_parametrizations_smooth_manifold}
Assume $\manifoldx$ is embedded in $\mathbb{R}^d$ and translated so as not to contain the origin. Then, each of the following parametrizations~$\parametrization$ is smooth and generically large:
\begin{align*}
  v \in \mathbb{R}^d & \mapsto (x \in \manifoldx \mapsto \langle v,x \rangle \in \mathbb{R}) \\
  p\in \mathbb{R}^d & \mapsto( x\in \manifoldx \mapsto |x-p|^2 \in \mathbb{R}) \\
  A \in S_+(\mathbb{R}^d) &\mapsto(x \in \manifoldx \mapsto \frac{1}{2}\langle Ax,x\rangle\in \mathbb{R})
  \end{align*}
\end{example}

\subsection{A simple example}

Take the ground space~$\manifoldx$ to be the torus~$\mathbb{S}^1\times\mathbb{S}^1$ embedded in $\mathbb{R}^3$, the parameter space~$\paramspace{}$ to be the $2$-sphere~$\mathbb{S}^2$, and the parametrization~$\parametrization$ to be the family of height filtrations, i.e $\parametrization{}: \param\in \mathbb{S}^2\mapsto (x\in \manifoldx \mapsto \langle \param,x \rangle \in \mathbb{R})$. For a generic direction $\param\in \mathbb{S}^2$, the induced height function, which we denote by $h_\param$, will be Morse and no two critical points are in the same level set. In this case we can track the critical points smoothly as we vary~$\param$, and the barcodes~$\Dgm_p(h_\param)$ also evolve smoothly. An example of this situation is given in Figure \ref{fig:generic_height_torus}.

\begin{figure}[htb]
  \centering
  \includegraphics[width=0.8\textwidth]{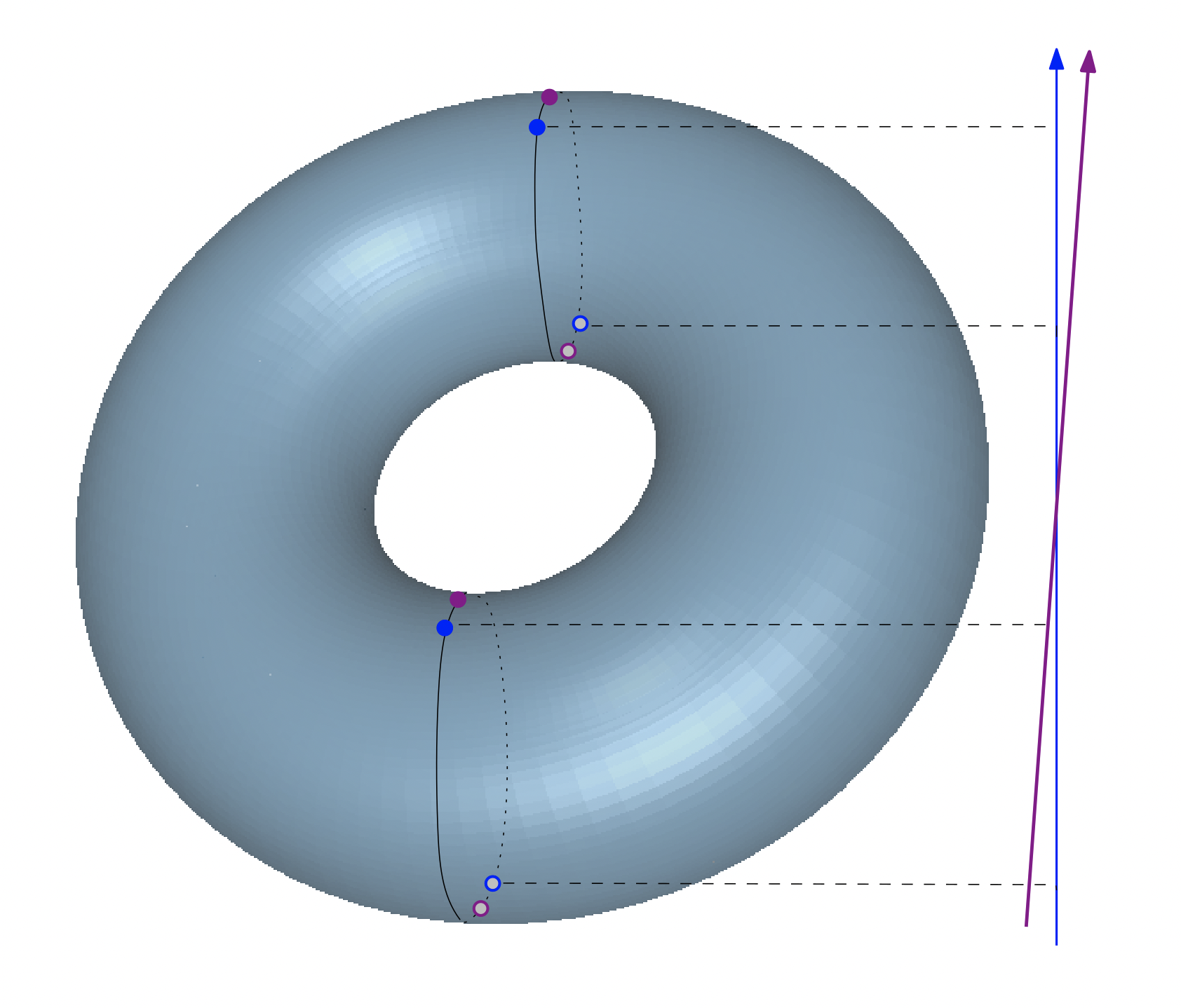}
  \caption{A torus filtered by a generic height function. The blue arrow indicates the direction~$\param$. By the correspondence of Proposition~\ref{endpoints_in_barcode_are_critical_values}, the 4 critical points (blue dots) correspond from bottom to top to an infinite bar in degree 0, an infinite bar in degree 1, another infinite bar in degree 1, and an infinite bar in degree 2 of the resulting barcode $\Dgm(h_\param)$. The implicit function theorem applied to these critical points allows us to track them smoothly when perturbing the height function (purple arrow). The correspondence to points in the barcode remains unchanged. }
  \label{fig:generic_height_torus}
\end{figure}

Even in this elementary situation, the singular parameters~$\param\in \mathbb{S}^2$ can exhibit pathological behaviors. There are two specific heights, on opposite sides of the sphere~$\mathbb{S}^2$, that produce \textit{Morse-Bott} functions. We show one of them in Figure~\ref{fig:MorseBottTorus}. At such a parameter~$\param$, the critical sets are codimension-$1$ submanifolds of~$\manifoldx$, and smooth perturbations of~$\param$ may result in discontinuous changes in the critical set.

\begin{figure}[htb]
  \centering
  \includegraphics[width=0.8 \textwidth]{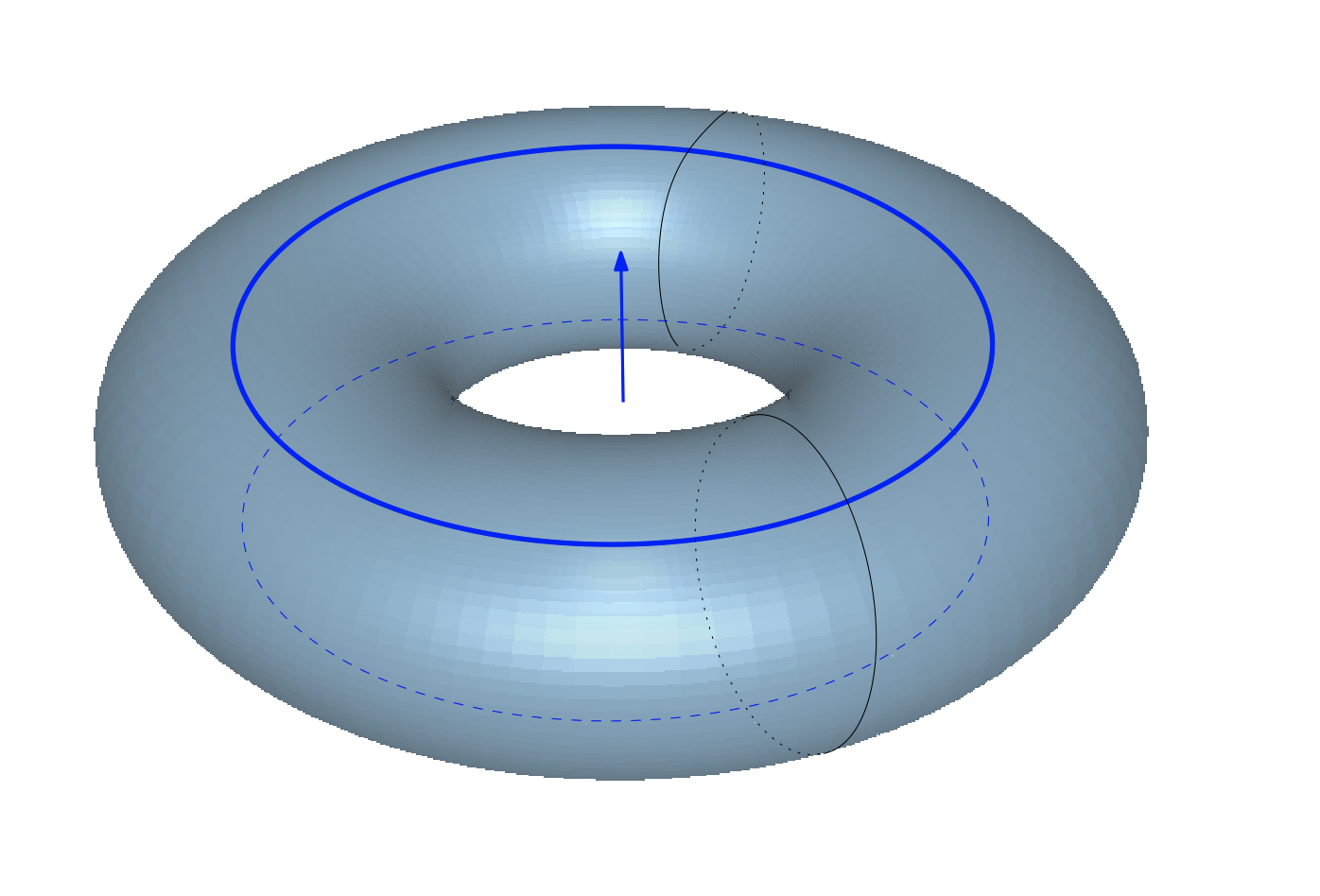}
  \caption{Horizontal torus filtered by the vertical height function~$h_\param$. The critical sets are the two blue circles, one of which corresponds to both a birth of a connected component and a loop, while the other corresponds to the births of a loop and a $2$-cycle. Observe that any slight perturbation of~$\param$ results in a valid Morse function with 4 critical points, however, the locations of these points do not vary smoothly, and not even continuously, at~$\param$. }
  \label{fig:MorseBottTorus}
\end{figure}

There are other directions~$\param$ at which the assumptions of Theorem~\ref{smoothness_barcode_manifold} are not met, yet the interval endpoints in the barcode can still be tracked smoothly. Such a case is shown in Figure~\ref{fig:two_critical_point_same_set}, where the height function~$h_\param$ is Morse but with a critical value of multiplicity 2. In this specific case, the implicit function theorem still applies to both critical points and provides a smooth local coordinate system for the barcode of~$h_\param$. 
\begin{figure}[htb]
  \centering
  \includegraphics[width=0.75 \textwidth]{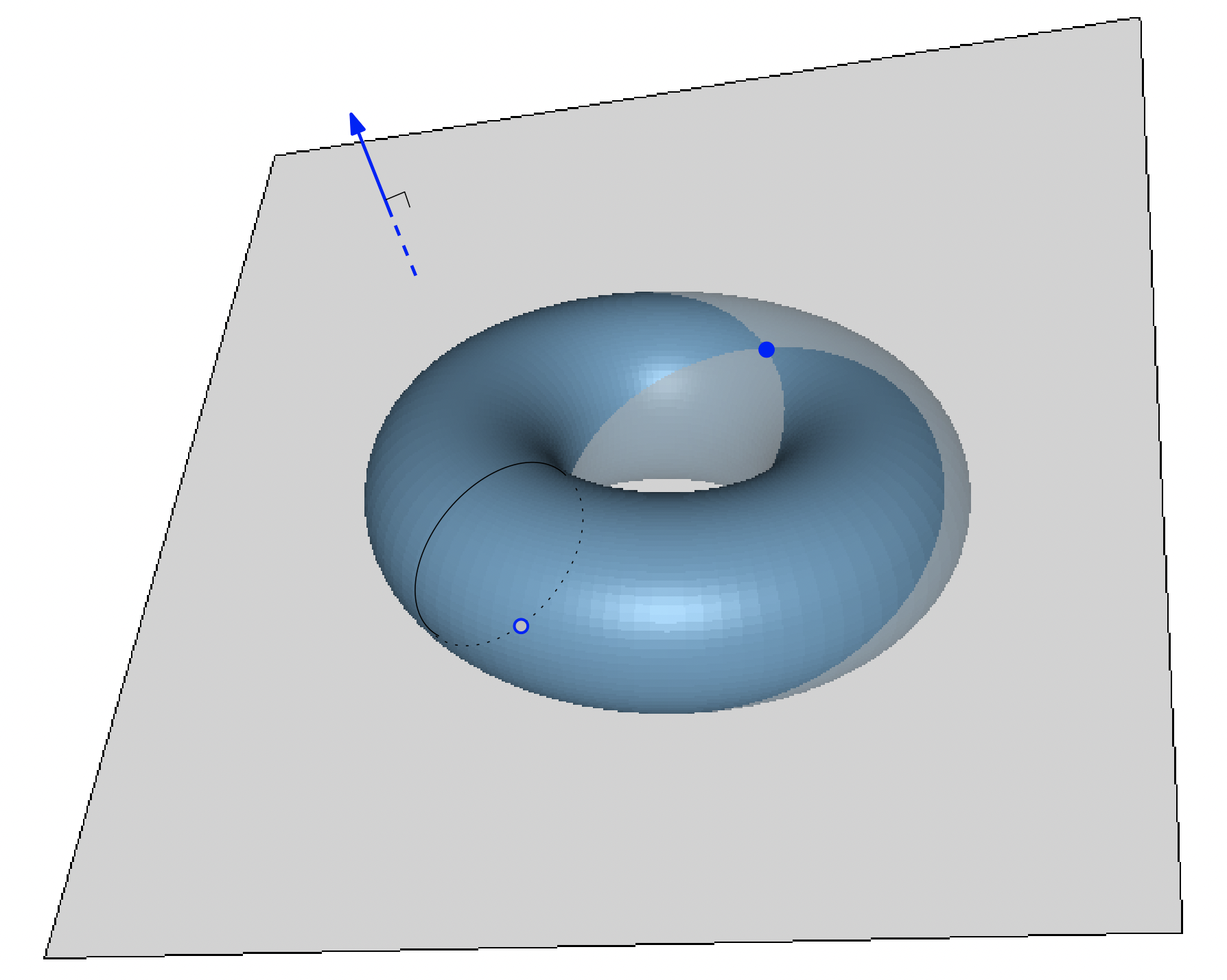}
  \caption{A height function~$h_\param$ (blue arrow) that is Morse with two critical points (blue dots) in the same level set (the hyperplane), producing two distinct loops in the persistence diagram. The critical points can still be tracked smoothly around~$\param$, and no change in the pairing occurs in the barcode~$\Dgm(h_\param)$. }
  \label{fig:two_critical_point_same_set}
\end{figure}
However, in the general case, such a Morse function with two critical points sitting in the same level-set can induce a change in the correspondence with interval endpoints in the barcode, potentially resulting in non-smooth behavior of the barcode valued map $B$. An example is given in Figure~\ref{fig:two_torus}.

\begin{figure}[htb]
\centering
\begin{subfigure}{0.92\textwidth}
  \centering
  \includegraphics[width=0.92\textwidth]{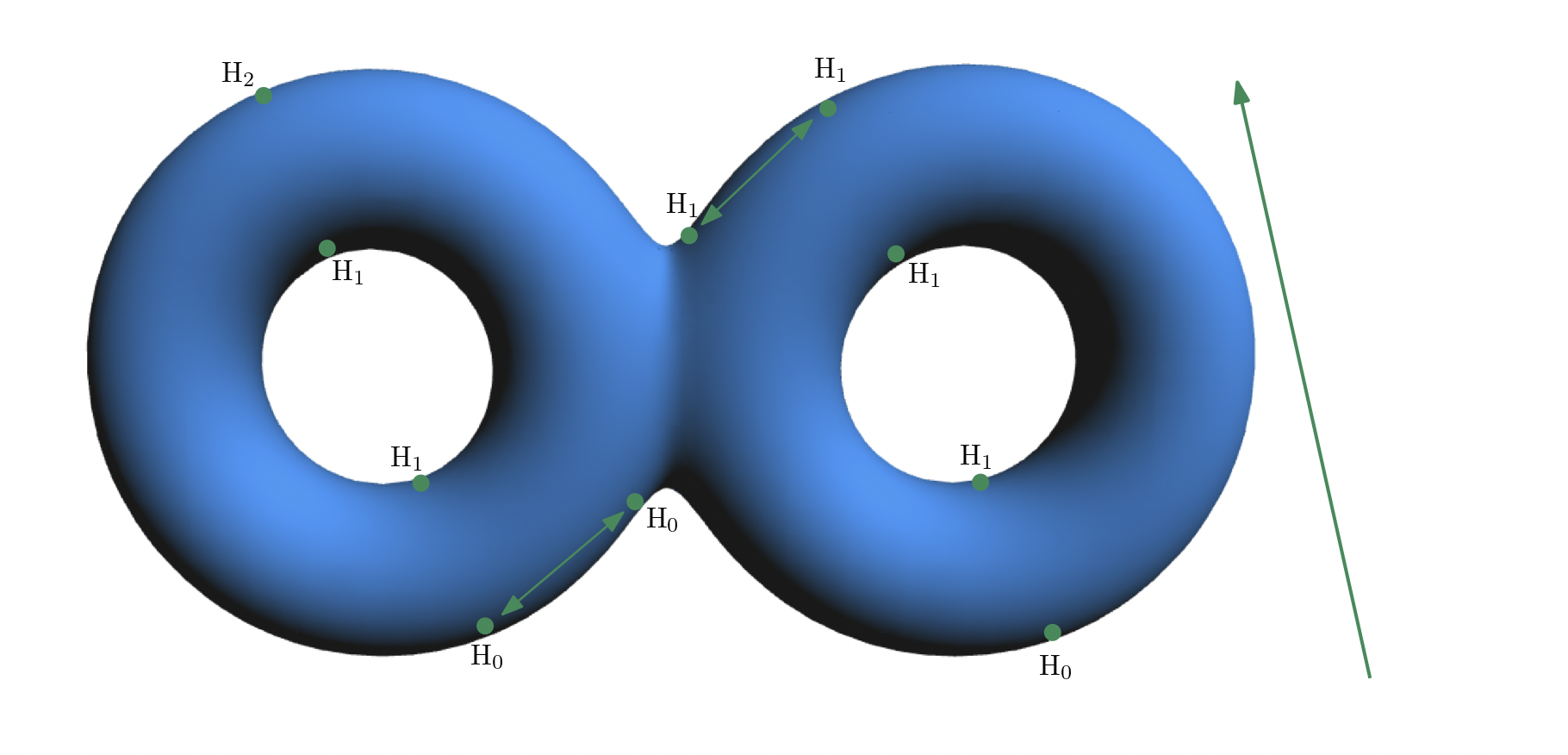}
  \label{fig:sub4}
\end{subfigure}

\begin{subfigure}{0.87\textwidth}
  \centering
  \includegraphics[width=0.87\textwidth]{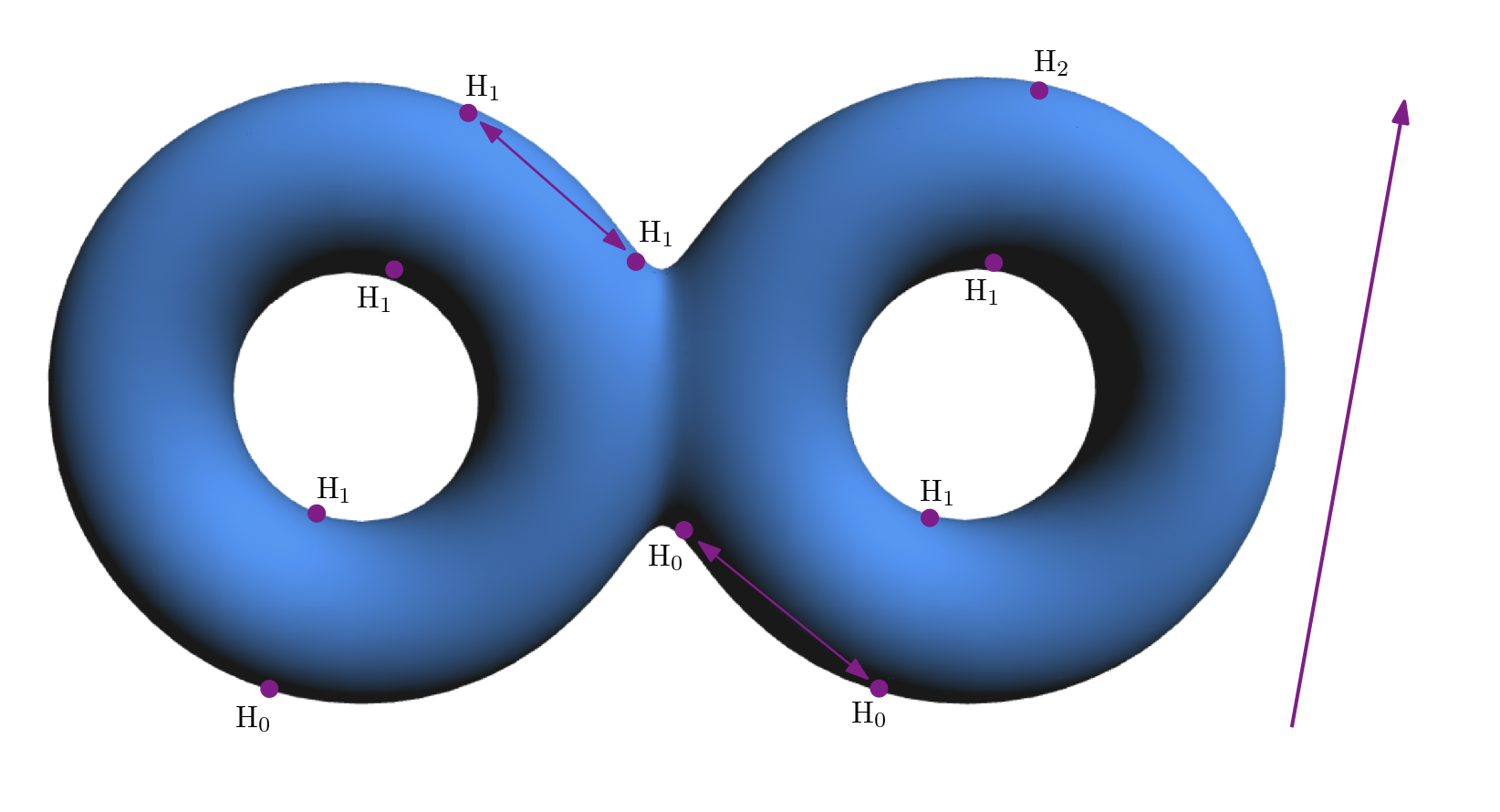}
  \label{fig:sub5}
\end{subfigure}
\caption{A 2-torus filtered by two infinitesimal perturbations of the vertical height function, together with their critical points and labels to indicate the dimension of the homology group that they affect. Paired critical points correspond to bounded intervals in the associated barcodes. Here, the vertical height function is Morse and the critical points evolve smoothly. However, the pairing between critical points is not constant, nor their homological dimensions. Therefore, the barcode valued map is not smooth at the vertical direction.}
\label{fig:two_torus}
\end{figure}

\section{The case of maps on barcodes derived from vectorizations and loss functions}
\label{section_differentiability_vectorization}

We continue on with examples of differentiable maps, this time focusing on maps~$V: Bar \to \manifoldn$ defined on barcodes and valued in a smooth finite-dimensional manifold.
There is a plethora of examples of such maps~$V$ in the literature on topological data analysis \citep{adams2017persistence,bubenik2015statistical,carriere2015stable,carriere2017sliced,di2015comparing,verovsek2016tropical}.
Most of them take $\manifoldn$ to be a Euclidean or Hilbert space, and they 
were designed to
provide meaningful (e.g. stable, discriminative) representations of barcodes that can be fed to machine learning algorithms. A prototypical example of such a map is the  persistence image of \cite{adams2017persistence}, which we study in Section~\ref{subsection_persistence_images}. Other maps even have $\manifoldn=\R$ as codomain, and they are meant to be used as loss terms in optimization tasks \citep{chen2019topological,gabrielsson2018topology,hu2019topology}. Many examples of such vectorizations and loss functions are part of the wide class of {\em linear representations}, which we study in Section~\ref{subsection_differentiability_linear_representation}.  In Section~\ref{subsection_differentiability_bottleneck}, we study an important example of non-linear loss, namely the bottleneck distance to a fixed barcode, which we believe can be of interest in the context of inverse problems. The machinery developed in this section is likely to be adaptable to other examples of maps on barcodes, however the purpose of the section is to provide a proof of concept rather than an exhaustive treatment.

\subsection{The differentiability of persistence images}
\label{subsection_persistence_images}


Recall that $Bar$ is equipped with the bottleneck topology. Let $Bar_n$ be the subset of $Bar$ containing the barcodes with $n$ infinite intervals. In particular, $Bar_0$ is the set of barcodes whose intervals are bounded.  
\begin{proposition} 
\label{proposition_connected_component_Bar }
The set of path connected components of $Bar$ is enumerable. More precisely, $\pi_0(Bar)= \{Bar_n\}_{n=0}^{+\infty}.$
\end{proposition}
\begin{proof}
Since $Bar=\bigsqcup_{n=0}^{+\infty}Bar_n$, we only need to prove that each $Bar_n$ is a maximal connected subset of $Bar$. 
First note that $Bar_n$ is path connected, as we can always move $n$ infinite intervals to $n$ other ones continuously, and similarly move the bounded off-diagonal intervals to the diagonal.
We now prove the maximality of $Bar_n$. Let $A\subseteq Bar\setminus Bar_n$ be non-empty. Any element in $A$ has infinite bottleneck distance to any element in $Bar_n$, since their numbers of infinite intervals are different.
Therefore, $A\cup Bar_n$ cannot be path-connected, and so $Bar_n$ is maximal.
\end{proof}
 We view the persistence image as a map $V:Bar_0\rightarrow \mathbb{R}^{n^2}$ for some \textit{discretization step}~$n\in \mathds{N}$:
 \begin{definition}
\normalfont
\label{definition_persistence_images}
Let $D\in Bar_0$. 
We fix a \textit{weighting function} $\omega:\mathbb{R} \rightarrow \mathbb{R}$ that is zero at the origin. For $(b,d)\in \mathbb{R}^2$, consider the Gaussian
\[ g_{b,d}:(x,y)\in \mathbb{R}^2\mapsto \frac{1}{2\pi \sigma^2} e^{-[(x-b)^2+(y-(d-b))^2]/2\sigma^2} \]
for some fixed variance $\sigma>0$. 
The \textit{persistence surface} associated to $D$ is the map
\[ \rho_D: (x,y)\in \mathbb{R}^2 \mapsto \sum_{(b,d)\in D} \omega(d-b) g_{b,d}(x,y). \]
Given a square~$B\subset \mathbb{R}^2$, we subdivide it into~$n^2$ regular squares $B_{k,l}$ for $1\leqslant k,l\leqslant n$. Then we define the \textit{persistence image} of~$D$ to be the histogram
$$V_{B,n}: D\in Bar_0 \mapsto \left( \int_{(x,y)\in B_{k,l}} \rho_D(x,y) dxdy \right)_{1\leqslant k,l\leqslant n}\in \mathbb{R}^{n^2}$$
\end{definition}
\begin{proposition}
\label{proposition_persistence_image_smooth}
If~$\omega$ is~$C^r$ over~$\mathbb{R}^2$ for some integer~$r\in \mathbb{N}$, then~$V_{B,n}$ is $r$-differentiable everywhere in~$Bar_0$.
\end{proposition}
\begin{proof}
%
The maps $(b,d)\in \mathbb{R}^2\mapsto \int_{(x,y)\in B_{k,l}}g_{b,d}(x,y)dxdy \in \mathbb{R}$ are $C^\infty$ for any fixed box~$B_{k,l}$. For any space of ordered barcodes $\mathbb{R}^{2m}\times \mathbb{R}^0$ and any $\obarcode=(b_1,d_1,...,b_m,d_m)\in \mathbb{R}^{2m}\times \mathbb{R}^0$,
\[ V_{B,n}(Q_{m,0}(\obarcode))=\left( \sum_{1\leqslant i \leqslant m} \omega(d_i-b_i) \int_{(x,y)\in B_{k,l}} g_{b_i,d_i}(x,y) dxdy\right)_{1\leqslant k,l\leqslant n}\in \mathbb{R}^{n^2}, \]
which is $C^r$ at every~$\obarcode\in \mathbb{R}^{2m}\times \mathbb{R}^0$.
\end{proof}
In~\cite{adams2017persistence}, the weighting function~$\omega$ is chosen to be the ramp function $\omega_t:\mathbb{R}\mapsto \mathbb{R}$ defined as
\begin{equation}
\label{eq:ramp_function}
 \omega_t(u) =   \left\{
   \begin{aligned}
     0 \text{ if } & u\leqslant 0  \\
    \frac{u}{t} \text{ if } & 0\leqslant u \leqslant t  \\
    1 \text{ if } & t \leqslant u
   \end{aligned}
   \right.
\end{equation}
for some parameter $t>0$. 
Thus, the ramp function is differentiable everywhere except at~$0$ and~$t$. This implies that the persistence image $V_{B,n}$ is nowhere differentiable, as every neighborhood of a barcode always contains some neighborhood of the diagonal~$\Delta$. Thanks to Proposition~\ref{proposition_persistence_image_smooth}, this issue can be resolved by taking any~$C^r$ approximation of the ramp function, which makes the persistence image $r$-differentiable over~$Bar_0$.

\subsection{Linear representations of barcodes}
\label{subsection_differentiability_linear_representation}
The analysis of persistence images in the previous section can be generalized to the following wide class of vectorizations:
\begin{definition}
Let $\phi: \R^2\rightarrow \R^k$, $\psi: \R\rightarrow \R^k$ and $\omega:\R\rightarrow \R$ be continuous maps such that $\omega(0)=0$. The associated {\em linear representation} is the map
\[V: D\in Bar \longmapsto \sum_{(b,d)\in D \text{ bounded}} \omega(d-b)\phi(b,d) + \sum_{(v,+\infty)\in D} \psi(v)\in \R^k. \]
\end{definition}
Properties of linear representations valued in Banach spaces such as continuity, lipschitzness and stochastic convergence are analyzed in~\cite{chazal2018density,divol2019understanding}. Many vectorizations in the litterature are linear representations, e.g. persistence images~\cite{adams2017persistence} and its variations~\cite{chen2015statistical, kusano2016persistence, reininghaus2015stable}, persistence silhouettes~\cite{chazal2014stochastic} and weighted Betti curves~\cite{umeda2017time}. 

When~$k=1$, a linear representation may be viewed as a loss function on persistence diagrams. The total persistence in Example~\ref{example_total_persistence}, and more generally the $q$-Wasserstein distance to the empty diagram, are such loss functions. In addition, the structure elements of~\cite[Definition~9]{hofer2019learning} form a wide class of parametrized linear losses and linear representations that can be optimised. 

In all these examples, the maps~$\phi,\psi$ and~$\omega$ are not necessarily smooth by design, see e.g. the ramp function in Eq.~\eqref{eq:ramp_function} for persistence images, but one can always replace them with smooth approximations. We then get $r$-differentiable maps on barcodes, as expressed in the following result. 

\begin{proposition}
\label{proposition_linear_representation_differentiable}
If the maps $\phi,\psi$ are~$C^r$ on generic subsets of~$\R^2$ containing the diagonal~$\Delta$,  and if~$\omega$ is~$C^r$ on a generic subset of~$\R$ containing the origin, then the associated linear representation~$V$ is generically $r$-differentiable. Whenever $\phi,\psi$ and~$\omega$ are in fact~$C^r$ everywhere, then~$V$ is $r$-differentiable everywhere. 
\end{proposition}

\begin{proof}
The subspace of barcodes whose intervals avoid the set of non-differentiability of~$\phi,\psi$ and~$\omega$ is clearly generic in~$Bar$. Let~$D$ be a barcode therein. For any space of ordered barcodes $\mathbb{R}^{2m}\times \mathbb{R}^n$ and pre-image $\obarcode=[(b_i,d_i)_{i=1}^m,(v_j)_{j=1}^n]\in \mathbb{R}^{2m}\times \mathbb{R}^n$ of~$D$, we have
\[ V(Q_{m,n}(\obarcode))=\sum_{1\leqslant i \leqslant m} \omega(d_i-b_i) \phi(b_i,d_i) + \sum_{1\leqslant j \leqslant n} \psi(v_j), \]
which is~$C^r$ in a neighborhood of~$\obarcode$.
\end{proof}
Let us consider an everywhere $r$-differentiable linear representation~$V$, and a barcode valued map~$B$ on a simplicial complex, which is (generically) differentiable (Theorem~\ref{theorem_PH_differentiable_global}). Using the chain rule~\ref{proposition_chain_rule_vectorization_barcode}, the composition~$V\circ B$ is then itself (generically) differentiable, hence amenable to gradient descent based optimisation.

\subsection{Semi-algebraic and subanalytic functions on barcodes}
\label{subsection_semialgebraic_vectorisation}
We consider another important class of examples arising from loss functions on barcodes that restrict to semi-algebraic maps on the spaces of ordered barcodes. The subanalytic and definable counterparts are analogously defined and the results of this section are valid in these situations as well. See also~\cite{carriere2020note} for a full treatment of semi-algebraic loss functions in persistence.
\begin{definition}
\label{definition_semi_algebraic_loss}
We say that a map~$V: Bar\rightarrow \R$ is {\em semi-algebraic} if all the precompositions~$V\circ Q_{m,n}: \R^{2m}\times \R^n\rightarrow \R$ are semi-algebraic. 
\end{definition}
A prototypical example of semi-algebraic loss on barcodes is the distance to a target barcode~$D_0$:
\[ d_{q}(D_0,.) :D\in Bar \mapsto d_{q}(D_0,D)\in \mathbb{R}\cup \{+\infty\}. \]
Here,~$d_q$ is the $q$-th Wasserstein distance on barcodes for any~$q\in \R_+^{*}$ as defined in Eq.~\eqref{eq_wasserstein}, and~$\db$ is the bottleneck distance. 
\begin{proposition}
\label{prop_distance_diagram_semi-_algebraic}
For any target barcode~$D_0$ and non-negative number~$q\in \R_+^{*}$, the map $d_{q}(D_0,.):Bar \rightarrow \R$ is semi-algebraic. 
\end{proposition}
\begin{proof}
We consider the case where~$q=\infty$, as the same line of arguments works for arbitrary Wasserstein metrics, and rewrite~$d_{q}(D_0,.)$ as~$d_{D_0}$ for simplicity. Let~$m,n\in \N$. We assume that~$n$ is the number of infinite intervals in~$D_0$, as otherwise the map~$d_{D_0}\circ Q_{m,n}: \R^{2m}\times \R^n\rightarrow \R$ takes infinite value everywhere. Then,~$d_{D_0}\circ Q_{m,n}$ can be expressed as a minimum of finitely many cost functions, $\min c(\gamma_{m,n})(.)$, each of which is defined in terms of a fixed partial matching~$\gamma_{m,n}$ of coordinates in~$\R^{2m}\times \R^n$ with interval endpoints of~$D_0$. As a point-wise maximum of finitely many absolute values, each cost function~$c(\gamma_{m,n})(.)$ is semi-algebraic, and so~$d_{D_0}\circ Q_{m,n}$ is semi-algebraic. 
\end{proof}

Semi-algebraic functions~$V$ on barcodes are particularly useful in the context of optimisation when pre-composed with a semi-algebraic parametrization of filter functions~$\parametrization: \paramspace \rightarrow \R^K$ on a fixed simplicial complex~$K$. Indeed, composition preserves semi-algebraicity, and so from Remark~\ref{remark_definable_lift_for_definable_parametrization} the loss function given by the composition
\begin{equation}
\label{eq_loss_semi_algebraic}
\mathcal{L}: \xymatrix{\paramspace{} \ar[rr]^-{\parametrization} && \R^K \ar[rr]^-{\Dgm_p} && Bar \ar[rr]^-{V} && \R}
\end{equation}
is a semi-algebraic map. Then, \cite[Corollary~5.9]{davis2020stochastic} guarantees that the well-known stochastic gradient descent (SGD) algorithm converges almost surely to critical points of~$\mathcal{L}$.\footnote{The loss~$\mathcal{L}$ must also be locally Lipschitz for this result to hold. By the Stability Theorem~\cite{cohen2007stability},~$\Dgm_p$ is Lipschitz continuous, hence this additional mild requirement is met whenever~$\parametrization$ and~$V$ are locally Lipschitz (for instance when~$V=d_q(D_0,.)$ is the distance to a fixed barcode).} 

This guarantee can be applied to various optimisation problems. When choosing the Rips parametrization~$\parametrization$ of point clouds as in Section~\ref{section_discrete_smoothness_subsection_examples_ex2}, minimizing the loss~$\mathcal{L}=d_q(D_0,.)\circ \Dgm_p\circ \parametrization$ amounts to solving the problem of point cloud inference originally proposed in~\cite{gameiro2016continuation}, see~\cite{1905.12200} for implementations. Besides, from Section~\ref{section_discrete_smoothness_subsection_examples_ex4}, for~$\parametrization$ the parametrization of all filter functions on a fixed simplicial complex and an adequate target barcode~$D_0$, the minimisation of~$\mathcal{L}$ yields an approach to function simplification. However, when~$\parametrization$ is not semi-algebraic, typically in the continuous setting developped in Section~\ref{section_continuous_smoothness_theorem}, and more generally for an arbitrary barcode valued map~$B:\paramspace\rightarrow Bar$, it is unclear how to perform full-fledged continuous gradient descent to minimize
\begin{equation}
\label{eq_general_loss_again}
\mathcal{L}: \xymatrix{\paramspace{} \ar[rr]^-{B} &&  Bar \ar[rr]^-{d_q(D_0,.)} && \R}.
\end{equation}
While implementing a solution to this problem is beyond the scope of this paper, it serves as a motivation for the next section where we show that the bottleneck distance to~$D_0$ is generically $\infty$-differentiable, as then the chain rule of Proposition~\ref{proposition_chain_rule_vectorization_barcode} enables the use of gradient descent. 
\subsection{The bottleneck distance to a diagram}
\label{subsection_differentiability_bottleneck}
%
For simplicity, we denote the bottleneck distance to a fixed barcode~$D_0$ by:
\[d_{D_0}: D\in Bar\longmapsto \db(D,D_0)\in \R.\]
For ease of exposition, we consider the special case where~$D_0=\Delta^{\infty}$ is the empty diagram (the diagonal $\Delta$ with infinite multiplicity). The analysis of the general case of an arbitrary fixed barcode $D_0$ is technically more involved and is deferred to Appendix~\ref{appendix_bottleneck}. 

Recall that $d_{\Delta^{\infty}}(D)=+\infty$ for any diagram $D\in Bar$ with infinite bars. Consequently, we consider the restriction of $d_{\Delta^{\infty}}$ to the subset~$Bar_0$ introduced in Section~\ref{subsection_persistence_images}. This restriction is valued in the real line: $d_{\Delta^{\infty}}:Bar_0 \rightarrow \mathbb{R}$.
Consider the set~$Bar_{\Delta}$ of barcodes which admit a unique point at maximal distance to the diagonal~$\Delta$:
\begin{equation}
    \label{equation_bar_delta_empty}
    Bar_{\Delta}:=\left\{D\in Bar_0\ \mid\ \# \argmax_{(b,d)\in D}\ \frac{|d-b|}{2} = 1 \right\}.
\end{equation}
For~$D\in Bar_{\Delta}$, we let $(\bar{b}_D,\bar{d}_D)\in D$ be the unique interval in the set~$\argmax_{(b,d)\in D}\ \frac{|d-b|}{2}$.
\begin{proposition}
$Bar_{\Delta}$ is generic in $Bar_0$. Moreover, given $D\in Bar_{\Delta}$, for $\e>0$ small enough, any $D'$ at bottleneck distance less than~$\e$ from~$D$ satisfies~$d_{\Delta^{\infty}}(D')=\frac{|\bar d_{D'}- \bar b_{D'}|}{2}$ and~$\|(\bar b_{D'}, \bar d_{D'})- (\bar b_D, \bar d_D)\|_\infty < \e$.
\label{proposition_bar_delta_generic}
\end{proposition}
\begin{proof}
  Given $D\in Bar_0$, consider the set $\argmax_{(b,d)\in D}\ \frac{|d-b|}{2}$. If this set is not a singleton, then we can move infinitesimally one of its elements away from the diagonal, so as to get a diagram in $Bar_{\Delta}$. Thus, $Bar_{\Delta}$ is dense in $Bar_0$. Let now $D\in Bar_{\Delta}$, and let $\delta$ be the second maximal distance to the diagonal:
  \[ \delta := \max_{(b,d) \in D\setminus \{(\bar b_D, \bar d_D)\}}\frac{|d-b|}{2} \]
  and $\alpha:=\frac{|\bar d_D- \bar b_D|}{2}- \delta > 0$. Take $\e\in \left(0, \frac{\alpha}{4}\right)$.
If $D'$ is at bottleneck distance less than $\e$ from
  $D$, all the points of $D'$ are within distance less than~$\e$ either from the diagonal or from an
  off-diagonal point of~$D$. As we have picked $\e <
  \frac{\alpha}{4}$, there is a unique off-diagonal point
  $(\bar b', \bar d')$ of $D'$ that is within distance less than $\e$ from
  $(\bar b_D, \bar d_D)$, and it must be the unique furthest point from~$\Delta$ in~$D'$. So indeed~$D'\in Bar_{\Delta}$ and $(\bar b_{D'}, \bar d_{D'})=(\bar b', \bar d')$. 
Therefore, $Bar_{\Delta}$ is open, which concludes the proof.
\end{proof}
%
Not surprisingly, $d_{\Delta^{\infty}}$ is smooth at every $D\in Bar_{\Delta}$, with partial derivatives related to the ones of the map $(\bar b_D, \bar d_D)\mapsto \frac{|\bar d_D- \bar b_D|}{2}$.
\begin{proposition}
\label{prop:dist_empty_diag_diff}
  For any $D\in Bar_{\Delta}$, 
\begin{itemize}
    \item[(i)] $d_{\Delta^{\infty}}$ is $\infty$-differentiable at $D$, and
    \item[(ii)] for any $m\in \mathds{N}$ and $\obarcode{}\in \mathbb{R}^{2m}\times \mathbb{R}^0$ such that $Q_{m,0}(\obarcode{})=D$, there are exactly two non-zero components in the gradient $\nabla_{\obarcode{}} (d_{\Delta^{\infty}} \circ Q_{m,0})$, one with value $\frac{1}{2}$ and the other with value $-\frac{1}{2}$.
\end{itemize}
\end{proposition}
\begin{proof}
Let $m\in \mathds{N}$ and $\obarcode{}\in \mathbb{R}^{2m}\times \mathbb{R}^0$ be such that $Q_{m,0}(\obarcode{})=D$. Without loss of generality, we can write $\obarcode{}=(\bar b_D, \bar d_D, b_2, d_2, ..., b_m, d_m)$ where~$(b_i, d_i)$ is distinct from $(\bar b_D, \bar d_D)$ for all $2 \leq i \leq m$. By Proposition \ref{proposition_continuity_quotient_map}, $Q_{m,0}$ is continuous. Therefore, by Proposition \ref{proposition_bar_delta_generic}, there is an open neighborhood $U$ of $\obarcode{}$, such that for any $\obarcode{}'=(\bar b_{D'}, \bar d_{D'}, b'_2, d'_2, ..., b'_m, d'_m)\in U$, $Q_{m,0}(\obarcode{}')$ is in $Bar_\Delta$ and $d_{\Delta^{\infty}}(Q_{m,0}(\obarcode{}')) = \frac{|\bar d_{D'}- \bar b_{D'}|}{2}>0$. Assertions~(i) and~(ii) follow.
\end{proof}

\bibliographystyle{alpha}
\bibliography{reference}

\appendix
\section{Local isometry of the barcode on a simplicial complex}
\label{appendix_proof_coercivity}
\begin{proof}[Proof of Proposition~\ref{proposition_local_coercivity}]
We have several persistence diagrams to compare, so we first simplify the problem as follows. Given two vectors $D=(D_0,...,D_d)\in Bar^{d+1}$ and $D'=(D'_0,...,D'_d)^{d+1}$ of $d+1$ barcodes, let $\Gamma(D,D')$ be the set of \textit{multi-matchings} between $D$ and $D'$, where a multi-matching is a bijection $\gamma: \bigsqcup_{p=0}^d D_i \rightarrow \bigsqcup_{p=0}^d D'_i $ such that $\gamma(D_p)=D'_p$ for all $0\leqslant p \leqslant d$. The notions of cost $c(\gamma)$ and optimality are the same as for ordinary matchings. Specifically, for an optimal $\gamma$ in $\Gamma(\Dgm(f), \Dgm(g))$, we have $c(\gamma)=\max_{0\leqslant p \leqslant d}\db(\Dgm_p(f),\Dgm_p(g))$.

We fix an ordering $\simplex_1<...<\simplex_{\#K}$ of the simplices of $K$, which yields an isomorphism $\phi:\R^K \to \R^{\#K}$. We denote by $f_i$ the $i$-th component of $f$ through this isomorphism, i.e $f_i=f(\simplex_i)$. Let us assume that $f,g \in \mathbb{R}^{K}$ are two filter functions in a common top dimensional stratum $S$. If we can prove~\eqref{equation_coercivity} in this case then it will hold for $f,g$ in the closure of $S$ by a continuity argument. Since $f$ and $g$ are both in $S$, they induce the same strict order on the simplices, and without loss of generality we can assume that $f_1<...<f_{\#K}$ and $g_1<...<g_{\#K}$. By Proposition \ref{proposition_global_lift_permutation}, we can write $\Dgm(f)=Q(P(S)f)$ and $\Dgm(g)=Q(P(S)g)$ for a fixed permutation matrix $P(S)$, which implies that:
\begin{align}
    \label{equation_correspondance_coercivity_proof}
    \forall 1\leqslant i\neq j \leqslant \#K, \forall 0\leqslant p \leqslant d, \, \, (f_i,f_j)\in \Dgm_p(f) \Leftrightarrow( g_i,g_j)\in \Dgm_p(g) \\
    \forall 1\leqslant k \leqslant \#K, \forall 0\leqslant p \leqslant d, \, \, (f_k,+\infty)\in \Dgm_p(f) \Leftrightarrow( g_k,+\infty)\in \Dgm_p(g) \nonumber 
\end{align}
Let $\gamma\in \Gamma(\Dgm(f), \Dgm(g))$ be optimal. Consider the case where $\gamma$ sends an off-diagonal point $(b,d)$ of $\Dgm(f)$ onto the diagonal $\Delta$. As $(b,d)$ is of the form $(f_i,f_j)$ (or $(f_i,+\infty)$), this implies that $c(\gamma)\geqslant \frac{|f_i-f_j|}{2} \geqslant d_0(f)$. In addition, $\Dgm(f)$ and $ \Dgm(g)$ have exactly the same number of bounded and unbounded intervals in each degree, which implies that there exists an off-diagonal interval $(b',d')$ of $\Dgm(g)$ which has pre-image in the diagonal $\Delta$. Again, $(b',d')$ must be of the form $(g_k,g_l)$ (or $(g_k,+\infty)$), so that $c(\gamma)\geqslant \frac{|g_k-g_l|}{2} \geqslant d_0(g)$. Therefore, $c(\gamma)\geqslant  \max(d_0(f),d_0(g))$ and we are done. 

We now treat the case where all off-diagonal intervals are sent to off-diagonal intervals by~$\gamma$. We denote by $O(f,g)\subset \Gamma(\Dgm(f), \Dgm(g))$ the set of multi-matchings that send off-diagonal intervals to off-diagonal intervals. By the decomposition $\Dgm(f)=Q(P(S)f)$ (resp. $\Dgm(g)=Q(P(S)g)$) and from the fact that no two values of $f$ (resp. of $g$) are equal, the bounded end-points of off-diagonal intervals of $\Dgm(f)$ (resp. $\Dgm(g)$) are in bijection with the set $\{f_1,...,f_{\#K}\}$ (resp. $\{g_1,...,g_{\#K}\}$). Therefore, any multi-matching $\nu \in O(f,g)$ induces a permutation $\pi(\nu)$ of $\{1,...,\#K\}$. Let us denote by $c(\pi):=\max_i(|f_i-g_{\pi(i)}|)$ the \textit{cost} of a permutation $\pi$ of  $\{1,...,\#K\}$. In this formulation, we have:
\begin{equation}
    \label{equation_cost_permutation_equals_bottleneck}
    \max_{0\leqslant p \leqslant d}\db(\Dgm_p(f),\Dgm_p(g))=c(\gamma)= \min_{\nu \in O(f,g)} c(\nu)= \min_{\nu \in O(f,g)} c(\pi(\nu))
\end{equation}
Consider the following relaxed optimization problem, in which the pairing of coordinates in~\eqref{equation_correspondance_coercivity_proof} is ignored: 
\begin{equation}
    \label{equation_relaxed_permutation_optimization}
    \min_{\pi \text{ permutation of } \{1,...,\#K\}} c(\pi)
\end{equation}

From the fact that $f_1<...<f_{\#K}$ and $g_1<...<g_{\#K}$, the monotonicity of the optimal transport map for the $\infty$-Wasserstein distance in $\R$ \citep{rabin2011wasserstein} guarantees that $\pi=\text{Id}$ is a minimizer\footnote{Indeed, for a permutation $\pi$, let $\text{inv}(\pi)$ denote its number of inversions. Let $\pi$ be a minimizer \eqref{equation_relaxed_permutation_optimization} with minimal $\text{inv}(\pi)$. Assuming by contradiction that $\text{inv}(\pi)>0$, there exist $i<j$ such that $\pi(i)>\pi(j)$. Let $\tau{}$ be the transposition that swaps $\pi(i)$ and $\pi(j)$. Since $f_1<...<f_{\#K}$ and $g_1<...<g_{\#K}$, a simple case analysis shows that $c(\tau\circ \pi)\leqslant c(\pi)$, thus raising a contradiction with the minimality of $\text{inv}(\pi)$. Therefore $\pi=\text{Id}$ is a minimizer of~\eqref{equation_relaxed_permutation_optimization}. } of~\eqref{equation_relaxed_permutation_optimization}. 
Therefore,
\[\max_{0\leqslant p \leqslant d}\db(\Dgm_p(f),\Dgm_p(g))= \min_{\nu \in O(f,g)} c(\nu)\geqslant \min_{\pi \text{ permutation }} c(\pi)= c(\text{Id})=\|f-g\|_{\infty}\]
and we are done.

We now address the second part of the proposition. Let $f\in \mathbb{R}^{K}$ be a filter function. By the stability Theorem~\ref{stability theorem}, showing that Eq.~\eqref{equation_first_isometry} holds amounts to showing that $\max_{0\leqslant p \leqslant d }\db(\Dgm_p(f),\Dgm_p(g)) \geqslant \|f-g\|_\infty$. We denote by $S$ the top dimensional stratum $S$ that contains $f$, and let $g \in \mathbb{R}^{K}$ be another filter such that $\|f-g\|_\infty \leqslant d_0(f)$. This implies that $g$ is also in the (closure of the) stratum $S$. We can then apply~\eqref{equation_coercivity}, and since by assumption $\|f-g\|_{\infty}\leqslant d_0(f)\leqslant \max(d_0(f),d_0(g))$, we have the desired result. 

Using similar arguments, we finally prove Eq.~\eqref{equation_second_isometry}.  Let $f\in \mathbb{R}^{K}$ be a filter function in some top dimensional stratum $S$, and $g,h\in \R^K$ be such that $\|f-g\|\leqslant \frac{d_0(f)}{3}$ and $\|f-h\|\leqslant \frac{d_0(f)}{3}$. By the stability Theorem~\ref{stability theorem}, showing that Eq. \eqref{equation_second_isometry} holds amounts to showing that $\max_{0\leqslant p \leqslant d }\db(\Dgm_p(g),\Dgm_p(h)) \geqslant \|g-h\|_\infty$. For every $i\neq j \in \{1,...,\#K\}$, 
\[|g_i-g_j|=|(f_i-f_j)- [(f_i-g_i)+(g_j-f_j)]|\geqslant |f_i-f_j| - 2\|f-g\|_\infty \geqslant 2d_0(f)-2\|f-g\|_\infty\geqslant \frac{4d_0(f)}{3},\]
so that $d_0(g)\geqslant \frac{2d_0(f)}{3}$. Similarly, $d_0(h)\geqslant \frac{2d_0(f)}{3}$. Meanwhile, 
\[\|g-h\|_\infty\leqslant \|f-g\|_\infty+\|f-h\|_\infty \leqslant \frac{2d_0(f)}{3}. \]
Therefore, $\|g-h\|_\infty \leqslant \max(d_0(g),d_0(h))$, and since both $g,h$ are in (the closure of) $S$, we conclude by using Eq.~\eqref{equation_coercivity}.
\end{proof}
\section{The bottleneck distance to a fixed diagram: the general case}
\label{appendix_bottleneck}

Throughout, we denote by $\Delta_\epsilon$ the set of elements in $\R^2$ that are at distance less than $\epsilon>0$ to the diagonal $\Delta$. We equip, for the purpose of this section only, the spaces of ordered barcodes with the supremum norm $\|.\|_\infty$ rather than the Euclidean norm. Note that (the proof of) Proposition~\ref{proposition_continuity_quotient_map} ensures that the quotient maps $Q_{m,n}$ are $1$-Lipchitz with respect to the metrics in place. We denote by $\mathcal{B}(.,*)$ the ball centered at $.$ with radius $*$ with respect to the supremum norm or bottleneck metric depending on the context.

In this section we generalize Proposition~\ref{prop:dist_empty_diag_diff}, namely  we show the generic differentiability of the bottleneck distance $d_{D_0}:Bar\rightarrow \R\cup \{+\infty\}$ to an arbitrary fixed diagram $D_0\in Bar$. 
\begin{proposition}
\label{prop:dist__diag_diff}
Let $D_0\in Bar$ and $n$ be the number of infinite bars in $D_0$. For generic $D\in Bar_n$, $d_{D_0}$ is $\infty$-differentiable at $D$. Moreover, for any $m\in \mathds{N}$ and $\obarcode{}\in \mathbb{R}^{2m}\times \mathbb{R}^n$ such that $Q_{m,n}(\obarcode{})=D$, exactly one of the following possibilities holds:
\begin{itemize}
    \item[(i)] either the gradient $\nabla_{\obarcode{}} (d_{D_0} \circ Q_{m,n})$ has exactly two non-zero components, one with value $\frac{1}{2}$ and the other with value $-\frac{1}{2}$; or
    \item[(ii)] the gradient $\nabla_{\obarcode{}} (d_{D_0} \circ Q_{m,n})$ has a unique non-zero component with value $1$ or $-1$.
\end{itemize}
\end{proposition}

Proposition~\ref{prop:dist__diag_diff} states the generic smoothness of $d_{D_0}$. We first observe that all the compositions $d_{D_0}\circ Q_{m,n}$ are smooth on a generic subset of $\R^{2m+n}$.

\begin{lemma}
\label{lemma_generic_diff_bottleneck_pre_comp_quotient}
For every $m\in \mathbb{N}$, the map \[d_{D_0}\circ Q_{m,n}:\R^{2m+n}\rightarrow \R\]
is generically smooth, with gradients that are either $0$ or as in $(i)$ or $(ii)$ of Proposition~\ref{prop:dist__diag_diff}.
\end{lemma}
\begin{proof}
Let $m\in \mathbb{N}$. Define an {\em ordered matching} $\tilde{\gamma}: \R^{2m+n}\rightarrow \R^{2m+n}$ to be an affine map whose first $m$ pairs of coordinate functions (resp. last  $n$ coordinate functions) are of the form  $\obarcode:=[(b_i,d_i)_{i=1}^m,(v_j)_{j=1}^n] \mapsto (b_i,d_i)-(b_{0,i},d_{0,i})$ where $(b_{0,i},d_{0,i})$ is either an off-diagonal point in $D_0$ or $(b_{0,i},d_{0,i})=(\frac{b_i+d_i}{2},\frac{b_i+d_i}{2})$ is the orthogonal projection of $(b_i,d_i)$ onto $\Delta$ (resp. are of the form $\obarcode \mapsto v_j-v_{0,j}$ for some infinite interval $(v_{0,j},+\infty)$ in $D_0$). We further require that the collection of intervals $(b_{0,i},d_{0,i})$ (resp. $(v_{0,j},+\infty)$) involved in this way are distinct elements in $D_0$. We denote by $D_0(\tilde{\gamma})$ the set of bounded off-diagonal intervals $(b_0,d_0)\in D_0$ that are not in the collection $\{b_{0,i},d_{0,i}\}_{i=1}^m$. 

Since the maximum of smooth functions over $\R^{2m+n}$ is smooth\footnote{This is the same argument as in the proof of Proposition~\ref{theorem_PH_differentiable_global}. Namely, the set where at least two of the smooth functions involved in the maximum are equal is closed, and therefore the boundary of this set has generic complement. On this complement the maximum of the smooth functions locally equals a unique smooth function.} on a generic subset of $\R^{2m+n}$, the map
\begin{equation}
\label{equation_ordered_cost}
   \tilde{c}(\tilde{\gamma}): \obarcode \in \R^{2m+n} \mapsto \max(\|\tilde{\gamma}(\obarcode)\|_\infty, \{\frac{|d_0-b_0|}{2}\}_{(b_0,d_0)\in D_0(\tilde{\gamma})}) \in \R 
\end{equation}
is itself $C^\infty$ on a generic subset of $\R^{2m+n}$, with gradients either equal to $0$ or as in $(i)$ or $(ii)$ of Proposition~\ref{prop:dist__diag_diff}. Let $\tilde{\Gamma}_m$ be the set of ordered matchings $\tilde{\gamma}: \R^{2m+n}\rightarrow \R^{2m+n}$, which is non-empty and finite. Then the map
\[\tilde{d}_{D_0,m}: \obarcode{} \in \R^{2m+n} \mapsto \min_{\tilde{\gamma}\in \tilde{\Gamma}_m} \tilde{c}(\tilde{\gamma})(\obarcode{}) \in \R\]
is $C^\infty$ on a generic subset of $\R^{2m+n}$, with gradients either equal to $0$ or as in $(i)$ or $(ii)$ of Proposition~\ref{prop:dist__diag_diff}. 

We will be done if we can show that the two maps $d_{D_0}\circ Q_{m,n}$ and $\tilde{d}_{D_0,m}$ are equal over $\R^{2m+n}$. Fix an ordered barcode $\obarcode{}\in \R^{2m+n}$ and let $D:=Q_{m,n}(\obarcode{})$. Let $\tilde{\gamma}:\R^{2m+n}\rightarrow \R^{2m+n}$ be an ordered matching. The components of $\tilde{\gamma}$ determine a matching $\gamma$ between $D$ and $D_0$, sending $(b_i,d_i)$ onto $(b_{0,i},d_{0,i})$ and $(v_j,+\infty)$ onto $(v_{0,j},+\infty)$. By definition of the cost of a matching~\ref{definition_matching} and Equation~\eqref{equation_ordered_cost}, we have $c(\gamma)=\tilde{c}(\tilde{\gamma})(\obarcode)$. This yields $\tilde{d}_{D_0,m}(\obarcode{})\geqslant d_{D_0}(D)=d_{D_0}\circ Q_{m,n}(\obarcode)$. Conversely, among the optimal matchings from $D$ to $D_0$, it is always possible to find one that sends off-diagonal points of $D$ (and $D_0$) on the diagonal only by orthogonal projection. This allows us to lift $\gamma$ at the level of $\obarcode$ and to define an ordered matching $\tilde{\gamma}$ such that $\tilde{c}(\tilde{\gamma})(\obarcode{})=c(\gamma)$. This yields $\tilde{d}_{D_0,m}(\obarcode{})\leqslant d_{D_0}(D)=d_{D_0}\circ Q_{m,n}(\obarcode)$ and therefore $d_{D_0}\circ Q_{m,n}=\tilde{d}_{D_0,m}$ on $\R^{2m+n}$.

\end{proof}
We cannot directly use Lemma~\ref{lemma_generic_diff_bottleneck_pre_comp_quotient} to prove Proposition~\ref{prop:dist__diag_diff}. As a matter of fact, by the definition of $\infty$-differentiability (Definition~\ref{definition_smooth_vectorization}),  Proposition~\ref{prop:dist__diag_diff} is asking that for generic $D\in Bar_n$ {\em all} the maps $d_{D_0}\circ Q_{m,n}$, for varying $m\in \mathbb{N}$, should be smooth at pre-images of $D$. However, Lemma~\ref{lemma_generic_diff_bottleneck_pre_comp_quotient} only guarantees that the maps $d_{D_0}\circ Q_{m,n}$, taken individually, are smooth over generic subsets of $\R^{2m+n}$, and it is not clear a priori how to glue at the level of barcodes these generic subsets lying in different spaces of ordered barcodes $\R^{2m+n}$. In order to leverage Lemma~\ref{lemma_generic_diff_bottleneck_pre_comp_quotient}, we devise intermediate results that infer the smoothness of the maps $d_{D_0}\circ Q_{m',n}$ from the knowledge of the smoothness of a well-chosen map $d_{D_0}\circ Q_{m,n}$. The high-level intuition of each of these intermediate steps is as follows:
\begin{enumerate}
    \item Infinitesimal perturbations of a given diagram $D$ can be understood as infinitesimal moves of the off-diagonal points of $D$, together with appearances of small intervals from the diagonal. In Lemma~\ref{lemma_sending_points_diagonal}, we devise a generic condition on $D$ ensuring that these new small off-diagonal intervals appearing when perturbing $D$ do not play any role in the bottleneck distance to $D_0$.
    \item Given a barcode $D$, we take a pre-image $\obarcode{}_m\in Q_{m,n}^{-1}(D)$ of $D$ which is minimal in the sense that its pairs of adjacent components are not trivial, i.e not of the form $(b,b)$. In other words, $\obarcode{}_m$ is an ordering of the endpoints of off-diagonal intervals appearing in $D$ without extra pairs $(b,b)$ lying on the diagonal. Up to an infinitesimal perturbation of $\obarcode{}_m$, Lemma~\ref{lemma_generic_diff_bottleneck_pre_comp_quotient} ensures that $d_{D_0}\circ Q_{m,n}$ is smooth in an open neighborhood of $\obarcode_{m}$. It is easy to observe that for any other pre-image $\obarcode{}_{m'}$ of $D$, the components of the ordered barcode $\obarcode{}_{m'}$ only differ with those of $\obarcode_{m}$ by the addition of trivial pairs of the form $(b,b)$. According to the previous item, those trivial pairs do not play any role when computing the bottleneck distance to $D_0$. Therefore, since $d_{D_0}\circ Q_{m,n}$ is smooth in a neighborhood of $\obarcode{}_m$, the map $d_{D_0}\circ Q_{m',n}$ is itself smooth in an open neighborhood of $\obarcode{}_{m'}$. We make these intuitions rigorous in Lemma~\ref{lemma_minimality_implies_compositions_smooth}.
    \item The previous arguments allow us to construct open balls $\mathcal{B}(\obarcode_{m'},\epsilon)$ of the same radius $\epsilon>0$ around all pre-images $\obarcode_{m'}\in \R^{2m'+n}$ of a generic diagram $D\in Bar$ over which all maps $d_{D_0}\circ Q_{m',n}$ are smooth. To conclude that $d_{D_0}$ itself is $\infty$-differentiable in a neighborhood of $D$, we show in Lemma~\ref{lemma_convering_bottleneck_ball_Euclidean_ball} that the $\epsilon$-bottleneck ball around $D$ is covered by the union of the images of the balls $\mathcal{B}(\obarcode_{m'},\epsilon)$.
\end{enumerate}
Let $\hat{Bar}$ be the set of barcodes $D\in Bar_n$ such that no intervals of $D_0$ is at distance $\db(D,D_0)$ to its diagonal projection. It is easy to check that $\hat{Bar}$ is generic in $Bar_n$. When perturbing a given barcode $D$ infinitesimally, an arbitrary number of new off-diagonal points may appear from the diagonal. We show that, for $D\in \hat{Bar}$, these new off-diagonal intervals can be disregarded when computing the bottleneck distance to $D_0$. 

\begin{lemma}
\label{lemma_sending_points_diagonal}
Let $D\in \hat{Bar}$. There exists $\epsilon>0$ such that for any barcode $D'$ which is $\epsilon$-close to $D$ we have that $\db(D',D_0)>\epsilon$, and there exists an optimal matching from $D'$ to $D_0$ sending $D'\cap \Delta_\epsilon$ (i.e. those points of $D'$ that are $\epsilon$-close to the diagonal), onto the diagonal $\Delta$.
\end{lemma}
 
\begin{proof}
Let $D\in \hat{Bar}$. Denote by $\alpha$ the minimal gap $|\frac{|d_0-b_0|}{2}- \db(D,D_0)|$ between the distance of off-diagonal intervals $(b_0,d_0)$ of $D_0$ from their diagonal projections and $\db(D,D_0)$. Since $D\in \hat{Bar}$, $\alpha$ is strictly positive.


We have $\db(D,D_0)>0$ as otherwise $\db(D,D_0)=0$ would imply that $D\notin \hat{Bar}$ (as the distance from a diagonal element of $D_0$ to its diagonal projection would also be $0$), so we can pick $\epsilon>0$ such that %
\[\epsilon < \min(\frac{\db(D,D_0)}{2},\frac{\alpha}{2}).\]
We now prove that the conclusion of the Lemma holds in the bottleneck ball $\mathcal{B}(D,\epsilon)$. Let $D'\in \mathcal{B}(D,\epsilon)$. Since $\epsilon < \frac{\db(D,D_0)}{2}$, we have $\db(D',D_0)>\epsilon$.
We assume, seeking contradiction, that there is no optimal matching from $D'$ to $D_0$ that sends all points of $D'\cap \Delta_\epsilon$ onto $\Delta$.

We restrict our attention to the set $\Gamma^*(D',D_0)$ of optimal matchings from $D'$ to $D_0$ that are allowed to send off-diagonal points of $D'$ and $D_0$ to the diagonal only by orthogonal projections. This set is finite and non-empty. We define the {\em $\Delta$-degree} of a matching $\gamma\in \Gamma^*(D',D_0)$ to be the number of off-diagonal points of $D'$ and $D_0$ that are sent to their diagonal projections, and take $\gamma$ with maximal $\Delta$-degree. By assumption, there exists an off-diagonal point $(b',d')\in D'\cap \Delta_\epsilon$ sent to some off-diagonal point $(b_0,d_0)\in D_0$. Recall that $|\frac{|d_0-b_0|}{2}-\db(D,D_0)|\geqslant \alpha$. We divide the analysis into two cases: either $\frac{|d_0-b_0|}{2}\geqslant \db(D,D_0)+ \alpha$, or $\frac{|d_0-b_0|}{2}\leqslant \db(D,D_0)- \alpha$.

In the case where $\frac{|d_0-b_0|}{2}\geqslant \db(D,D_0)+ \alpha$, we have:
\begin{align*}
\|(b_0,d_0)-(b',d')\|_\infty &= \|[(b_0,d_0)-(\frac{b'+d'}{2},\frac{b'+d'}{2})]-[(b',d')-(\frac{b'+d'}{2},\frac{b'+d'}{2})]\|_\infty \\  &\geqslant \|(b_0,d_0)-(\frac{b'+d'}{2},\frac{b'+d'}{2})\|_\infty - \|(b',d')-(\frac{b'+d'}{2},\frac{b'+d'}{2})\|_\infty\\ 
&\geqslant  \frac{|d_0-b_0|}{2}- \frac{|d'-b'|}{2}\\
&\geqslant  \db(D,D_0)+ \alpha - \epsilon \\
& \geqslant \db(D',D_0) + \alpha - 2\epsilon\\
&> \db(D',D_0)
\end{align*}
where the first inequality holds by the triangle inequality, the second from the fact that a minimizer of the distance from $(b_0,d_0)$ to the diagonal is the orthogonal projection of $(b_0,d_0)$ onto $\Delta$, the third by assumption on $(b_0,d_0)$ and $(b',d')$, the fourth by the triangle inequality and the last one by $\epsilon<\frac{\alpha}{2}$. This yields a contradiction as $\gamma$ is optimal and its cost may not exceed $\db(D',D_0)$.

Consider now the case where $\frac{|d_0-b_0|}{2}\leqslant \db(D,D_0)- \alpha$. On the one hand, by the triangle inequality and by the fact that $\epsilon<\alpha$:
\[\frac{|d_0-b_0|}{2}\leqslant \db(D,D_0)- \alpha \leqslant \db(D',D_0) +\epsilon - \alpha <  \db(D',D_0)\]
On the other hand, since $\epsilon < \frac{\db(D,D_0)}{2}$, 
\[\frac{|d'-b'|}{2}\leqslant \epsilon < \frac{\db(D,D_0)}{2}\leqslant \frac{\db(D',D_0) +\epsilon}{2} <  \db(D',D_0),\]
where the last inequality comes from the fact that $\epsilon < \db(D',D_0)$ because $\epsilon <  \frac{\db(D',D_0)+ \epsilon}{2}$.
To sum up, both quantities $\frac{|d_0-b_0|}{2}$ and $\frac{|d'-b'|}{2}$ are upper-bounded by $\db(D',D_0)$. Modifying $\gamma$ by sending $(b_0,d_0)$ and $(b',d')$ to their diagonal projections, we obtain a matching in $\Gamma^*(D',D_0)$ with $\Delta$-degree strictly higher than that of $\gamma$, which contradicts the maximality of the $\Delta$-degree of $\gamma$.
\end{proof}

We say that an ordered barcode $\obarcode_{m}=[(b_i,d_i)_{i=1}^m,(v_j)_{j=1}^n] \in\R^{2m+n}$ is {\em minimal} if $b_i\neq d_i$ for $1\leqslant i \leqslant m$. This terminology is justified by the fact that the image $D:=Q_{m,n}(\obarcode_{m})\in Bar_n$ contains exactly $m$ bounded-off diagonal intervals and $n$ unbounded ones, and therefore any other pre-image $\obarcode{}_{m'}\in \R^{2m'+n}$ of $D$ must lie in a space of ordered barcodes of dimension at least $2m+n$ (i.e. $m'\geqslant m$). We show that under suitable assumptions, the differentiability of all the maps $d_{D_0}\circ Q_{m',n}$ at pre-images $\obarcode{}_{m'}$ of $D$ can be inferred from the differentiability of $d_{D_0}\circ Q_{m,n}$ at the minimal pre-image $\obarcode{}_m$.

\begin{lemma}
\label{lemma_minimality_implies_compositions_smooth}
For every $m\in \mathbb{N}$, the set of minimal ordered barcodes in $\R^{2m+n}$ is open. Moreover, given a minimal $\obarcode{}_m\in \R^{2m+n}$ with $D:=Q_{m,n}(\obarcode{}_m)\in \hat{Bar}$, if $d_{D_0}\circ Q_{m,n}$ is $C^\infty$ in an open neighborhood of $\obarcode{}_m$, then there is an $\epsilon>0$ such that for all other pre-images $\obarcode{}_{m'}$ of $D$, the map $d_{D_0}\circ Q_{m',n}$ is $C^\infty$ in $\mathcal{B}(\obarcode_{m'},\epsilon)$, with gradients as in $(i)$ or $(ii)$ of Proposition~\ref{prop:dist__diag_diff}.
\end{lemma}
\begin{proof}
Is is clear that the set of minimal ordered barcodes in $\R^{2m+n}$ is open. We address the second part of the Lemma. Let $\obarcode{}_m\in \R^{2m+n}$ be a minimal ordered barcode such that $D:=Q_{m,n}(\obarcode{}_m)\in \hat{Bar}$, and assume there is an open neighborhood $U$ of $\obarcode{}_m$ within which $d_{D_0}\circ Q_{m,n}$ is $C^\infty$. By continuity of the quotient map and from the fact that $\hat{Bar}$ is open, we can assume without loss of generality that $Q_{m,n}(U)$ is contained in $\hat{Bar}$.

For any other pre-image $\obarcode_{m'}\in \R^{2m'+n}$ of $D$, i.e. an ordered barcode such that $Q_{m',n}(\obarcode_{m'})=D=Q_{m,n}(\obarcode_{m})$, the first $m'$ adjacent pairs of components of $\obarcode_{m'}$ must describe  in an arbitrary order the $m$ bounded off-diagonal points of $D$ together with $m'-m$ trivial pairs of the form $(b,b)$. The last $n$ components of $\obarcode_{m'}$ must be in correspondance with the left endpoints of infinite intervals in $D$. In other words, the first $2m'$ components of $\obarcode{}_{m'}$ consist of a re-ordering of the first $2m$ components of $\obarcode{}_{m}$, together with $m'-m$ trivial pairs of the form $(b,b)$. The last $n$ components of $\obarcode{}_{m'}$ consist of a re-ordering of those of $\obarcode{}_{m}$. 

To every pre-image $\obarcode_{m'}$ of $D$ as above, we associate the linear projection $L_{m',m}:\R^{2m'+n}\rightarrow \R^{2m+n}$ that sends $\obarcode{}_{m'}$ to $\obarcode{}_{m}$ by re-arranging the $m$ non trivial pairs of components and the $n$ last components, and forgetting the $m'-m$ trivial pairs. Since $D\in \hat{Bar}$, Lemma~\ref{lemma_sending_points_diagonal} provides an $\epsilon>0$ such that for any $D'\in \mathcal{B}(D,\epsilon)$, the points of $D'$ that are in $\Delta_\epsilon$ may be sent onto the diagonal when computing the bottleneck distance from $D'$ to $D_0$, and furthermore they can be disregarded when computing $\db(D',D_0)$. Therefore, using that the quotient map $Q_{m',n}$ is $1$-Lipschitz, we know that for any pre-image $\obarcode{}_{m'}$ of $D$ and $\obarcode{}'_{m'}\in \mathcal{B}(\obarcode{}_{m'}, \epsilon)$, the $m'-m$ pairs of components $\obarcode{}'_{m'}$ with persistence less than $\epsilon$ can be disregarded when computing $d_{D_0}\circ Q_{m',n}(\obarcode{}'_{m'})$. Formally, for every $m'\in \mathbb{N}$,
\begin{equation}
\label{equation_equality_minimal_db_all_preimages}
    \forall \obarcode_{m'}\in Q_{m',n}^{-1}(D), \, \forall  \obarcode{}'_{m'} \in \mathcal{B}(\obarcode{}_{m'}, \epsilon), \, \, d_{D_0}\circ Q_{m',n}(\obarcode{}'_{m'})=d_{D_0}\circ Q_{m,n} \circ L_{m',m}(\obarcode{}'_{m'}).
\end{equation}
Note that the maps $L_{m',m}$ are $1$-Lipschitz. Therefore, we can reduce $\epsilon$ in order to ensure that $L_{m',m}(\mathcal{B}(\obarcode{}_{m'},\epsilon))\subset U$ for every pre-image $\obarcode{}_{m'}$ of $D$. Applying the chain rule on $d_{D_0}\circ Q_{m,n}$ and $L_{m',m}$ ---which is an affine map hence $C^\infty$--- in Equation~\eqref{equation_equality_minimal_db_all_preimages}, we obtain that all the maps $d_{D_0}\circ Q_{m',n}$ are $C^\infty$ in $\mathcal{B}(\obarcode{}_{m'}, \epsilon)$. Also by the chain rule, by definition of $L_{m',m}$, the components of the gradients of the maps $d_{D_0}\circ Q_{m',n}$ are a re-ordering of the components of the gradient of $d_{D_0}\circ Q_{m,n}$. By Lemma~\ref{lemma_generic_diff_bottleneck_pre_comp_quotient}, the gradient of the latter is either $0$ or as in $(i)$ or $(ii)$ of Proposition~\ref{prop:dist__diag_diff}. However, the gradient of $d_{D_0}\circ Q_{m,n}$ being $0$ at some elements $\obarcode{}_m'\in U$ would mean that the bottleneck distance $\db(Q_{m',n}(\obarcode{}_m'),D_0)$ equals the distance of some off-diagonal interval $(b_0,d_0)$ to its diagonal projection, which is impossible since $Q_{m',n}(\obarcode{}_m')\in \hat{Bar}$.
\end{proof}
By means of Lemma~\ref{lemma_minimality_implies_compositions_smooth}, we can deduce at once the differentiability of all the maps $d_{D_0}\circ Q_{m',n}$ over balls of the same radius. We need a last result that connects these balls to an actual open neighborhood of $D$ in $Bar_n$.

\begin{lemma}
\label{lemma_convering_bottleneck_ball_Euclidean_ball}
For any $D\in Bar_n$, there exists an $\epsilon>0$ such that for every $m'\in \mathbb{N}$,
\[Q_{m',n}^{-1}(\mathcal{B}(D,\epsilon))\subseteq \bigcup_{\obarcode_{m'} \in R^{2m'+n}, Q_{m',n}(\obarcode{}_{m'})=D} \mathcal{B}(\obarcode{}_{m'},\epsilon).\]
\end{lemma}
\begin{proof}
Let $D\in Bar_n$, and $\eta>0$ be less than all the pairwise distances between geometrically distinct off-diagonal points in $D$, and less than all the distances from off-diagonal points in $D$ to the diagonal. We take $\epsilon>0$ such that $\epsilon<\frac{\eta}{2}$. Let $D'\in \mathcal{B}(D,\epsilon)$. Then, for every off-diagonal point $(b,d)$ of $D$, the number of (off-diagonal) points of $D'$ lying in $\mathcal{B}((b,d),\epsilon)$ equals the multiplicity of $(b,d)$ in $D$. Let us say that these points in $D'$ are of {\em type} {\bf (a)}. The points of $D'$ that are not in the balls $\mathcal{B}((b,d),\epsilon)$, for $(b,d)$ ranging over off-diagonal intervals of $D$, must be $\epsilon$-close to the diagonal, and we say that these points are of type {\bf (b)}. Note that we can accordingly characterize the components of a pre-image $\obarcode{}'_{m'}\in Q_{m',n}^{-1}(D')$: the pairs of components in $\obarcode{}'_{m'}$ must either be trivial (i.e of the form $(b,b)$), or equal to some off-diagonal point of type {\bf (a)} or {\bf (b)}. All off-diagonal points of $D'$, of type {\bf (a)} or {\bf (b)}, counted with multiplicity, must appear as a pair in $\obarcode{}'_{m'}$.

Given such a pre-image $\obarcode{}'_{m'}\in Q_{m',n}^{-1}(\mathcal{B}(D,\epsilon))$ of $D'$, we construct another ordered barcode $\obarcode{}_{m'}\in \R^{2m'+n}$ by modifying the components of $\obarcode{}'_{m'}$ at cost less than $\epsilon$ (i.e such that $\|\obarcode{}'_{m'}-\obarcode{}_{m'}\|_\infty <\epsilon$) as follows:
\begin{itemize}
    \item The last $n$ components of $\obarcode{}'_{m'}$ parametrize the left endpoints of infinite intervals in $D'$. We change them at cost less than $\epsilon$ into the left endpoints of infinite intervals in $D$.
    \item If a pair $(b',d')$ among the first $m'$ pairs of components of $\obarcode{}'_{m'}$ is of type {\bf (a)}, it is $\epsilon$-close to a unique off-diagonal point $(b,d)$ of $D$. We change it into $(b,d)$. 
    \item If a pair $(b',d')$ among the first $m'$ pairs of components of $\obarcode{}'_{m'}$ is of type {\bf (b)}, it is $\epsilon$-close to the diagonal. We transform it into $(\frac{b'+d'}{2},\frac{b'+d'}{2})$.
    \item The remaining pairs in the first~$m'$ pairs of components of $\obarcode{}'_{m'}$ must be trivial, and we leave them unchanged.
\end{itemize}
In this way, we have constructed an ordered barcode $\obarcode{}_{m'}$ such that $\obarcode'_{m'}\in \mathcal{B}(\obarcode_{m'},\epsilon)$ and also, by construction, $\obarcode_{m'}$ is a pre-image of $D$, i.e $Q_{m',n}(\obarcode_{m'})=D$.
\end{proof}
We are now ready to prove Proposition~\ref{prop:dist__diag_diff}.
\begin{proof}[Proof of Proposition~\ref{prop:dist__diag_diff}]
Consider the set of barcodes $D\in Bar_n$ that admit an open neighborhood within which $d_{D_0}$ is $\infty$-differentiable. By definition, this set is open in $Bar_n$, and we are left to show that it is also dense. Given an arbitrary $D\in Bar^n$, we will perform a series of infinitesimal perturbations of $D$, so that there exists a (small) open neighborhood $U$ of $D$ over which $d_{D_0}$ is $\infty$-differentiable.

Since $\hat{Bar}$ is generic in $Bar_n$, up to an infinitesimal perturbation, we can assume that $D$ lies in $\hat{Bar}$. Let $\obarcode{}_m\in \R^{2m+n}$ be a minimal pre-image of $D$. By Lemma~\ref{lemma_minimality_implies_compositions_smooth}, the set of minimal ordered barcodes in $\R^{2m+n}$ is open. Moreover, $d_{D_0}\circ Q_{m,n}$ is smooth on a generic subset of~$\R^{2m+n}$ by Lemma~\ref{lemma_generic_diff_bottleneck_pre_comp_quotient}. Therefore, up to an infinitesimal perturbation of $\obarcode{}_m$ (which results in an infinitesimal perturbation of~$D$ by continuity of~$Q_{m,n}$), we can further assume that~$d_{D_0}\circ Q_{m,n}$ is smooth on a ball $\mathcal{B}(\obarcode{}_m,\epsilon)$ for some~$\epsilon>0$, while $\obarcode{}_m$ remains minimal and~$D$ stays in~$\hat{Bar}$. 

Reducing $\epsilon$ if necessary, by Lemma~\ref{lemma_minimality_implies_compositions_smooth} all the maps $d_{D_0}\circ Q_{m',n}$ are smooth over $\mathcal{B}(\obarcode{}_{m'},\epsilon)$, with gradients as in $(i)$ or $(ii)$ of Proposition~\ref{prop:dist__diag_diff}, where $\obarcode{}_{m'}$ ranges over the pre-images of $D$. Reducing $\epsilon$ further if necessary, we conclude that $d_{D_0}$ is $\infty$-differentiable over $\mathcal{B}(D,\epsilon)$ by Lemma~\ref{lemma_convering_bottleneck_ball_Euclidean_ball}.
\end{proof}
\end{document}